\documentclass[10pt]{article}
\usepackage{amsthm}
\usepackage{amsmath}
\usepackage{amsfonts}
\usepackage{amssymb}
\usepackage{empheq}
\usepackage{xcolor,dsfont}
\usepackage{url}

\usepackage[margin=25mm]{geometry}

\definecolor{gray}{HTML}{d2d2d2}

\newtheorem{thm}{Theorem}

\newtheorem{ex}[thm]{Example}
\newtheorem{lemma}[thm]{Lemma}
\newtheorem{prop}[thm]{Proposition}
\newtheorem{cor}[thm]{Corollary}
\newtheorem{rem}[thm]{Remark}

\newcommand{\R}{\mathds{R}}
\newcommand{\de}{\partial}

\newcommand{\pk}{para-K\"ahler }
\newcommand{\Ric}{\operatorname{Ric}}
\newcommand{\grad}{\operatorname{grad}}
\newcommand{\tr}{\operatorname{tr}}
\newcommand{\Id}{\mathrm{Id}}

\newcommand{\E}{\mathrm{E}}

\newcommand{\T}{\mathcal{T}}

\newcommand{\K }{K\"{a}hler }

\newcommand{\C}{\mathds{C}}

\title{
Projectively equivalent \pk and para-K\"ahler-Einstein  metrics with non-parallel Benenti tensors and their normal forms in dimension four
} \author{Gianni Manno, Filippo Salis }

\begin{document}

\maketitle

\begin{abstract}
The study of projectively equivalent metrics, i.e., metrics sharing the same unparametrized geodesics, is a classical and well-established area of investigation. In the \K context, such branch of research goes by the name of \emph{c-projective geometry}: it mainly studies \emph{c-projectively equivalent metrics}, i.e., \K metrics sharing the same 
curves that are the complex analogue of the geodesics, called \emph{$J$-planar}, where $J$ is the complex structure. In this paper, we develop the theory of the projective equivalence in the \pk context by studying \pk metrics sharing the same \emph{$\T$-planar curves}, where $\T$ is the para-complex structure: we call such metrics \emph{pc-projectively equivalent}. After establishing some general results in arbitrary dimension, we focus on the $4$-dimensional case. One of the main achievement is a local description of $4$-dimensional pc-projectively (but not affinely) equivalent metrics and, as an application of this result, we characterize which of them are of Einstein type.
\end{abstract}



%

\section{Introduction}
 
To start with, a justification of the name \emph{pc-projective geometry} is due: it stands for \emph{para-complex projective geometry}, whose basic definitions are given in Section \ref{intro} below. It takes its cue from the name \emph{c-projective geometry}, that stands for \emph{complex projective geometry}, a well-known branch of research area \cite{bmmr,bmr,cemn,cmr,cg,dgw,fg,msv,mr} whose main subject are K\"ahler metrics sharing the same \emph{J-planar curves}, where $J$ is the complex structure, that are the complex analogue of the geodesics of (pseudo-)Riemannian metrics. 

\subsection{Basic definitions of pc-projective geometry}\label{intro}

An \emph{almost para-complex manifold} is a $2n$-dimensional manifold $M$ provided with a field of endomorphisms $\T$ such that $\T^2=\Id$, having eigenvalues $1$ and $-1$, whose associated eigendistributions 
$$
T^{\pm}:=\ker(\T\mp\Id)
$$
are $n$-dimensional. An almost para-complex manifold whose distributions $T^{\pm}$ are integrable is called a \emph{para-complex manifold}: this is the same to require the vanishing of the Nijenhuis tensor of $\mathcal{T}$.
%
A \emph{\pk manifold} is a para-complex manifold $M$ endowed with a \emph{para-Hermitian} metric $g$, that is, a pseudo-Riemannian metric satisfying
$$
g\bigl(\T(\cdot)\,  ,\T (\cdot)\bigr)=- g(\cdot\,,\,\cdot)\, ,
$$
whose associated fundamental $2$-form
\begin{equation}\label{eq:omega}
\omega(\cdot\, ,\, \cdot)=g\bigl(\mathcal{T}(\cdot)\, ,\,\cdot\,\bigr)
\end{equation}
is symplectic. 
From now on, we denote such a \pk manifold by the triple $(M,g,\mathcal{T})$.
 It turns out that $g$ possesses a neutral signature, i.e., $(n,n)$ and, moreover, distributions $T^{\pm}$ are $g$-isotropic, namely   $g$ satisfies $g(T^+,T^+)=g(T^-,T^-)=0$. Taking this in account, a natural choice of local coordinates on a $n$-dimensional \pk manifold are the \emph{null coordinates}: since the eigendistributions $T^+$ and $T^-$ of $\T$ are integrable, one can locally choose {adapted coordinates} $(x^1,\dots,x^{2n})$ to $T^+$ and $T^-$, namely coordinates such that $\partial_{x^i}$ and $\partial_{x^{n+i}}$ lie, respectively, in $T^+$ and $T^-$ for  $i=1,\dots, n$.
In these coordinates, the matrix of the components of the para-complex structure $\T$ takes the simple block form
\begin{equation}\label{eq:matrix.T.adapted}
\T=\left(
\begin{array}{cc}
\Id_n & 0_n
\\
0_n & -\Id_n
\end{array}
\right)
\end{equation}
whereas the \pk metric $g$ an off-diagonal block form.
%

\smallskip\noindent
A regular curve $\gamma : I \subseteq \mathbb{R} \to M$ is called {\it $\T$-planar} if there exist smooth functions $\alpha, \beta\in C^\infty(I)$ such that
  \begin{equation}\label{tauplanar}
  \nabla^g_{\dot\gamma} \dot\gamma = \alpha \,\dot\gamma + \beta\,\T(\dot\gamma),
  \end{equation}
where $\nabla^g$ here, and in what follows, denotes the Levi-Civita connection of $g$. We underline that the definition of $\mathcal{T}$-planar curves can in fact be formulated without any reference to a metric. Indeed, 
given an almost para-complex manifold $(M,\mathcal{T})$, one may consider a \emph{para-complex connection}, that is, a torsion-free linear connection $\nabla$ on $TM$ against which $\mathcal{T}$ is parallel, i.e., such that
$\nabla \mathcal{T} = 0$: for any such connection, the notion of $\mathcal{T}$-planarity of a curve $\gamma$ can be defined exactly as above\footnote{The notion of $\T$-planar curve closely follows that of $J$-planar curve in the case of K\"ahler manifolds, where $J$ is the complex structure, see \cite{bmmr} for more details and motivations.  Both the concepts of $\T$- and $J$- planar curve are special cases of a more general definition, i.e., that of \emph{$F$-planar curve}, where $F$ is a generic field of endomorphisms. One can consult \cite{mikes_siny,mikes1,mikes_book} for the basics on this topic.}.  

A vector field $X$ on $M$ is called  \emph{(infinitesimal) pc-projective symmetry} if its local flow sends $\T$-planar curves to $\T$-planar curves. 

\smallskip
One can then consider families of para-complex connections sharing the same $\mathcal{T}$-planar curves. 
We shall refer to such a family as a \emph{pc-projective class}, and say that two para-complex connections $\nabla$ and $\widehat\nabla$ are \emph{pc-projectively equivalent} if they belong to the same pc-projective class, denoted by $[\nabla]$. Accordingly, two \pk metrics $g$ and $\hat{g}$  are \emph{pc-projectively equivalent} if their Levi-Civita connections are so. We denote by $[g]$ the set of pc-projectively equivalent metrics to $g$. Two \pk metrics are \emph{affinely equivalent} if they share the same Levi-Civita connection. Obviously, two affinely equivalent \pk metrics are also pc-projectively equivalent.

One of the most important properties of pc-projectively equivalent metrics is that we can construct a $2$-parametric family of such metrics if we know two of them which are non-proportional. More precisely, if $g$ and $\hat{g}$ belong to the same pc-projective class, then the metrics of the following family
\begin{equation}\label{eq:lin.comb.metrics}
\frac{\left(\alpha\det(g)^{\frac{1}{2(n+1)}}g^{-1} + \beta\det(\hat{g})^{\frac{1}{2(n+1)}}\hat{g}^{-1} \right)^{-1}}{\det\left(\alpha\det(g)^{\frac{1}{2(n+1)}}g^{-1} +  \beta\det(\hat{g})^{\frac{1}{2(n+1)}}\hat{g}^{-1} \right)^{\frac12}}\,,\quad \alpha, \beta\in\R\,,
\end{equation}
also belong to the same pc-projective class. Of course formula \eqref{eq:lin.comb.metrics} can be iterated if one knows an arbitrary number of pc-projectively equivalent \pk metrics. This results is not trivial and will be proved and discussed in Section \ref{sec:pc-invariant}, in particular, Corollary \ref{cor:lin.comb.metrics}.


\subsection{Motivations}\label{sec:motivations}

Para-K\"ahler manifolds can be equivalently described in terms of \emph{bi-Lagrangian structures}.  
Indeed, if  $(M, g, \T)$ is a \pk manifold, then the eigendistributions $T^{\pm}$ of $\mathcal{T}$
form two complementary integrable  Lagrangian distributions with respect to $\omega$ given by \eqref{eq:omega}.  
Conversely, given a symplectic manifold $(M, \omega)$ endowed with two complementary integrable Lagrangian distributions $T^{\pm}$, one can define an endomorphism $\mathcal{T}$ by setting $\mathcal{T}|_{T^{\pm}} = \pm \Id$.  
Then, $\mathcal{T}$ is a para-complex structure and $(M, g, \mathcal{T})$ becomes a \pk manifold, where $g=\omega\left(\T(\cdot),\,\cdot\,\right)$. 
These are important aspects for which para-K\"ahler
geometry and, more in general, para-complex geometry is an area of research rather active, as evidenced
by lots of papers in this field. In this regard, the survey \cite{rockym} shows a large spectrum of applications of this
geometry, mainly focused on actions of Lie groups on the aforementioned manifolds, together with a historical
introduction. One can also consult \cite{al,amt,amt2} for further study and their bibliography for further readings. For
different applications, one can see \cite{hl} and references therein, where, among other things, it was shown a
direct relation between Lagrangian submanifolds of para-K\"ahler manifolds and the Monge-Kantorovich mass
transport problem.

Within the pc-projective geometry, many natural questions arise that have direct analogues in real and c-projective geometry (see \cite{bmmr,bmm,cemn,Lie,m_alone} and references therein).  

To start with, a classical problem is the so-called \emph{metrizability problem}, that goes as follows: given a pc-projective class $[\nabla]$, does it contain the Levi-Civita connection of some \pk metric? 
In the case that a pc-projective class $[\nabla]$ is metrizable, i.e., $[\nabla]=[\nabla^g]$ for some \pk metric $g$, the next natural question  is to understand ``how big'' is the dimension of the class $[g]$ of the pc-projectively equivalent metrics to $g$, i.e., its \emph{degree of mobility}. One can also investigate, within a pc-projective class, the existence of a metric with some distinguished properties, for instance, \emph{Einstein metrics} 
or \pk surfaces admitting pc-projective vector fields, i.e., the \emph{Lie problem} in the \pk context\footnote{In the original paper \cite{Lie}, S. Lie formulated the problem for real $2$-dimensional manifolds, even if such problem can be easily generalized to any dimension. One can also consult \cite{bmm,m_alone}. The \K counterpart can be found, for instance, in \cite{bmmr}.}.
%
%


The local geometry of \pk metrics clearly plays a fundamental role in the problems outlined above.  
Throughout the paper, we establish several results that illuminate different aspects of these questions within the \pk framework.  
The core of the work, however, is an exhaustive local description of pc-projectively equivalent \pk metrics on \pk surfaces, with particular attention devoted to the subclass of Einstein metrics.  
Although the problem can, in principle, be formulated in any even dimension, already in real dimension four the situation becomes substantially richer than in the \K setting, see Section \ref{sec.analogie}.

\subsection*{Notation and conventions}

For the definition of para-complex numbers, para-holomorphic functions, para-holomorphic atlas and para-holomorphic sectional curvature  we refer to \cite{al,amt,amt2,rockym,gm,hl,ms3,ms4}. 

\medskip
\noindent
We use the Einstein convention, i.e., repeated indices are implicitly summed over.

\noindent
Moreover, we denote by
\begin{itemize}
\item[-] $f_{x^i}$ the derivative of $f$ w.r.t. the $x^i$ variable;
\item[-] $\partial_{i}:=\partial_{x^i}$ the $i$-th coordinate vector field associated to a system of coordinate $(x^1,\dots,x^{2n})$, if there is no risk of ambiguity;
\item[-] $dx^idx^j$ the symmetric product between $dx^i$ and $dx^j$, i.e.,
$$
dx^idx^j:=\frac12 (dx^i\otimes dx^j + dx^j\otimes dx^i)\,;
$$
\item[-] $\mathcal{L}_X$ the Lie derivative along the vector field $X$;
\item[-] $\Gamma(E)$ the module of local sections of a bundle $E$;
\item[-] $\nabla^g$ the Levi-Civita connection of a metric $g$.
\end{itemize}


\subsection{Description of the main results}

\subsubsection{Preliminary results: Benenti tensors in the \pk context}

Even though affinely equivalent metrics remain of independent interest, in line with \cite{bmmr},
this paper focuses on the case that is most natural from the perspective of projective geometry: metrics that are pc-projectively equivalent but not affinely equivalent. To study such metrics, of fundamental importance is the following tensor field $A$ of type $(1,1)$: let $g,\hat g$ be two pc-projectively equivalent metrics, then we define
\begin{equation}\label{A.def}
A:=\left(\frac{\det \hat g}{\det g}\right)^{\frac{1}{2(n+1)}}\hat g^{-1}g.
\end{equation}
(We note that the previous quantity is well defined as $\frac{\det \hat g}{\det g}>0$ since \pk metrics possess neutral signature. Consequently, we also have 
$\det A>0$).
%
An important property of $A$ is that it is parallel w.r.t. the Levi-Civita connection $\nabla^{\hat g}$ of $\hat g$ (equivalently, w.r.t. $\nabla^g$) if and only if $g$ and $\hat g$ are affinely equivalent.  Indeed,
from \eqref{A.def}, we get $g = (\det A)^{1/2}\, \hat g A$.
Differentiating with respect to $\nabla^{\hat g}$ yields
$$\nabla^{\hat g} g =(\det A)^{1/2}\, \hat g \left(\tfrac12 \tr(A^{-1}\nabla^{\hat g} A)\, A  + \nabla^{\hat g} A
\right).$$
Hence, 
$$\nabla^{\hat g} g = 0
  \;\;\Longleftrightarrow\;\;
\tfrac12\,\tr(A^{-1}\nabla^{\hat g} A)\,\mathrm{Id}
   + A^{-1}\nabla^{\hat g} A = 0.$$
Taking the trace of the latter relation shows that $\tr(A^{-1}\nabla^{\hat g} A)=0$, and substituting back gives
$\nabla^{\hat g}A=0$.  
In particular, the vanishing of $\nabla^{\hat g} g$ is equivalent to $\nabla^{\hat g} A=0$, which precisely means that the Levi-Civita connections of $g$ and $\hat g$ coincide.

\smallskip  
To address the problem of finding normal forms of pc-projectively equivalent \pk metrics, as shown in Section \ref{sec:pc-invariant}, it is crucial to study the space  
\begin{equation}\label{eq:A.cal}
\mathcal{A}(g,\T)=\{A\in\Gamma(\mathrm{End}(TM))\;|\; A \text{ is $g$-symmetric, } [A,\T]=0,\ 
A \text{ satisfies \eqref{eq.A} below $\forall$ vector fields } X\},
\end{equation}
\begin{equation}\label{eq.A}
\nabla^g_X A
= X^\flat\otimes \Lambda+\Lambda^\flat \otimes X
-(\mathcal{T}X)^\flat\otimes \mathcal{T}\Lambda
-(\mathcal{T}\Lambda)^\flat \otimes \mathcal{T}X,
\end{equation}
where
\begin{equation}\label{tildeL}
\Lambda=\tfrac{1}{4}\grad\tr A\,,
\end{equation}
with ${}^\flat: TM\to T^*M$ the musical isomorphism associated to $g$. We note that
 \eqref{eq.A} is the  para-complex analogue of the equation governing c-projectively equivalent \K metrics (cf. \cite[equation (2.4)]{bmmr}).

In fact,
an important property of  tensors \eqref{A.def}, shown in the paper, is that they belong to   $\mathcal{A}(g,\T)$. Moreover, any non-degenerate $A\in\mathcal{A}(g,\T)$ is of type \eqref{A.def} (see Corollary \ref{cor:A-nondeg} below), in particular, its determinant is greater than zero in view of the fact the \pk metrics have neutral signature. Hence, by means of \eqref{A.def}, the metric 
\begin{equation}\label{eq: companion.A}
\hat{g}=(\det A)^{-\frac12} g A^{-1}
\end{equation}
is pc-projectively equivalent to $g$. In analogy with the definition given in the real case \cite{bm}, we call the subset of the non-degenerate tensors of $\mathcal{A}(g,\T)$ the space of the \emph{Benenti tensors} of the \pk metric $(g,\T)$:
$$
\textrm{\emph{Benenti tensors of} } (g,\T):= \{\, A\in\mathcal{A}(g,\T)\,\,|\,\, \det(A)\neq 0 \,\}\,.
$$ 
Indeed, as in the real case, in view of what  we said above, the knowledge of the space of the Benenti tensors of $(g,\T)$ allows to construct the entire class $[g]$ of pc-projectively equivalent metrics to $g$. In fact, if $A$ is a Benenti tensor, then the metric $\hat{g}$ given by \eqref{eq: companion.A} belongs to $[g]$ and formula \eqref{eq:lin.comb.metrics} gives a family of pc-projectively equivalent metrics to $g$.

\smallskip
On the other hand, from the symplectic viewpoint, if we denote $\phi=g(A\T\cdot,\cdot)$, where $A\in\mathcal{A}(g,\T)$, a straightforward computation shows that \eqref{eq.A} is equivalent to
\begin{equation}\label{eq:formula.Apos}
2\nabla_X^g \phi=d\tr_\omega\phi  \wedge(\T X)^\flat- \T d\tr_\omega \phi\wedge X^\flat\,.
\end{equation}
We note that this is precisely the equation of Hamiltonian $2$-forms, extensively studied in the K\"ahler setting (see, e.g., \cite{acg.hamiltonian.1, acg.hamiltonian.2, acg.hamiltonian.3, acg.hamiltonian.4, acg.self.dual}).

\subsubsection{Main theorems}
The first main result, that holds in any dimension, is contained in the theorem below and it concerns para-K\"ahler-Einstein metrics that are pc-projectively equivalent.
%
 %
%
\begin{thm}\label{thm:main.einstein}
Let $g$ and $\hat{g}$ be pc-projectively equivalent para-K\"ahler-Einstein metrics on the same manifold. Then metrics \eqref{eq:lin.comb.metrics} are para-K\"ahler-Einstein metrics belonging to the same pc-projective class of $g$.
\end{thm}
Using this theorem and the results contained in the forthcoming Theorem \ref{thm:main}, we shall classify, in Section
\ref{sec:Einstein.4.dim}, all four-dimensional para-K\"ahler-Einstein metrics possessing a non-parallel Benenti tensor.

\medskip\noindent
By considering all that we said so far, it is clear, in this context, the importance to study \pk metrics together with their Benenti tensors, i.e., the triple $(g,\T,A)$, with $A\in\mathcal{A}(g,\T)$ non-degenerate, up to local equivalence.
More precisely, we consider two triples $(g,\T,A)$ and $(g',\T',A')$ to be \emph{locally equivalent} if and only if there exists a local diffeomorphism 
$\phi$ such that
$$
\phi^{*} (g') = g, \quad 
\phi^{*} (\T') = \T, \quad 
\phi^{*} (A') = A .
$$
This can be done by taking into account  the scalar invariants of $A$.  
In particular, in the four-dimensional case, we consider the $\T$-invariant distribution  
\begin{equation}\label{eq:D}
D=\mathrm{span}\left\{\tfrac12\grad\tr A\,,\,\, \grad\sqrt{\det A} \,,\,\, \tfrac12\T \grad\tr A \,,\,\, \T  \grad\sqrt{\det A}\right\}\,,\end{equation}
whose rank is one of the tools for distinguishing different local models of $(g,\T,A)$.

\smallskip\noindent
Indeed, the second main result, that is contained in the theorem below, is a list of normal forms of \pk surfaces according to $\dim D$, i.e., a list of triple $(g,\T,A)$ up to the aforementioned equivalence. More precise criteria for distinguishing such normal forms are contained in Table \ref{tab:ranks}, page \pageref{tab:ranks}. 

\smallskip\noindent
Before stating such result, a consideration is due. In the four-dimensional case, Beneneti tensors can have at most two distinct eigenvalues. This can be seen if one writes such tensors in null coordinates as they assume a block-diagonal form with $2$ identical blocks (see Remark \ref{rem:sqrtA} for more details).
\begin{thm}\label{thm:main}
Let $(M,g,\T)$ be a \pk surface and $A\in\mathcal{A}(g,\T)$ be a non-parallel (w.r.t. the Levi-Civita connection $\nabla^g$ of $g$) Benenti tensor of $(g,\T)$. Let $\rho$ and $\sigma$ be the eigenvalues of $A$. Then, in a neighborhood of almost every point of $M$, there exists a system of coordinates $(x^1,x^2,x^3,x^4)$ in which the triple $(g,\T,A)$ takes one of the following forms: 
\begin{itemize}
\item Non-degenerate case: $\dim D=4$.
\begin{itemize}
\item real Liouville metrics:
\begin{eqnarray*}
&g\hspace{-0.2cm}&=(\rho-\sigma)\,(dx^1 dx^1+\varepsilon\,dx^2 dx^2)
-\frac{1}{\rho-\sigma}\Big(\rho_{x^1}^2(dx^3+\sigma\,dx^4)^2
+\varepsilon\,\sigma_{x^2}^2(dx^3+\rho\,dx^4)^2\Big),
\\
&\T\hspace{-0.2cm}&= g^{-1}\omega\,, \quad \mathrm{with}\quad\omega=\rho_{x^1}\, dx^1\wedge( dx^3+\sigma\,dx^4)+\sigma_{x^2}\,dx^2\wedge( dx^3+\rho\, dx^4)\,, 
\\[0.2cm]
&A\hspace{-0.2cm}&=\rho\,\partial_{1}\!\otimes dx^1
+\sigma\,\partial_{2}\!\otimes dx^2
+(\rho+\sigma)\,\partial_{3}\!\otimes dx^3
-\partial_{4}\!\otimes dx^3
+\rho\sigma\,\partial_{3}\!\otimes dx^4
\end{eqnarray*}
where $\varepsilon=\pm1$, $\rho=\rho(x^1)$ and $\sigma=\sigma(x^2)$ are arbitrary functions such that $\rho\neq 0$, $\sigma\neq 0$, $\rho_{x^1}\neq 0$, $\sigma_{x^2}\neq 0$.
\item complex Liouville metrics:
\begin{eqnarray*}
&g\hspace{-0.2cm}&=\tfrac14(\bar\rho-\rho)\,\big(dz^2 - d\bar z^2\big)
-\frac{4}{\rho-\bar\rho}\left(
\bar\rho_{\bar z}^2\big(dx^3+\rho\,dx^4\big)^2
-\rho_z^2\big(dx^3+\bar\rho\,dx^4\big)^2
\right),
\\
&\T\hspace{-0.2cm}&=g^{-1}\omega\,, \quad \mathrm{with}\quad\omega=\rho_z \;dz\wedge( dx^3+\bar \rho\,dx^4)+{\bar\rho}_{\bar z}\;d\bar z\wedge( dx^3+\rho\, dx^4),
\\[0.2cm]
&A\hspace{-0.2cm}&=\rho\,\partial_{z}\!\otimes dz
+\bar\rho\,\partial_{\bar z}\!\otimes d\bar z
+(\rho+\bar\rho)\,\partial_{x^3}\!\otimes dx^3
-\partial_{x^4}\!\otimes dx^3
+|\rho|^{2}\,\partial_{x^3}\!\otimes dx^4,
\end{eqnarray*}
%
%
where $z=x^1+i\,x^2$, $\rho=\rho(z)$ is an arbitrary function such that $\rho\neq 0$ and $\rho_z\neq 0$.
\end{itemize}
\item Degenerate case: $\dim D<4$.
\begin{itemize}
\item $\dim D=3$: no $(g,\T,A)$ satisfying the hypotheses of the theorem.
\item $\dim D=2$:
\begin{itemize}
\item 
$
g=\frac{\rho_{x^2}}{\mu} \, dx^1 dx^1 -  2\frac{\nu\rho_{x^2}}{\mu}\, dx^1 dx^3
-\mu\rho_{x^2} \, dx^2 dx^2 + \frac{\nu^2\rho_{x^2}}{\mu}\,  dx^3 dx^3
+2 (c-\rho)\nu_{x^4}\,  dx^3 dx^4,
$

\medskip

$\T=-\mu\,  \partial_1\otimes dx^2 + \nu\,  \partial_{1}\, \otimes dx^3 - \frac{1}{\mu}\,  \partial_2\otimes dx^1 + \frac{\nu}{\mu}\,  \partial_2\otimes dx^3 + \partial_3\otimes \partial_3  - \partial_4\otimes dx^4$,

\medskip

$
A=\rho\, (\partial_{1}\otimes dx^1+ \partial_{2}\otimes dx^2) + c\, (\partial_{3}\otimes dx^3+ \partial_{4}\otimes dx^4) + (c-\rho) \nu\, \partial_1\otimes dx^3,
$

\medskip
where $c\in\R\setminus\{0\}$, $\rho=\rho(x^2)$, $\mu=\mu(x^2)$ and $\nu=\nu(x^3,x^4)$ are arbitrary functions such that $\mu\neq 0$, $\rho_{x^2}\neq 0$, $\nu_{x^4}\neq 0$. Here $\sigma=c$.

\bigskip

\item $
g=2\rho_{x^3}\, dx^1dx^3  + 2\sigma_{x^4}\, dx^1 dx^4  + 2\sigma\rho_{x^3}\, dx^2 dx^3
+2 \rho\sigma_{x^4}\, dx^2 dx^4,  
$

\medskip

$\T$ is given by \eqref{eq:matrix.T.adapted} with $n=2$, 

\medskip

$A=(\rho+\sigma)\, \partial_{1}\otimes dx^1 +\rho\sigma\,  \partial_{1}\otimes dx^2 - \partial_{2}\otimes dx^1 + \rho\,  \partial_{3}\otimes dx^3 + \sigma\,  \partial_{4}\otimes dx^4$,

\medskip

where $\rho=\rho(x^3)$ and $\sigma=\sigma(x^4)$ are arbitrary functions such that $\rho\neq 0$, $\sigma\neq 0$, $\rho_{x^3}\neq 0$, $\sigma_{x^4}\neq 0$.
 
\bigskip

\item $(g,-\T,A)$ where $g,\T,A$ are the same as the above point.

\bigskip

\item 
$g=2\rho_{x^3}\, dx^1 dx^3  -2 \sigma_{x^4}\, dx^1dx^4 +2\sigma\rho_{x^3}\, 2dx^2 dx^3  - 2\rho\sigma_{x^4}\, dx^2 dx^4$,

\medskip

$\T=\frac{\rho+\sigma}{\rho-\sigma}\, \partial_1\otimes dx^1 +\frac{2\rho\sigma}{\rho-\sigma}\, \partial_1\otimes dx^2 -\frac{2}{\rho-\sigma}\, \partial_2\otimes dx^1 - \frac{\rho+\sigma}{\rho-\sigma}\, \partial_2\otimes dx^2 -\partial_3\otimes dx^3 + \partial_4\otimes dx^4$,

\medskip

$A=(\rho+\sigma)\, \partial_1\otimes dx^1 + \rho\sigma\, \partial_1\otimes dx^2 - \frac{k\sigma\rho_{x^3}}{\rho-\sigma}\, \partial_1\otimes dx^3 - \frac{k\rho\sigma_{x^4}}{\rho-\sigma}\, \partial_1\otimes dx^4 - \partial_2\otimes dx^1 + \frac{k\rho_{x^3}}{\rho-\sigma}\, \partial_2\otimes dx^3 + \frac{k\sigma_{x^4}}{\rho-\sigma}\,  \partial_2\otimes dx^4 + \rho\, \partial_3\otimes dx^3 + \sigma \, \partial_4\otimes dx^4$,

\medskip

where $k\in\R$, $\rho=\rho(x^3)$ and $\sigma=\sigma(x^4)$ are arbitrary functions such that $\rho\neq 0$, $\sigma\neq 0$, $\rho_{x^3}\neq 0$, $\sigma_{x^4}\neq 0$.

\end{itemize}
\item $\dim D=1$:
\begin{itemize}
\item 
$
g=2\rho_{x^3} \, dx^1 dx^3 + 2\big((\rho-c)F\big)_{x^3} \, dx^2 dx^3  + 
2(\rho-c)F_{x^4}\,  dx^2dx^4$,

\medskip

$\T$ is given by \eqref{eq:matrix.T.adapted} with $n=2$, 

\medskip

$
A=\rho\,\partial_1\otimes dx^1 + (\rho-c)F\,\partial_1\otimes dx^2 + c\, \partial_2\otimes dx^2 + \rho\,\partial_3\otimes dx^3 + (c-\rho)\frac{F_{x^3}}{F_{x^4}}\,\partial_4\otimes dx^3+ c\, \partial_4\otimes dx^4
$,

\medskip

where  $c\in\R\setminus\{0\}$, $\rho=\rho(x^3)$ and $F=F(x^2,\varphi(x^3,x^4))$ are arbitrary functions such that $\rho\neq 0$, $\rho_{x^3}\neq 0$, $F_{x^4}\neq 0$. Here $\sigma=c$.

\bigskip

\item $(g,-\T,A)$ where $g,\T,A$ are the same as the above point.
\end{itemize}
\end{itemize}

\end{itemize}
\end{thm}

\begin{cor}
Let $g$ and $\hat{g}$ be two pc-projectively equivalent metrics on the same \pk surface. Then either there exists a system of coordinates $(x^1,x^2,x^3,x^4)$ where  $g$ assumes one of the form listed in Theorem \ref{thm:main} and $\hat{g}$ the form 
\eqref{eq: companion.A} or they are affinely equivalent.
\end{cor}

\begin{cor}
As a by-product, the results of Theorem \ref{thm:main} straightforwardly provide normal forms of Hamiltonian $2$-forms in the \pk setting, see discussions around formula \eqref{eq:formula.Apos}.
\end{cor}

\begin{rem}
It is worth noting that several of the degenerate normal forms listed in Theorem \ref{thm:main} give rise to the same metric up to local isometries. For example, the second, third and fourth families in the case $\dim D = 2$  are  locally isometric to the standard flat metric with neutral signature.
Their presence is not accidental. Indeed, our classification concerns the triples $(g,\T,A)$ rather than the metric $g$ alone.

\end{rem}
\subsubsection{Differences with the \K case and new strategies}\label{sec.analogie}
The variety of normal forms arising in the \pk setting is richer than in the  \K one. Indeed, while the non-degenerate situation $\dim D = 4$ and the first family in the case $\dim D = 2$ admit direct counterparts in the \K classification (cf. \cite[Theorem 3.1]{bmmr}), none of the other degenerate configurations occurring in Theorem \ref{thm:main}  has a \K analogue.
This discrepancy reflects intrinsic algebraic differences between the para-complex and complex structures. In particular, phenomena such as isotropic gradients of the eigenvalues of Benenti tensors and the appearance of odd-dimensional ranks for the distribution $D$, have no counterpart in the \K context. These structural reasons, discussed more in details in Remark \ref{rem:kahler.contrast}, explain why the \pk theory naturally exhibits a broader spectrum of local models.  Consequently, one cannot rely on what is known in the \K setting to guide the analysis in the \pk case. For instance, the construction of local frames adapted to $D$ in the degenerate cases, offers a particularly illustrative example of a situation where techniques from the \K case cannot be adapted.
Even though certain aspects of the analysis may still be guided by the conceptual analogy, without however becoming straightforward as a consequence, these analogies nevertheless remain insufficient to fully capture the phenomena arising in the \pk setting. The latter ultimately requires additional arguments specifically tailored to its own structural features.

\subsection{Structure of the paper}
%

\medskip
In Section \ref{sec:pc-invariant} we show that the metrizability problem, raised in Section \ref{sec:motivations}, though apparently non-linear, actually reduces to the solvability of a pc-projectively invariant linear system of PDEs. Consequently, being the space of its solutions a vector space,  one can meaningfully speak of its dimension (called \emph{degree of mobility}, cf. Section \ref{sec:motivations}). By expressing the aforementioned system in terms of Benenti tensors, we obtain equation \eqref{eq.A}. This viewpoint provides the conceptual foundation for several results developed in the paper.

%

\smallskip
In Section \ref{sec:killing} we construct canonical Killing vector fields naturally associated with tensors $A\in\mathcal A(g,\mathcal T)$.  These vector fields, obtained from scalar invariants of $A$, constitute the first step towards the construction of the coordinate systems where the normal forms of Theorem \ref{thm:main} will be expressed.  A key result of the section is a formula describing the difference of the Ricci tensors of two pc-projectively equivalent metrics, an identity that will play a central role also in the proof of Theorem \ref{thm:main.einstein}.

\smallskip
In Section \ref{sec:pcEinstein} we completely dedicate our attention to the proof of Theorem \ref{thm:main.einstein}.

\smallskip
In Sections \ref{sec:preliminary}, \ref{sec:proof} and \ref{sec:Einstein.4.dim} we deal with results specific for \pk surfaces, in contrast with those of the previous sections which apply in arbitrary dimension.

\smallskip
In Section \ref{sec:preliminary} we describe the dense open subset where Theorem \ref{thm:main} holds and we establish the precise criteria, summarized in Table \ref{tab:ranks} at page \pageref{tab:ranks}, for distinguishing the various types of normal forms.  We also study, from the symplectic viewpoint, the canonical Killing vector fields constructed in Section 3 together with their images under the para-complex structure, showing that they are also para-holomorphic. This analysis plays a crucial role in the construction of the local coordinates used in Section \ref{sec:proof}.

\smallskip
Indeed, in Section \ref{sec:proof} we focus our attention mainly to the proof of Theorem \ref{thm:main}.  The different cases leading to the normal forms of Theorem \ref{thm:main}
are analyzed separately in different subsections. For each case, we construct a distinguished set of four commuting vector fields that yield suitable local coordinates where the equation \eqref{eq.A} is explicitly studied.
Some vector fields are chosen within the canonical Killing vector fields introduced in Section \ref{sec:killing} together with their images under the \pk structure; in some situations, this choice needs to be completed by additional vector fields, which are discussed case by case in the relevant subsections.

\smallskip
Finally, in Section \ref{sec:Einstein.4.dim} we determine when a metric $g$ occurring in Theorem \ref{thm:main} is Einstein and when the associated metric $\hat g$ given by
\eqref{eq: companion.A} is Einstein simultaneously to $g$. This way, as a by-product of Theorem \ref{thm:main.einstein}, we obtain a $2$-parametric family of para-K\"ahler-Einstein metrics admitting a non-parallel Benenti tensor.

%

\section{Pc-projectively invariant equation, metrizability problem and Benenti tensors}\label{sec:pc-invariant}
 
In Section \ref{sec:nablaXsigma} we show that the metrizability problem, i.e., the problem to find a Levi-Civita connection $\nabla^g$ within a given pc-projective class of connections $[\nabla]$, can be stated in terms of existence of solutions of a pc-projectively invariant linear system of PDEs, equation \eqref{eq:pcequivglob}, having a $(2,0)$ weighted tensor field as unknown, that we denote by $\sigma$.

\smallskip
In Section \ref{sec:nablaXA}, assuming the metrizability of the pc-projective class,
we rephrase the aforementioned linear system in terms of certain $(1,1)$-tensor fields, finally obtaining equation \eqref{eq.A}, that we shall prove to constitute the set $\mathcal{A}(g,\T)$.

\subsection{Pc-projectively invariant equation in $(2,0)$ weighted tensors}\label{sec:nablaXsigma}

Let $\nabla$ and $\widehat\nabla$ be two pc-projectively equivalent connections on a para-complex manifold $(M,\T)$.
By comparing the corresponding covariant derivatives, one straightforwardly finds that their difference can be expressed in terms of a suitable $1$-form $\Psi$ as
\begin{equation}\label{eq.nabladifference}
\widehat\nabla_X Y - \nabla_X Y
= \Psi(X)Y + \Psi(Y)X + \Psi(\T X)\T Y + \Psi(\T Y)\T X\,.
\end{equation}
In local coordinates, the previous equation takes the form
\begin{equation}\label{eq:tildeG-G}
\widehat\Gamma^k_{ij} - \Gamma^k_{ij}
= \Psi_i \delta^k_j + \Psi_j \delta^k_i
  + \Psi_p \T^p_i \T^k_j + \Psi_p \T^p_j \T^k_i\,,
\end{equation}
where $\Gamma^k_{ij}$ and $\widehat\Gamma^k_{ij}$ denote the Christoffel symbols of $\nabla$ and $\widehat\nabla$, respectively, and $\Psi_i$ are the local components of the $1$-form $\Psi$.
Conversely, if \eqref{eq.nabladifference} holds, then $\nabla$ and $\widehat\nabla$ are pc-projectively equivalent; a detailed justification of this fact is given in the proof of Proposition \ref{prop:sigma-metric} below.

\smallskip
We denote by 
$$\mathcal{S}^2_{\T} TM 
= \{\eta \in \Gamma(S^2 TM) \;:\; (\T\otimes\T)\, \eta=-\eta\}$$
the set of the sections of  the bundle of para-Hermitian symmetric $(2,0)$-tensors.  In local coordinates, $\eta\in \mathcal{S}^2_{\T} TM$ if and only if
 \begin{equation}\label{sigma.paraHerm}
\T^j_p \eta^{pk} =- \eta^{jp} \T^k_p.\end{equation}
Then, we  consider a section $\sigma$ belonging to $$
\mathcal{S}^2_{\T} TM \otimes\Gamma\left(\mathrm{Vol}^{\frac{1}{n+1}}\right)\,,$$
where
 $\mathrm{Vol}^{\frac{1}{n+1}}$  denotes the space of tensor fields $\Lambda^{2n}\,T^*M$  with weight $\frac{1}{n+1}$. 
In local coordinates,  
\begin{equation*}
\sigma = \sigma^{jk} (\partial_j \otimes \partial_k) \otimes \mathrm{vol}^{\frac{1}{n+1}},\quad \mathrm{vol}:=dx^1\wedge\cdots\wedge dx^{2n}.
\end{equation*}
We recall, by definition of weighted tensors, that if we perform a change of coordinates $x\to y(x)$, denoting $J$ its Jacobian, then the new components of $\sigma$ are
$$
\tilde  \sigma^{km} =
  |\det J\, |^{-\frac{1}{n+1}}\sigma^{ij}J^k_i J^m_j.
$$
Moreover, the covariant derivative $\nabla \sigma$ reads locally as
\begin{equation}\label{nabla.sigma}
\nabla_i\sigma^{jk}
=\sigma^{jk}_{,i}+\Gamma^j_{im}\sigma^{mk} +\Gamma^k_{im} \sigma^{mj}- 
\tfrac{1}{n+1} \Gamma^p_{ip}\sigma^{jk}.
\end{equation}

\begin{prop} Let $(M,\T)$ be a para-complex manifold and  $\nabla$ be a para-complex connection on $M$. Then, the equation
\begin{equation}\label{eq:pcequivglob}
 \nabla_X \sigma- X\otimes L - L \otimes X +\T X\otimes \T L +\T L \otimes \T X=0,
\end{equation}
where $L=\tfrac{1}{2n}\nabla_l \sigma^{lk}\,\partial_k\,\otimes (\mathrm{vol})^{\frac{1}{n+1}}$,
is \emph{pc-projectively invariant}, in the sense that it does not depend on the choice of $\nabla$ in its pc-projective class $[\nabla]$. 
\end{prop}
\begin{proof}
Let $\widehat\nabla$ and $\nabla$ be two pc-projectively equivalent connections.
Formulas  \eqref{eq:tildeG-G} and  \eqref{nabla.sigma} allow one to compare the covariant derivatives of $\sigma$ with respect to the  two pc-projectively equivalent connections, leading to the local equations
\begin{equation*}
\left(\widehat\nabla_i - \nabla_i\right)\,\sigma^{jk}
=\Psi_p \delta^j_i \sigma^{pk}+ \Psi_p \delta^k_i \sigma^{jp} + 
 \Psi_l \left(\T^l_p \T^j_i\sigma^{pk}
+  \T^l_p \T^k_i\sigma^{jp}\right).
\end{equation*}
Keeping in mind that $\sigma$ is para-Hermitian, see \eqref{sigma.paraHerm}, we obtain
\begin{equation}\label{nabla.sigma.calcoli}
\left(\widehat\nabla_i - \nabla_i\right)\,\sigma^{jk}
=\Psi_p \delta^j_i \sigma^{pk}+ \Psi_p \delta^k_i \sigma^{jp} -
  \T^j_i  \T^k_p\Psi_l \sigma^{lp}-     \T^k_ i\T^j_p \Psi_l\sigma^{lp},
\end{equation}
which in turn implies that
\begin{equation}\label{eq:nabla.sigma.tr}
\widehat \nabla_l\sigma^{lj}=\nabla_l \sigma^{lj}+
 2n\,\Psi_p  \sigma^{pj}\,.
\end{equation}
By using the relation just derived in formula  \eqref{nabla.sigma.calcoli}, we have that 
\begin{equation}\label{eq:pcequiv}
 \nabla_i\sigma^{jk}-\frac{1}{2n}\left(
\delta^j_i \,\nabla_l \sigma^{lk}+\delta^k_i \, \nabla_l \sigma^{lj}-
  \T^j_i  \T^k_p\,\nabla_l \sigma^{lp}-     \T^k_ i \T^j_p\, \nabla_l \sigma^{lp} \right)
\end{equation}
does not depend  on the choice of $\nabla$ in the pc-projective class $[\nabla]$.
This is precisely the coordinate expression of the left-hand side of equation \eqref{eq:pcequivglob}.
\end{proof}

\smallskip 
To test the existence of a \pk metric in a given pc-projective class, we can look for a non-degenerate solution $\sigma$ of equation \eqref{eq:pcequivglob}.  
The precise correspondence between \pk metrics and such solutions is written more explicitly in the following proposition. 
\begin{prop}\label{prop:sigma-metric}
Let $(M,\mathcal{T})$ be a para-complex manifold and $[\nabla]$ a pc-projective class of para-complex connections on $M$.  
Then there exists a \pk metric $g$ whose Levi-Civita connection $\nabla^g$ belongs to $[\nabla]$ if and only if the  pc-projectively invariant equation \eqref{eq:pcequivglob}  admits a solution whose matrix $(\sigma^{ij})$ is non-degenerate at every point. 
In this case, $g$ and $\sigma$ are related by
\begin{equation}\label{g.sigma}
g^{ij} = |\det\sigma|^{\frac{1}{2}}\, \sigma^{ij},\qquad\quad \sigma^{ij} = |\det g|^{\frac{1}{2(n+1)}}\, g^{ij}.\end{equation}
\end{prop}
\begin{proof}
We first prove the sufficiency.  
Assume that $\sigma$ is a non-degenerate solution of equation \eqref{eq:pcequivglob}, and define the tensor field $g^{ij}$ according to the first equality in \eqref{g.sigma}.
Its inverse $g_{ij}$ defines a pseudo-Riemannian metric $g$ satisfying
$$g(\T X,\T Y)=-\,g(X,Y),$$
so that $(M,g,\T)$ is a para-Hermitian manifold. 

\smallskip
We now construct a connection $\widehat\nabla$ that will turn out to be the Levi-Civita connection $\nabla^g$ of $g$ and that belongs to the pc-projective class $[\nabla]$.  
Let $\widehat\nabla$ be the connection whose Christoffel symbols are related to those of $\nabla$ by \eqref{eq:tildeG-G} with
\begin{equation}\label{psi.def}
\Psi_i=-\frac{1}{2n}\,\sigma_{im}\,\nabla_l\sigma^{ml}.
\end{equation}
From formula \eqref{eq.nabladifference}, one sees that $\widehat\nabla$ is torsion-free whenever $\nabla$ is so.  
Moreover, the same relation shows that $\T$ is $\widehat\nabla$-parallel: indeed, for all vector fields $X,Y$,
\begin{align*}
\widehat\nabla_X(\T Y)
&=\nabla_X(\T Y)
+\Psi(X)\T Y+\Psi(Y)\T X
+\Psi(\T X)Y+\Psi(\T Y)X\\
&=\T\bigl(\nabla_X Y+\Psi(X)Y+\Psi(Y)X
+\Psi(\T X)\T Y+\Psi(\T Y)\T X\bigr)
=\T \bigl(\widehat\nabla_X Y\bigr),
\end{align*}
so that $\widehat\nabla\T=0$.  
Hence $\widehat\nabla$ is a para-complex connection pc-projectively related to $\nabla$.

\smallskip
Since equation \eqref{eq:pcequivglob} is pc-projectively invariant, we can perform computations  by considering any representative of the class $[\nabla]$.  
Using the local expression \eqref{eq:pcequiv} of \eqref{eq:pcequivglob} and taking into account \eqref{eq:nabla.sigma.tr} with $\Psi$ given by \eqref{psi.def}, one readily checks that
$$\widehat\nabla_k\sigma^{ij}=0.$$
Hence, 
$$
\widehat\nabla_k\!\left(\sigma^{ij}\,|\det\sigma|^{\frac12}\right)=0\,.
$$
By \eqref{g.sigma}, this shows that $\widehat\nabla g^{ij}=0$, i.e., $\widehat\nabla$ is the Levi-Civita connection of $g$.  
Therefore $\widehat\nabla=\nabla^g$, completing the proof of the sufficiency.

\smallskip
Now, we prove the necessary condition. Assume that $g$ is a \pk metric whose Levi-Civita connection $\nabla^g$ belongs to the pc-projective class $[\nabla]$. Define $\sigma$ according to the second equality in \eqref{g.sigma}.
Since $\nabla^g \,g^{ij}=0$, it follows that $\nabla^g\,\sigma^{ij}=0$.  
Hence $\sigma$ is a non-degenerate, para-Hermitian solution of \eqref{eq:pcequivglob} for $\nabla=\nabla^g$, which proves the necessary condition.
\end{proof}

\begin{cor}\label{cor:lin.comb.metrics}
If $\sigma$ and $\hat{\sigma}$ are solutions to equation \eqref{eq:pcequivglob}, also $\alpha\sigma + \beta\hat{\sigma}$, $\alpha,\beta\in\R$, is a solution as \eqref{eq:pcequivglob} is a linear system of PDEs. In the case that both $\sigma$ and $\hat{\sigma}$ are non-degenerate, metrics corresponding to  $\alpha\sigma + \beta\hat{\sigma}$ via formulas \eqref{g.sigma} are exactly \eqref{eq:lin.comb.metrics}, that then turn out to be projectively equivalent each other.
\end{cor}

\subsection{Equation governing pc-projectively equivalent \pk metrics }\label{sec:nablaXA}

Here we specialize the results of previous section in the case that the metrizability problem has a positive answer, hence we start from a \pk manifold $(M,g,\T)$. In what follows $\sigma$ is given by \eqref{g.sigma}.

\begin{prop}\label{prop:A-equation}
Let $(M,g,\T)$ be a \pk manifold. 
Let  $\hat\sigma$ be an arbitrary solution (possibly degenerate)  of the  pc-projectively invariant equation \eqref{eq:pcequivglob} with $\nabla=\nabla^g$. Then, the tensor field 
\begin{equation}\label{eq:Adef}
A = \hat\sigma\, \sigma^{-1}
= |\det g|^{-\frac{1}{2(n+1)}}\,\hat\sigma\,g\,
\end{equation}
is  $g$-symmetric, commutes with $\T$ and satisfies equation \eqref{eq.A}.
Such tensors constitute precisely the set $\mathcal{A}(g,\T)$.
\end{prop}
\begin{proof}
From the definition \eqref{eq:Adef}, it is immediate to see that $A$ is $g$-symmetric and commutes with $\T$.  

\noindent
Now, we prove that $A$ satisfies \eqref{eq.A}.
Since $g$, and hence $\sigma$, are parallel for the Levi-Civita connection $\nabla^g$, we have $\nabla^g\sigma^{-1}=0$.  
Therefore,  for any vector field  $X$, we have
$$\nabla^g_X A=(\nabla^g_X\hat\sigma)\,\sigma^{-1}.$$
By assumption, $\hat\sigma$ satisfies the pc-projectively invariant equation \eqref{eq:pcequivglob}. By inserting that expression into the previous equality, we obtain
$$
\nabla^g_X A
= \bigl(X^\flat\otimes L + L^\flat\otimes X
-(\mathcal{T}X)^\flat\otimes\mathcal{T}L
-(\mathcal{T}L)^\flat\otimes\mathcal{T}X\bigr)\,\sigma^{-1}\,.
$$ 
Hence, there exists a vector field $\Lambda$ (depending on $L$ and $\sigma^{-1}$) such that the above equality can be written as
$$
\nabla^g_X A
= X^\flat\otimes\Lambda+\Lambda^\flat\otimes X
-(\mathcal{T}X)^\flat\otimes\mathcal{T}\Lambda
-(\mathcal{T}\Lambda)^\flat\otimes\mathcal{T}X.
$$
It remains to identify $\Lambda$.  Taking the trace of the last equality, yields
$$
X(\tr A)=4\,g(\Lambda,X)\quad\text{for any }X.
$$
Therefore, $\Lambda=\tfrac14\grad\tr A$, and the proposition follows.
\end{proof}
To sum up, if $\hat\sigma$ is a para-Hermitian solution of the pc-projectively invariant equation \eqref{eq:pcequivglob}, then the tensor $A$ defined by \eqref{eq:Adef} satisfies equation \eqref{eq.A}. Hence, any tensor obtained in this way belongs to $\mathcal{A}(g,\T)$, as proved by Proposition \ref{prop:A-equation}.  Conversely, every $A\in\mathcal{A}(g,\T)$ can be written as $A=\hat\sigma\,\sigma^{-1}$ for some para-Hermitian solution $\hat\sigma$ of the pc-projectively invariant equation \eqref{eq:pcequivglob}.  Indeed, given $A\in\mathcal{A}(g,\T)$, setting $\hat\sigma:=A\sigma$, we get a para-Hermitian tensor that satisfies this equation since $\nabla^g\sigma=0$ and the defining relation \eqref{eq.A} for $A$ directly implies the required form of $\nabla^g\hat\sigma$.
In addition, Proposition \ref{prop:sigma-metric} further relates invertible solutions of the pc-projectively invariant equation \eqref{eq:pcequivglob} to \pk metrics whose Levi-Civita connection lies in a given pc-projective class.  
Since the ultimate goal of this work is the study of \pk metrics that are pc-projectively equivalent to a fixed metric $g$, we shall hence restrict our attention to the non-degenerate solutions of the pc-projectively invariant equation \eqref{eq:pcequivglob}, that is, to the non-degenerate elements of $\mathcal{A}(g,\T)$, i.e., Benenti tensors.

\begin{cor}\label{cor:A-nondeg}
Benenti tensors 
are precisely those of the form \eqref{A.def},
where $\hat g$ is a \pk metric pc-projectively equivalent to $g$.
\end{cor}
\begin{proof}
The claim follows directly from the definition of $A$ in \eqref{eq:Adef} and from Proposition \ref{prop:sigma-metric}, in particular from the second equality of \eqref{g.sigma}.
\end{proof}
\begin{rem}\label{rem:sqrtA}
If we fix a local system of null coordinates, i.e., coordinates adapted to the eigendistributions $T^{\pm}$ of $\T$ (see Section \ref{intro}), then, as a consequence of Corollary \ref{cor:A-nondeg}, the matrix of components of $A$ takes a block-diagonal form,
$$A = 
\begin{pmatrix}
A' & 0\\
0 & A'
\end{pmatrix},$$
where $A'$ is a $n\times n$ matrix. Indeed, the \pk structure $\T$ is as in \eqref{eq:matrix.T.adapted} and \pk metrics, being para-Hermitian, in these coordinates have an off-diagonal block form. In particular, one has $\det A>0$, and
\begin{equation*}
\det A' =:\sqrt{\det A}
= \left( \frac{\det \hat g}{\det g} \right)^{-\tfrac{1}{2(n+1)}},
\end{equation*}
where $\hat g$ is a \pk metric pc-projectively equivalent to $g$.  The last equality follows directly from the definition \eqref{A.def} of $A$.
Moreover, it is interesting to note that the function $\det A'$ is directly related to the 1-form $\Psi$ appearing in \eqref{eq.nabladifference}. Indeed, whenever that relation holds, taking its trace, we obtain
\begin{equation*}
\Psi_i = \frac{1}{2(n+1)}\bigl(\widehat\Gamma_{ij}^j - \Gamma_{ij}^j\bigr).
\end{equation*}
Hence, $\Psi$ is exact and it can be written as
\begin{equation}\label{eq:Psi.exact}
\Psi = d\psi, \qquad \text{with}\quad \psi = \frac{1}{4(n+1)}\log\left(\frac{\det \hat g}{\det g}\right)=-\frac14\log \det A\,.
\end{equation}
Since the Levi-Civita connections $\nabla^{\hat g}$ and $\nabla^g$ satisfy \eqref{eq.nabladifference}, then
it  follows  that
\begin{equation}\label{eq:sqrtdetA.exp}
\det A' = e^{-2\psi}. 
\end{equation}
\end{rem}

\section{Canonical Killing vector fields and Ricci tensors of pc-projectively equivalent \pk metrics} \label{sec:killing}
In this section, we study vector fields naturally associated with  scalar invariants of $A\in\mathcal{A}(g,\T)$, which are the coefficients of its characteristic polynomial.  
Their geometric significance will become particularly clear in Section \ref{sec:proof}, where they serve as building blocks of the coordinate systems used to express the local normal forms contained in Theorem \ref{thm:main}. In analysing these vector fields, we shall also obtain a general identity for the difference of the Ricci tensors of pc-projectively equivalent \pk metrics.  This formula, stated in \eqref{eq:main}, will play a central role in the proof of Theorem \ref{thm:main.einstein}.

\begin{prop}\label{prop:det.killing}
Let $(M,g,\T)$ be a \pk manifold and let $A \in \mathcal{A}(g,\T)$.
Then the vector field
$$\T\grad\sqrt{\det A}$$
is a Killing vector field of $g$.
\end{prop}
\begin{proof}
To prove the proposition, it suffices to check locally that the vector field $\T \grad \sqrt{\det A}$ is Killing, that is,
\begin{equation}\label{eq:killing.local}
\nabla_i^g (\T \grad \sqrt{\det A})_j + \nabla_j^g (\T \grad \sqrt{\det A})_i = 0
\quad \text{for all } i,j.
\end{equation}
Let us recall that the symplectic form $\omega$ of the \pk manifold is given by \eqref{eq:omega}. Since, by definition, $\T \grad \sqrt{\det A} = \omega^{-1}\, d(\sqrt{\det A})$,  we have, by Remark \ref{rem:sqrtA} and in particular by formulas \eqref{eq:Psi.exact} and \eqref{eq:sqrtdetA.exp}, that
\begin{equation*}
\T \grad \sqrt{\det A} = -2\, e^{-2\psi}\, \omega^{-1}(\Psi).
\end{equation*}
Recalling that $\nabla^g$  preserves $\T$, i.e. $\nabla^g\T=0$, a direct computation then yields
\begin{equation*}
\nabla_i^g (\T \grad \sqrt{\det A})_j + \nabla_j^g (\T \grad \sqrt{\det A})_i
=-2 \bigl(-2 \Psi_i \Psi_k + \nabla_i^g \Psi_k\bigr)\, \T^k_j
-2 \bigl(-2 \Psi_j \Psi_k + \nabla_j^g \Psi_k\bigr)\, \T^k_i.
\end{equation*}
Hence, the Killing condition \eqref{eq:killing.local} is equivalent to requiring that the $(0,2)$-tensor
\begin{equation*}
S := 2\, \Psi \otimes \Psi - \nabla^g \Psi
\end{equation*}
is para-Hermitian, i.e., 
$$
S(\T \cdot, \T \cdot) = -S(\cdot, \cdot)\,.
$$
Below we shall prove it.

\noindent
To this aim, we now consider a \pk metric $\hat g$ on $M$ pc-projectively equivalent to $g$.
In particular, we take into account the relation between the Ricci tensors of $g$ and $\hat g$  by exploiting formula \eqref{eq.nabladifference} and \eqref{eq:tildeG-G}.  
Recalling that the components of the curvature tensor $R$ of $g$ are given by
$$
R^k_{l ij} = \partial_i \Gamma^k_{l j} - \partial_j \Gamma^k_{l i} 
+ \Gamma^k_{i r}\Gamma^r_{l j} - \Gamma^k_{j r}\Gamma^r_{l i}
$$
(and analogously the components of the curvature tensor $\widehat R$ of $\hat{g}$), it follows that
\begin{equation*}
\widehat R^k_{l ij}- R^k_{l ij} =
 \nabla_i^g  \bigl(\widehat\Gamma^k_{l j} - \Gamma^k_{l j}\bigr) 
- \nabla_j^g  \bigl(\widehat\Gamma^k_{l i} - \Gamma^k_{l i}\bigr) 
+  \bigl(\widehat\Gamma^k_{i r}  - \Gamma^k_{i r} \bigr) 
   \bigl(\widehat\Gamma^r_{l j} - \Gamma^r_{l j}\bigr) 
-  \bigl(\widehat\Gamma^k_{j r} - \Gamma^k_{j r}\bigr)
   \bigl(\widehat\Gamma^r_{l i} - \Gamma^r_{l i}\bigr).
\end{equation*}
Taking the trace with respect to $k$ and $i$, using the fact that $\nabla^g\T=0$ and considering \eqref{eq:tildeG-G}, a tedious computation yields the following relation between the Ricci tensors of $g$ and $\hat g$:
\begin{equation*}
\widehat\Ric_{l j}-\Ric_{l j}
=  \nabla_l^g \Psi_j - (2n+1)\nabla_j^g \Psi_l
+ \T^r_l \T^k_j \nabla_k^g \Psi_r
+ \T^r_j \T^k_l \nabla_k^g \Psi_r
+ 2(n-1)\Psi_l\Psi_j
+ 2(n-1)\Psi_r \T^r_l\, \Psi_s \T^s_j.
\end{equation*}
Since $\Psi$ is an exact one-form, its covariant derivative $\nabla \Psi$ is symmetric.  Therefore, the previous relation can be conveniently rewritten as
\begin{equation}\label{eq:Ric.difference}
\widehat\Ric-\Ric
=-2n\bigl( \nabla^g\Psi-\Psi \otimes \Psi-(\Psi \circ \T) \otimes (\Psi \circ \T) \bigr)
+2\bigl(\nabla^g \Psi(\T\cdot, \T\cdot) -\Psi \otimes \Psi-(\Psi \circ \T) \otimes (\Psi \circ \T)\bigr).
\end{equation}
It is well known that the Ricci tensor of any \pk metric is para-Hermitian (see, for instance, \cite[Proposition 5.4]{amt}); consequently, the right-hand side of \eqref{eq:Ric.difference} must also be para-Hermitian.  Applying the defining property of para-Hermitian $(0,2)$ tensors, we obtain
\begin{multline*}
-2n\bigl( \nabla^g\Psi-\Psi \otimes \Psi-(\Psi \circ \T) \otimes (\Psi \circ \T) \bigr)
+2\bigl(\nabla^g \Psi(\T\cdot, \T\cdot) -\Psi \otimes \Psi-(\Psi \circ \T) \otimes (\Psi \circ \T)\bigr)
\\
=2\bigl( \nabla^g\Psi-\Psi \otimes \Psi-(\Psi \circ \T) \otimes (\Psi \circ \T) \bigr)
-2n\bigl(\nabla^g \Psi(\T\cdot, \T\cdot) -\Psi \otimes \Psi-(\Psi \circ \T) \otimes (\Psi \circ \T)\bigr),
\end{multline*}
which implies that
\begin{equation*}
\nabla^g\Psi-\Psi \otimes \Psi-(\Psi \circ \T) \otimes (\Psi \circ \T)
=\nabla^g \Psi(\T\cdot, \T\cdot) -\Psi \otimes \Psi-(\Psi \circ \T) \otimes (\Psi \circ \T).
\end{equation*}
Hence, the difference of the Ricci tensors \eqref{eq:Ric.difference} can be expressed in the more compact form:
\begin{equation}\label{eq:main}
\widehat\Ric-\Ric
=-2(n-1)\bigl( \nabla^g\Psi-\Psi \otimes \Psi-(\Psi \circ \T) \otimes (\Psi \circ \T) \bigr).
\end{equation}
Finally, observing that the left-hand side of
\begin{equation*}
\widehat\Ric-\Ric
+2(n-1)\bigl(\Psi \otimes \Psi-(\Psi \circ \T) \otimes (\Psi \circ \T)\bigr)
=2 (n-1)\bigl( 2\,\Psi \otimes \Psi- \nabla^g \Psi\bigr),
\end{equation*}
 is para-Hermitian, we conclude that 
$$
2\,\Psi \otimes \Psi- \nabla^g \Psi
$$
is para-Hermitian as well.  
This completes the proof.
\end{proof}

\smallskip
The previous proposition identifies a Killing vector field associated with any tensor $A\in\mathcal{A}(g,\T)$.  
We now extend this construction to associate with $A$ a set of Killing vector fields.

\smallskip
As a direct consequence of the linearity of equation \eqref{eq.A}, for any $A\in\mathcal{A}(g,\T)$ and any $t\in\R$, the tensor $ A - t\,\Id_{2n}$ also belongs to $\mathcal{A}(g,\T)$.  Moreover, the properties of the determinant of $A$ discussed in Remark \ref{rem:sqrtA} imply that $\det(A - t\,\Id_{2n})$ is a perfect square, so the function $\sqrt{\det(A - t\,\Id_{2n})}$ is a polynomial in $t$.  Expanding it as
\begin{equation}\label{eq:mu}
\sqrt{\det(A-t\,\Id_{2n})}
= \sum_{i=0}^{n} (-1)^{n-i}\,\mu_i\, t^{n-i},
\end{equation}
we obtain a distinguished family of smooth functions $\mu_i\in C^\infty(M)$ naturally associated with $A$. Moreover, we define the following  gradient vector fields
\begin{equation}\label{eq:vi}
V_i:=\grad \mu_i,\qquad\text{for } i=1,\mathellipsis,n.
\end{equation}

\begin{cor}\label{cor:mui.killing}
Let $(M,g,\T)$ be a \pk manifold and let $A\in\mathcal{A}(g,\T)$. Then, for each $i$, the vector field $\T V_i$ is Killing with respect to $g$.
\end{cor}
\begin{proof}
Since $A-t\,\Id_{2n}\in\mathcal{A}(g,\T)$ for all $t\in\R$, Proposition \ref{prop:det.killing} implies that
$$X(t):=\T\grad \sqrt{\det(A-t\,\Id_{2n})}$$
is a Killing vector field of $g$ for every $t$. Hence $\mathcal{L}_{X(t)}g=0$.  

From \eqref{eq:mu}, each coefficient $\mu_i$ can be written as
$$\mu_i=\frac{1}{(n-i)!}\,\frac{d^{\,n-i}}{dt^{\,n-i}}\Big|_{t=0}\sqrt{\det(A-t\,\Id_{2n})}.$$
Since the map $F:\R\times M\rightarrow\R$, defined as  $F(t,p)=\sqrt{\det(A(p)-t\,\Id_{2n})}$,
is smooth, differentiation with respect to $t$ commutes with all differential operators acting on $p$, such as the gradient and the Lie derivative.  Therefore, being $\T$ independent of t, we have
$$
\T V_i=\T\grad \mu_i=\frac{1}{(n-i)!}\,\frac{d^{\,n-i}}{dt^{\,n-i}}\Big|_{t=0}X(t),
$$
and consequently
$$
\mathcal{L}_{\frac{d^{\,n-i}}{dt^{\,n-i}}\big|_{t=0}X(t)}g
=\tfrac{d^{\,n-i}}{dt^{\,n-i}}\big|_{t=0}\,\mathcal{L}_{X(t)}g
=0.$$
It follows that $\T V_i$ is a Killing vector field.
\end{proof}
In particular, to each Benenti tensor of $(g,\T)$ one can canonically associate a family of Killing vector fields of $g$.  As shown in the following proposition, these vector fields admit a symplectic interpretation and turn out to be para-holomorphic. This latter property will be repeatedly exploited in constructing the coordinate vector fields for the degenerate cases of Theorem \ref{thm:main}, see Sections \ref{sec.rank.3}-\ref{sec.rank.1}.

\begin{prop}\label{prop:hamilt}
The vector fields $\T V_i$ and $V_i$, where $V_i$ given by \eqref{eq:vi}, are para-holomorphic for any $i$.
\end{prop}
\begin{proof}
The vector fields $\T V_i$ considered in Corollary \ref{cor:mui.killing} are in fact the Hamiltonian vector fields, with respect to the \pk form
$\omega$ given by \eqref{eq:omega}, associated with the functions $\mu_i$,
that is,
$$
\omega(\T V_i,\cdot)=d\mu_i.$$
Hamiltonian vector fields preserve $\omega$, namely
$$\mathcal{L}_{\T V_i}\omega=0,$$
so they are \emph{symplectic}.
By Corollary \ref{cor:mui.killing}, the same vector fields are also Killing with respect to $g$, that is, $\mathcal{L}_{\T V_i}g=0$.
Since $g$ and $\omega$ determine the endomorphism $\T$ by the relation \eqref{eq:omega}, it follows that
$$
g\bigl((\mathcal{L}_{\T V_i}\T)(\cdot),\cdot\bigr) = (\mathcal{L}_{\T V_i}\omega)(\cdot,\cdot)
= 0,$$
hence $\mathcal{L}_{\T V_i}\T=0$.
In other words, the vector fields  $\T V_i$ are \emph{para-holomorphic}. Moreover, the vanishing of the Nijenhuis tensor of $\T$ implies
 $$
\mathcal{L}_{V_i}\T = \T\,\mathcal{L}_{\T V_i}\T,
 $$
that is  each $V_i$ is also para-holomorphic.  
\end{proof}

\section{Proof of Theorem \ref{thm:main.einstein}}\label{sec:pcEinstein}

In order to prove Theorem \ref{thm:main.einstein}, we need some ancillary technical lemmas and a corollary, that we establish below. More precisely, starting from equation \eqref{eq:main}, our first target is to prove, by using Lemma \ref{lemma.Filippo.Einstein.1} and Corollary \ref{cor.Filippo.Einstein.1},  that the Ricci tensor of a metric $\tilde{g}$ of the form \eqref{eq:lin.comb.metrics} is proportional to $\tilde{g}$ by a factor $\tilde{\lambda}$ (Lemma \ref{lemma:tilde.lambda}), provided that both $g$ and $\hat{g}$ are Einstein metrics . Finally, we prove that such factor $\tilde{\lambda}$ is actually a real constant.

\medskip 
In view of \eqref{eq:Psi.exact}, we have that $\Psi=-\tfrac14\, d\,\log\det A$. Also, we recall that $\Lambda=\tfrac{1}{4}\grad\tr A$, see \eqref{tildeL}.
Since, for any vector field $X$, we have that
\begin{equation}\label{eq:dtrA.dlogdetA}
(d \tr A)(X)= \tr (\nabla_X^g A)\overset{\eqref{eq.A}}{=}4\,g(\Lambda,X)\,, \quad
(d\log\det A)(X)= \tr  \big(A^{-1}\nabla_X^g A\big)\overset{\eqref{eq.A}}{=}4\,g(A^{-1}\Lambda,X),
\end{equation}
we obtain the following equality, that will be employed several times in this section:
\begin{equation}\label{eq:PsiLambda}
\Psi(X)=-g(\Lambda, A^{-1}X).
\end{equation}
\begin{lemma}\label{lemma.Filippo.Einstein.1}
Let $g$ and $\hat{g}$ be pc-projectively equivalent metrics and $A$ given by \eqref{A.def}.
Then, from \eqref{eq:main}, we obtain
\begin{equation}\label{main.new}
\tfrac{1}{2(n-1)}\Big(\widehat\Ric-\Ric\Big)(X,Y)
=
 g(A^{-1}Y, \nabla^g_X \Lambda)- g(A^{-1} \Lambda,  \Lambda )\, g(Y,   A^{-1}X).
\end{equation}
\end{lemma}
\begin{proof}
Taking into account that
\begin{multline*}
(\nabla^g_X A^{-1})(Y)=-A^{-1}( \nabla^g_X A) (A^{-1}Y) \overset{\eqref{eq.A}}{=}\\
= -(A^{-1}X)^\flat(Y)\otimes A^{-1} \Lambda -(A^{-1} \Lambda)^\flat(Y) \otimes A^{-1}X
+(\T A^{-1}X)^\flat(Y)\otimes \T A^{-1} \Lambda 
+(\T A^{-1} \Lambda)^\flat (Y)\otimes \T A^{-1}X,
\end{multline*}
we have that
\begin{multline*}
(\nabla^g\Psi)(X,Y)= X\big( \Psi(Y)\big)-\Psi \big(\nabla_X^g Y\big)\overset{\eqref{eq:PsiLambda}}{=}\\
-X\big( g(\Lambda, A^{-1}Y)\big) +g \big(\Lambda, A^{-1}\nabla_X^g Y\big)=
-g\big(\nabla_X^g \Lambda, A^{-1}Y\big)-g\big( \Lambda, (\nabla_X^gA^{-1})Y\big)=\\
-g\big(\nabla_X^g \Lambda, A^{-1}Y\big)+ (A^{-1}X)^\flat(Y)\,g\big( \Lambda, A^{-1} \Lambda\big) +(A^{-1} \Lambda)^\flat(Y)\, g\big( \Lambda,A^{-1}X\big)\\
-(\T A^{-1}X)^\flat(Y)\, g\big( \Lambda,  \T A^{-1} \Lambda \big)-
(\T A^{-1} \Lambda)^\flat (Y)\, g \big( \Lambda, \T A^{-1}X\big)\,,
\end{multline*}
that allows to write \eqref{eq:main} as follows:
\begin{multline}\label{eq:ricci.filippo.quasi}
\tfrac{1}{2(n-1)}\Big(\widehat\Ric-\Ric\Big)(X,Y)
=
 g(A^{-1}Y, \nabla^g_X \Lambda)-g(A^{-1} \Lambda,  \Lambda )\, g(Y,   A^{-1}X)- g(\T A^{-1}  \Lambda, \Lambda )\, g(\T X,  A^{-1}Y).
\end{multline} 
Recalling that $[A^{-1},\T]=0$, that $A$ is $g$-symmetric and that the \pk metrics  are also para-Hermitian, then
\begin{equation}\label{gtALambda}
g(\T A^{-1}  \Lambda, \Lambda )=-g( A^{-1}  \Lambda, \T \Lambda )=-g(  \Lambda,  A^{-1}\T \Lambda )=-g(  \Lambda,  \T A^{-1}\Lambda )\quad \Rightarrow\quad g(\T A^{-1}  \Lambda, \Lambda )=0\,,
\end{equation}
that, substituted in \eqref{eq:ricci.filippo.quasi}, gives \eqref{main.new}.
\end{proof}
\begin{cor}\label{cor.Filippo.Einstein.1}
Let $g$ and $\hat{g}$ be pc-projectively equivalent and $A$ given by \eqref{A.def}. If $g$ and $\hat{g}$ are Einstein metrics with Einstein constants, respectively, equal to $\lambda$ and $\hat{\lambda}$, then
\begin{equation}\label{eq:nabla.Lambda}
\nabla^g\Lambda
=\left(\frac{\hat\lambda}{2(n-1)}(\det A)^{-1/2}+g(A^{-1}\Lambda,\Lambda)\right)\Id
-\frac{\lambda}{2(n-1)}\,A.
\end{equation} 
\end{cor}
\begin{proof}
The claim follows directly from formula \eqref{main.new} by considering that both $g$ and $\hat g$ are Einstein metrics.
\end{proof}

\bigskip\noindent
Let $g$ and $\hat{g}$ be pc-projectively equivalent metrics. Let $\sigma$ and $\hat{\sigma}$, respectively,  be the associated weighted tensors via \eqref{g.sigma}. Recalling that $A=\hat \sigma \sigma^{-1}$, see \eqref{eq:Adef}, we have that
\begin{equation}\label{eq:def.A.tilde}
\tilde A_{[\alpha,\beta]}:=(\alpha \hat \sigma + \beta \sigma)\sigma^{-1}=\alpha \, A+\beta\,  \mathrm{Id}\,,\quad \alpha,\beta\in\R\,.
\end{equation}
All the values of $(\alpha,\beta)$ for which $\tilde A_{[\alpha,\beta]}$ is non-degenerate define a Benenti tensor of $g$.  In particular, for such parameters, the corresponding \pk metrics
 (cf. \eqref{eq: companion.A})
\begin{equation*}
\tilde g_{[\alpha,\beta]}=\left(\det \tilde A_{[\alpha,\beta]}\right)^{-\tfrac12}\,g \,\tilde A^{-1}_{ [\alpha,\beta]} 
\end{equation*}
form a smooth $2$-parametric family of metrics pc-projectively equivalent to $g$ (and hence also to $\hat g$).
Taking the above discussion into account, we have the following lemma.
\begin{lemma}\label{lemma:tilde.lambda}
Let $g$ and $\hat{g}$ be pc-projectively equivalent Einstein metrics with Einstein constants equal to $\lambda$ and $\tilde{\lambda}$ respectively.
Let us denote $\tilde{g}=\tilde{g}_{[\alpha,\beta]}$ and $\tilde{A}=\tilde{A}_{[\alpha,\beta]}$. Then we have that
$\widetilde\Ric=\tilde\lambda\,\tilde g$,
where
\begin{equation}\label{eq:def.tilde.lambda}
\tilde \lambda=\tilde{\lambda}_{[\alpha,\beta]}
:=2(n-1)(\det\tilde A)^{1/2}\!
\left(
\frac{\hat\lambda\alpha}{2(n-1)}(\det A)^{-1/2}
+\alpha\,g(A^{-1}\Lambda,\Lambda)
-\alpha^2\,g(\tilde A^{-1}\Lambda,\Lambda)
+\frac{\lambda\beta}{2(n-1)}
\right)\,.
\end{equation}
\end{lemma}
\begin{proof}
Let
$$\tilde\Lambda:=\frac14  \grad_g \tr \tilde A=\frac 14  \grad_g (\alpha \tr A +2 \beta n)=\alpha\Lambda\,.$$ 
Then, we have that
\begin{multline}\label{eq:nablatildeL}
\nabla^g \tilde \Lambda=\alpha\, \nabla^g \Lambda \overset{\eqref{eq:nabla.Lambda}}{=}
\alpha\left(\tfrac{\hat\lambda}{2(n-1)} (\det A)^{-1/2} +g(A^{-1} \Lambda,  \Lambda ) \right) \mathrm{Id}    -\tfrac{\lambda}{2(n-1)} \, \alpha A=\\
 \alpha\left(\tfrac{\hat\lambda}{2(n-1)} (\det A)^{-1/2} +g(A^{-1} \Lambda,  \Lambda ) +\tfrac{\lambda\beta }{2\alpha (n-1)} \right) \mathrm{Id}    -\tfrac{\lambda}{2(n-1)} \, \tilde A.
\end{multline}
Therefore, we obtain the following relation between  $\Ric$ and $\widetilde \Ric$:
\begin{multline*}
\tfrac{2}{n-1}\Big(\widetilde\Ric-\Ric\Big)(X,Y)\overset{\eqref{main.new}}{=}
4\, g(\tilde A^{-1} \nabla^g_X \tilde \Lambda, Y)-4\, g(\tilde A^{-1} \tilde \Lambda,  \tilde \Lambda )\, g( \tilde A^{-1}X, Y)\overset{\eqref{eq:nablatildeL}}{=}\\
4\, g\left( \alpha\left(\tfrac{\hat\lambda}{2(n-1)} (\det A)^{-1/2} +g(A^{-1} \Lambda,  \Lambda ) +\tfrac{\lambda\beta }{2\alpha (n-1)} \right) \tilde A^{-1}X    -\tfrac{\lambda}{2(n-1)} \, X,\,  Y\right)
-4\, g(\tilde A^{-1} \tilde \Lambda,  \tilde \Lambda )\, g( \tilde A^{-1}X, Y)
\end{multline*}
Since we are assuming $\Ric=\lambda g$, we get 
\begin{equation*}
\frac{1}{2(n-1)}\widetilde\Ric
=(\det\tilde A)^{1/2}\!\left(\frac{\hat\lambda\alpha}{2(n-1)}(\det A)^{-1/2}+\alpha\,g(A^{-1}\Lambda-\tilde A^{-1}\tilde\Lambda,\Lambda)+\frac{\lambda\beta}{2(n-1)}\right)\tilde g,
\end{equation*}
proving the claim.
\end{proof}

\subsection*{Proof of Theorem \ref{thm:main.einstein}}

In view of Lemma \ref{lemma:tilde.lambda}, the theorem is proved if we show that $\tilde{\lambda}_{[\alpha,\beta]}$ defined in \eqref{eq:def.tilde.lambda} are actually constants for each $\alpha$ and $\beta$, implying that $\tilde{g}_{[\alpha,\beta]}$ are Einstein metrics with Einstein constants equal to $\tilde{\lambda}_{[\alpha,\beta]}$. To do this, it will be sufficient to show that $d\tilde{\lambda}=0$.
In order to prove it, we need some equalities, established below:
\begin{multline}\label{i}
\bullet \Big(d\,(\det\tilde A)^{1/2}\Big)(X)=\tfrac12(\det\tilde A)^{1/2}d(\log\det\tilde A)\,(X)= \tfrac12(\det\tilde A)^{1/2}\,\tr(\tilde A^{-1}\,\nabla^g_X\tilde A)\overset{\eqref{eq:def.A.tilde}}{=}\\
\tfrac12(\det\tilde A)^{1/2}\alpha\,\tr(\tilde A^{-1}\,\nabla^g_X A)\overset{\eqref{eq.A}}{=}
2\alpha\,(\det\tilde A)^{1/2}\,g(\tilde A^{-1}\Lambda,X)\,,
\end{multline}
\begin{multline}
\bullet \Big(d\,(\det A)^{-1/2}\Big)(X)=-\tfrac12(\det A)^{-1/2}d(\log\det A)\,(X)
\overset{\eqref{eq:dtrA.dlogdetA}}{=}-2(\det A)^{-1/2}\,g(A^{-1}\Lambda,X)\,, \hspace{2cm} 
\end{multline}
\begin{multline}
\bullet \Big(d\,\big(g(A^{-1}\Lambda,\Lambda)\big)\Big)\, (X)
\overset{\substack{\mathrm{A^{-1} is} \\ \mathrm{g-symmetric}}}{=} g\big((\nabla^g_X A^{-1})\Lambda,\Lambda\big)+2\,g\big(A^{-1}\Lambda,\nabla^g_X\Lambda\big)\overset{\eqref{eq:nabla.Lambda}}{=}\\
-g\big((\nabla^g_XA)A^{-1}\Lambda,\,A^{-1}\Lambda\big)+2\,\big(\tfrac{\hat\lambda}{2(n-1)}(\det A)^{-1/2}+g(A^{-1}\Lambda,\Lambda)\big)\,g(A^{-1}\Lambda,X)-\tfrac{\lambda}{n-1}\,g(\Lambda,X)\overset{\eqref{eq.A}+\eqref{gtALambda}}{=}\\
\frac{\hat\lambda}{\,n-1\,}(\det A)^{-1/2}\,g(A^{-1}\Lambda,X)-\frac{\lambda}{\,n-1\,}\,g(\Lambda,X)\,,
\end{multline}
\begin{multline}\label{iv}
\bullet \Big(d\,\big(g(\tilde A^{-1}\Lambda,\Lambda)\big)\Big)\, (X)
\overset{\substack{\mathrm{\tilde A^{-1} is}  \\ \mathrm{g-symmetric}}}{=} g\big((\nabla^g_X \tilde A^{-1})\Lambda,\Lambda\big)+2\,g\big(\tilde A^{-1}\Lambda,\nabla^g_X\Lambda\big)\overset{(*)}{=}\\
-2\alpha\,g(\Lambda,\tilde A^{-1}\Lambda)\,g(\tilde A^{-1}\Lambda,X)
+2\Big(\tfrac{\hat\lambda}{2(n-1)}(\det A)^{-1/2}+g(A^{-1}\Lambda,\Lambda)\Big)\,g(\tilde A^{-1}\Lambda,X)
-\frac{\lambda}{\,n-1\,}\,g(A\tilde A^{-1}\Lambda,X)\,.
\end{multline}
$(*)$ In this step we used $\nabla^g\tilde A^{-1}=-\alpha\, \tilde A^{-1}(\nabla^gA)\tilde A^{-1}$.

\bigskip\noindent
Finally, on account of \eqref{i}-\eqref{iv}, in view of \eqref{eq:def.A.tilde},
we can easily compute
\begin{multline*}
\tfrac{1}{2(n-1)}d \tilde \lambda\,(X)=
\frac{\lambda\alpha}{n-1}\, (\det\tilde A)^{1/2}
\left(\beta
\,g(\tilde A^{-1}\Lambda,X)-g(\Lambda,X)
+\alpha\,g(A\tilde A^{-1}\Lambda,X)  \right)=\\
\frac{\lambda\alpha}{n-1}\, (\det\tilde A)^{1/2}
\,g\left(\big((\alpha\,  A+  \beta\,  \Id) \tilde A^{-1}-\Id\big)\Lambda, X\right)=0.
\end{multline*}
implying $d\tilde\lambda=0$, as we wanted.

\section{Preliminary results needed for the proof of Theorem \ref{thm:main}}\label{sec:preliminary}

In the previous sections we obtained results valid in arbitrary dimension. From now on, we restrict our attention to \emph{\pk surfaces}, i.e., \pk manifolds of para-complex dimension $n$ equal to $2$.

\smallskip
As stated in Theorem \ref{thm:main}, our local classification holds  ``in a neighborhood of almost every point of $M$''; concretely, this means ``on the dense open subset $\widehat M^{\circ}\subseteq M$'' constructed in Section \ref{sec:gradfields}. In Section \ref{sec:partition}, we prepare this construction by defining the preliminary dense open set $M^{\circ}\subseteq M$. Furthermore, Section \ref{sec:gradfields}  clarifies the geometric criteria underlying the distinction among the various normal forms appearing in Theorem \ref{thm:main} (cf. also Table \ref{tab:ranks}).

\smallskip
While Theorem \ref{thm:main} concerns  Benenti tensors, the arguments developed below apply to any element of $\mathcal{A}(g,\T)$. 

\subsection{Eigenvalue partition of \pk surfaces: construction of an open dense subset}\label{sec:partition}

Let $(M,g,\T)$ be a \pk surface and let $A\in \mathcal{A}(g,\T)$.  In view of Remark \ref{rem:sqrtA}, the $(1,1)$ tensor field $A$ has at most two distinct eigenvalues. 
%
\begin{center}
\textbf{\underline{Notation}:} From now on,
we denote the eigenvalues of $A$ by $\rho$ and $\sigma$, unless otherwise specified\,.
\end{center}
%
The eigenvalue structure of $A$ induces a natural partition of $M$ according to the spectral type of $A$ at each point $p\in M$. More precisely, we define the following  three subsets of $M$: 
$$
\begin{aligned}
\bullet\quad p &\in M^{\mathrm{r}}  &&\text{if the eigenvalues $\rho(p)$ and $\sigma(p)$ of $A(p)$ are real and distinct, each of algebraic multiplicity $2$,} \\
\bullet\quad p &\in M^{\mathrm{c}}  &&\text{if the eigenvalues $\rho(p)$ and $\sigma(p)$ of $A(p)$ are complex-conjugate, each of algebraic multiplicity $2$,} \\
\bullet\quad p &\in M^{\mathrm{sing}} &&\text{if $A(p)$ has a single real eigenvalue of algebraic multiplicity $4$.}
\end{aligned}
$$

\smallskip\noindent
It follows that $M^{\mathrm{r}}$ and $M^{\mathrm{c}}$ are open subsets, whereas $M^{\mathrm{sing}}$ is closed.  
Indeed,
$$
M^{\mathrm{r}} = \Delta^{-1}\bigl((0,\infty)\bigr), \qquad 
M^{\mathrm{c}} = \Delta^{-1}\bigl((-\infty,0)\bigr), \qquad
M^{\mathrm{sing}} = \Delta^{-1}(0),
$$
where 
$$
\Delta := \mu_1^2 - 4\mu_2
$$
is the discriminant of the characteristic polynomial of $A$, and the functions $\mu_i$ are defined by \eqref{eq:mu}.
Thus, the eigenvalues $\rho$ and $\sigma$ are given by the functions
$$
\frac{-\mu_1\pm\sqrt{\Delta}}{2},
$$
that turn out to be continuous on $M$, but smooth only on $M^{\mathrm{r}}\cup M^{\mathrm{c}}$. 
Then, if neither $\rho$ nor $\sigma$ is constant on the whole of $M$, we introduce the open subset
$$
M^\circ := \left\{ p \in M \;\middle|\; \rho(p) \neq \sigma(p),\ d\rho(p) \neq 0,\ d\sigma(p) \neq 0 \right\}.
$$
In the special case in which $\sigma$ is constant on $M$, we instead set 
$$
M^\circ := \left\{ p \in M \;\middle|\; \rho(p) \neq \sigma,\ d\rho(p) \neq 0 \right\}.
$$
In Section \ref{sec:gradfields} we shall further restrict this set to  a (possibly disconnected) open dense subset 
$\widehat M^{\circ}\subseteq M^{\circ}$, 
where the distribution $D$ defined by \eqref{eq:D} has constant rank on each connected component.
Our aim in this section is to show that $M^{\circ}$ is dense in $M$.
\begin{lemma}\label{prop:Msing}
Let $(M,g,\T)$ be a connected \pk surface and let $A\in\mathcal{A}(g,\T)$ be 
non-parallel. Then the open set $M^{\mathrm{r}}\cup M^{\mathrm{c}}$ is dense in $M$.
\end{lemma}
\begin{proof}
Suppose, by contradiction, that $M^{\mathrm{sing}}$ contains a nonempty open subset $U\subseteq M$.
Then, on $U$, the tensor $A$ admits a single real eigenvalue $\rho$ of algebraic multiplicity 4.  
By Proposition \ref{prop:hamilt} and Corollary \ref{cor:mui.killing}, both $\rho$ and $\rho^2$ are Hamiltonian functions for Killing vector fields on $U$.  
In particular, a direct computation shows that $d\rho\circ\T$ and $d(\rho^2)\circ\T$ are constant along any geodesics. This implies that $\rho$ itself is constant on $U$.
Consequently, using the block form of $A$ introduced in Remark \ref{rem:sqrtA},  the Killing vector field $\T\grad \tr A'$ (which coincides with $\tfrac12 \T\grad \tr A$) vanishes on $U$, then it  vanishes on the whole of $M$, being $U$ an open set, by the unique continuation property of Killing fields. 
Equation \eqref{eq.A} then shows that $A$ is parallel, contradicting the hypothesis.  
Therefore, $M^{\mathrm{sing}}$ has empty interior, and $M^{\mathrm{r}}\cup M^{\mathrm{c}}$ is dense in $M$.
\end{proof}

If $A$ is non-parallel, there cannot exist a non-empty open subset on which both eigenvalues of $A$ are constant.  
Indeed, on such an open set, the Killing vector field $\T\grad \tr A$ would vanish identically and, as we said above, it would then vanish on whole of $M$, so that equation \eqref{eq.A} would imply that $A$ is parallel: a contradiction. Therefore, in the remaining part of the paper, we will adopt the following
\begin{center}
\textbf{\underline{Assumption}:} w.l.o.g. we assume that $\rho$ is an eigenvalue of $A$ non-constant on the whole of $M$.  
\end{center}

\smallskip
\begin{prop}\label{prop:Mdense}
Let $(M,g,\T)$ be a para-Kähler surface and let $A\in\mathcal{A}(g,\T)$ be non-parallel.  
Then $M^\circ$ is dense in $M$.
\end{prop}
\begin{proof}
Let $U \subseteq M$ be  an  open set where the differential $d\rho$ vanishes.  

\smallskip
Let firstly assume that $U \cap M^{\mathrm{r}} \neq \emptyset$.  
Consider the Killing vector fields (for each $t$)
$$
\T\grad \mu(t)\,,\qquad \text{where} \quad \mu(t) := t^2 - \mu_1 t + \mu_2.
$$
Since $\rho$ is constant on $U$, then $\mu(\rho)$ vanishes on it. Hence, the corresponding vector field $\T\grad\mu(\rho)$ vanishes identically on $U$.  
Since it is a Killing vector field, it must then vanish everywhere on $M$, as we already said above, in view of the unique continuation property. 
Hence $\rho$ is constant on the whole of $M$, contradicting the above Assumption.

\smallskip

Let now assume $U \cap M^{\mathrm{c}} \neq \emptyset$.  
Being $\rho$  constant on $U$, its complex conjugate eigenvalue $\sigma = \bar{\rho}$ is also constant there. It contradicts the fact that both eigenvalues of a non-parallel tensor field $A\in\mathcal{A}(g,\T)$ cannot be constant on the same open subset of $M$, as we underline after the proof of Lemma \ref{prop:Msing}. 

\smallskip
Since  $M^{\mathrm{r}}\cup M^{\mathrm{c}}$ is dense in $M$ by Lemma \ref{prop:Msing}, then $U=\emptyset$.
\end{proof}
\begin{rem}\label{rem:eigen.cases}
The previous proposition yields the following global picture of the eigenvalues of a non-parallel $A\in\mathcal A(g,\T)$.  
Exactly one of the following alternatives holds:
\begin{enumerate}
\item both eigenvalues are non-constant on $M$

or 

\item exactly one of them is constant on $M$, while the other is non-constant.
\end{enumerate}
In particular, it is impossible for an eigenvalue to be constant merely on a nonempty proper open subset of $M$.
\end{rem}
This completes the construction of the dense open subset $M^{\circ}$.  

\subsection{Para-holomorphic gradient fields and the distribution $D$}\label{sec:gradfields}

In this section we develop the geometric framework needed for the local description contained in Theorem \ref{thm:main}.  Besides the Killing vector fields introduced in Section \ref{sec:killing}, we identify additional vector fields naturally associated with tensor fields in $\mathcal{A}(g,\T)$, which will play a key role in Section \ref{sec:proof} for the construction of convenient coordinates.

\smallskip
With a slight abuse of notation, we use the same symbol $A$ to denote the $\C$-linear extension of a tensor field $A \in \mathcal{A}(g,\T)$ to the complexified tangent bundle $T^\C M$.  
Likewise, the \pk metric $g$, the Levi-Civita connection $\nabla^g$ and the para-complex structure $\T$ will always be understood as $\C$-linearly extended to $T^\C M$ whenever necessary. 

\begin{lemma}\label{lem:grad.eigenvectors}
Let $(M,g,\T)$ be a \pk surface and let $A \in \mathcal{A}(g,\T)$ be non-parallel.  
Then the gradients of  the eigenvalue functions $\rho$ and $\sigma$ of $A$ satisfy 
$$
A\grad \rho = \rho\, \grad \rho, \qquad 
A\grad \sigma = \sigma\, \grad \sigma.
$$
Furthermore, the eigenspaces of $A$ relative to $\rho$ and $\sigma$ are $g$-orthogonal.
\end{lemma}

\begin{proof}
The computation can be carried out on the dense open subset $M^{\mathrm r} \cup M^{\mathrm c}$, where the eigenvalues are smooth and distinct (see Lemma \ref{prop:Msing}).

\smallskip
Let $Y \in \Gamma(T^\C M)$ be an eigenvector of $A$ corresponding to the eigenvalue $\sigma$, i.e., 
\begin{equation}\label{eq:sigma.eigenv.A}
A Y = \sigma Y\,.
\end{equation}
Then,  for any $X \in \Gamma(T^\C M)$, by considering the covariant derivative of  \eqref{eq:sigma.eigenv.A} along $X$, equation \eqref{eq.A} yields
\begin{equation}\label{eq:Asigmanabla}
(A - \sigma\,\mathrm{Id})\, \nabla_X^g Y
= X(\sigma)\, Y - g(Y,X)\,\Lambda - g(Y,\Lambda)\,X
+ g(Y, \T X)\, \T \Lambda + g(Y, \T \Lambda)\, \T X.
\end{equation}
Now, take $X$ to be an eigenvector of $A$ with eigenvalue $\rho$.  
Since $A$ is $g$-symmetric and commutes with $\T$, we  obtain
\begin{equation}\label{eq:gothogonal}
g(X,Y)=0, \qquad g(\T X, Y)=0 \quad \text{on } M^{\mathrm r} \cup M^{\mathrm c}.
\end{equation}
Being $M^{\mathrm r} \cup M^{\mathrm c}$ dense in $M$ (see Lemma \ref{prop:Msing}), identity \eqref{eq:gothogonal} extends to the whole of $M$ by continuity, so that we obtain that the eigenspaces relative to $\rho$ and $\sigma$ are $g$-orthogonal on the whole of $M$.

\smallskip
Furthermore, in view of \eqref{eq:gothogonal}, equation \eqref{eq:Asigmanabla} simplifies to
$$
(A - \sigma\, \mathrm{Id})\, \nabla_X^g Y 
= X(\sigma)\, Y - g(Y,\Lambda)\, X + g(Y, \T\Lambda)\, \T X.
$$
Taking the inner product with $Y$, we obtain
$$
g\bigl((A-\sigma\,\mathrm{Id}) \nabla_X^g Y,\, Y\bigr) = X(\sigma)\, g(Y,Y).
$$
If $Y$ is non-isotropic, the left-hand side vanishes by the $g$-symmetry of $A$, and thus $X(\sigma)=0$.  
Equivalently,
\begin{equation}\label{eq:grad.condition}
g(\grad \sigma, X)=0 \quad \text{for all $X$ in the $\rho$-eigenspace of $A$}.
\end{equation}
From the block form of $A$ described in Remark \ref{rem:sqrtA}, it is clear that each eigenspace of $A$ is generated by one vector in $T^+$ and one in $T^-$.  
Since a \pk metric is non-degenerate on $T^+\times T^-$, its restriction to each eigenspace of $A$ is also non-degenerate.  
Moreover, as the two eigendistributions of $A$ are $g$-orthogonal and complementary, it follows from \eqref{eq:grad.condition} that
$$
A\grad \sigma = \sigma\, \grad \sigma \quad \text{on } M^{\mathrm r} \cup M^{\mathrm c}.
$$

Exchanging the roles of $\rho$ and $\sigma$ yields $A\grad \rho= \rho\, \grad \rho$.  
By the density of $M^{\mathrm r} \cup M^{\mathrm c}$ in $M$, these identities extend to the whole of $M$, completing the proof.

\end{proof}

Taking into account Lemma \ref{lem:grad.eigenvectors}, we now examine the gradient vector fields $V_i$, $i=1,2$, defined by \eqref{eq:vi}. The scalar invariants  $\mu_1$ and $\mu_2$ of $A$  can be expressed in terms of its eigenvalues, that essentially are the trace and the determinant of the $2\times 2$ matrix $A'$ appearing in Remark \ref{rem:sqrtA}:
\begin{equation}\label{eq:mu1.mu2}
\mu_1 = \rho + \sigma, \qquad\qquad \mu_2 = \rho\, \sigma\,.
\end{equation}
Hence, by Remark \ref{rem:eigen.cases} and Lemma \ref{lem:grad.eigenvectors}, the distribution
\begin{equation}\label{eq:mathcal.D}
\mathcal{D} := \operatorname{span}\{V_1,\, V_2\}=
\operatorname{span}\{\mathrm{grad}(\rho+\sigma),\, \mathrm{grad}(\rho\sigma)\}
\end{equation}
has constant rank on the open dense subset $M^\circ\subseteq M^{\mathrm r} \cup M^{\mathrm c}$, defined in Section \ref{sec:partition}.  Therefore, the distributions $\mathcal{D}$ and $\T\mathcal{D}$ have constant and equal rank on $M^\circ$, equal either to $2$ or $1$, depending on whether one of the eigenvalues of $A$ is constant. Moreover, the distribution $D$ defined by \eqref{eq:D} is given by
\begin{equation}\label{eq:D.2}
D=\mathcal{D}+\T\mathcal{D}\,.
\end{equation}
These distributions will be fundamental in describing and distinguishing local forms of $(g,\T,A)$, see Table \ref{tab:ranks}.
\begin{cor}\label{cor:D-orth}
As a direct consequence of Lemma \ref{lem:grad.eigenvectors}, the distributions 
$\mathcal{D}$ and $\T\mathcal{D}$ are $g$-orthogonal.  
\end{cor}


\begin{prop}\label{prop:para-holomorphic}
Let $(M,g,\T)$ be a \pk surface and let $A \in \mathcal{A}(g,\T)$ be non-parallel.  
Then the vector fields
\begin{equation}\label{eq:Vi}
V_1\,, \quad V_2\,, \quad \T V_1\,, \quad \T V_2
\end{equation}
mutually commute. In particular, both distributions $\mathcal{D}$ and $\T\mathcal{D}$ are integrable.
\end{prop}
\begin{proof}
Recalling the symplectic form \eqref{eq:omega},  Corollary \ref{cor:D-orth} and \eqref{eq:mu1.mu2}-\eqref{eq:mathcal.D},
the Poisson bracket of $\mu_1$ and $\mu_2$ satisfies
 $$
\{\mu_1,\mu_2\} = \omega(\T V_1,\T V_2) 
= g(\T^2 V_1,\T V_2) = 0.
 $$
Taking into account that $\T V_1$ and $\T V_2$ are Hamiltonian (see Proposition \ref{prop:hamilt}), they commute since $[\T V_1,\T V_2]$ coincides with the Hamiltonian vector field 
associated with $\{\mu_1,\mu_2\}$.
 
\smallskip
By considering again Proposition \ref{prop:hamilt}, each vector field $\T V_i$ is also para-holomorphic.  
Therefore,
 $$
[\T V_i, V_j]
= \mathcal{L}_{\T V_i}(V_j)
= \mathcal{L}_{\T V_i}(\T^2 V_j)
= (\mathcal{L}_{\T V_i}\T)(\T V_j)
+ \T(\mathcal{L}_{\T V_i}\T V_j)
= 0,
 $$
which shows that $\T V_i$ commutes with $V_j$ for all $i,j$.

\smallskip
Also, being each $V_i$   para-holomorphic, it follows that
 $$
[V_i,V_j]= \mathcal{L}_{V_i} V_j
= \mathcal{L}_{V_i}(\T^2 V_j)
= (\mathcal{L}_{V_i}\T)(\T V_j)
+ \T(\mathcal{L}_{V_i}\T V_j)
= 0.
 $$
Therefore, all the vector fields \eqref{eq:Vi}
mutually commute.
\end{proof}
Considering the explicit expressions of the generators of $\mathcal{D}$, one can readily check, by a straightforward computation, 
that the following proposition holds true.
\begin{prop}\label{prop:intersection}
The dimension of
$\mathcal{D} \cap \T\mathcal{D}$ is greater than $0$ if and only if one of the gradients among
$\grad \rho$ and $\grad \sigma$ is isotropic with respect to $g$.  
\end{prop}
According to Proposition \ref{prop:intersection}, we must further restrict the open dense set $M^{\circ}$.  More precisely, we shall consider the following open subset of $M^\circ$:
%
\begin{equation*}
\widehat M^{\circ}=\{\text{The biggest open subset of $M^{\circ}$ such that on each connected component the rank of \eqref{eq:D.2} is constant} \}\,.
\end{equation*}
It turns out to be dense by a standard semicontinuity argument.

\smallskip\noindent
From now on, we shall replace $M^{\circ}$ by $\widehat M^{\circ}$.

\smallskip\noindent
The possible configurations of the gradients of $\rho$ and $\sigma$ and the corresponding rank of \eqref{eq:D.2}, 
up to exchanging $\rho$ and $\sigma$, are listed in Table  \ref{tab:ranks}.  
Each of these cases will be analyzed separately in Section \ref{sec:proof}.

\begin{table}[h!]
\centering
\renewcommand{\arraystretch}{1.2}
\begin{tabular}{|c|c|c|}
\hline
$\mathrm{rank}\,(D=\mathcal{D}+\T\mathcal{D})$ 
& $\grad \rho$ & $\grad \sigma$ \\ 
\hline
$4$ & non-isotropic  with $\rho$ real  & non-isotropic  with $\sigma$ real  \\
 & non-isotropic  with $\rho$ complex & non-isotropic with $\sigma=\bar \rho$ \\
\hline
$3$ 
& $\in \Gamma(T^+)\setminus\{0\}$& non-isotropic  \\
& $\in \Gamma(T^-)\setminus\{0\}$ & non-isotropic  \\
\hline
$2$ 
& non-isotropic & $0$ (i.e.\ $\sigma$ constant)  \\
& $\in \Gamma(T^+)\setminus\{0\}$ & $\in \Gamma(T^+)\setminus\{0\}$  \\
& $\in \Gamma(T^-)\setminus\{0\}$ & $\in \Gamma(T^-)\setminus\{0\}$  \\
& $\in \Gamma(T^+)\setminus\{0\}$ & $\in \Gamma(T^-)\setminus\{0\}$  \\
\hline
$1$ 
& $\in \Gamma(T^+)\setminus\{0\}$ & $0$ (i.e.\ $\sigma$ constant)  \\
& $\in \Gamma(T^-)\setminus\{0\}$ & $0$ (i.e.\ $\sigma$ constant)  \\
\hline
\end{tabular}
\caption{Possible configurations of $\grad\rho$ and $\grad\sigma$, where $\rho$ and $\sigma$ are the eigenvalues of $A$,
and the corresponding rank of $D=\mathcal{D}+\T\mathcal{D}$ on a connected component of the dense subset $\widehat M^{\circ}\subseteq M$.}
\label{tab:ranks}
\end{table}

\begin{rem}
A natural question, which we shall not address here,  concerns the possible coexistence of different local configurations  within the same \pk surface, namely one may ask which of the cases listed in  Table \ref{tab:ranks} can occur simultaneously on distinct regions of $M$.  Remark \ref{rem:eigen.cases} already shows that certain combinations are excluded; for instance, the rank of $D$ cannot drop  from $4$ to $1$ on the same manifold. 
\end{rem}
\begin{rem}\label{rem:kahler.contrast}
It is worth emphasizing that most of the situations described above are peculiar of the \pk setting and have no analogue in the classical Kähler case (cf.  \cite{bmmr}).  
Indeed, for a \K surface $(M,g,J)$ and for each non-parallel $A \in \mathcal{A}(g,J)$, two structural features prevent such phenomena from occurring.

\smallskip\noindent
Firstly, since $J^2=-\mathrm{Id}$, the distribution $\mathcal{D}$ is always such that
$$
\mathrm{rank}(\mathcal{D} \cap J\mathcal{D}) = 0
\quad\text{and}\quad
\mathrm{rank}(\mathcal{D} + J\mathcal{D}) = 2\,\mathrm{rank}(\mathcal{D})$$
everywhere on $M$, which in particular excludes the occurrence of odd-dimensional ranks.

\smallskip\noindent
Secondly, in the \K setting, the metric $g$ restricts to each eigenspace of $A$ as a definite form (either positive or negative), so the gradients of the eigenvalues of $A$ can never be isotropic.  
By contrast, in the \pk case, as a direct consequence of the block form of $A$ described in Remark \ref{rem:sqrtA}, the restriction of $g$ to each eigenspace of $A$ has signature $(1,1)$, which allows the gradients of the eigenvalues to become isotropic, giving rise to the additional degeneracies discussed above and listed in Table \ref{tab:ranks}.
\end{rem}

\section{Proof of Theorem \ref{thm:main}}\label{sec:proof}
In view of Proposition \ref{prop:para-holomorphic}, on each connected component of $\widehat M^{\circ}$ where $D=\mathcal{D}+\T\mathcal{D}$ (cf. \eqref{eq:mathcal.D}--\eqref{eq:D.2}) has maximal rank, that is, in the non-degenerate case of Theorem  \ref{thm:main}, the four para-holomorphic vector fields \eqref{eq:Vi} form a local frame. 
We shall treat this situation in Section \ref{sec.proof.non.deg}.
In all the other (degenerate) cases, treated in Sections \ref{sec.rank.3}--\ref{sec.rank.1},  the construction of a local distinguished coordinate frame on  a (sufficiently small) neighborhood of any point of the connected component of  $\widehat M^{\circ}$ under consideration, will be carried out separately, according to the rank and to the isotropy properties of $\mathcal{D}$ and $\T\mathcal{D}$. 

\smallskip
Indeed, each subsection corresponds to a fixed value of  $\mathrm{rank}(D)$, whereas each subsubsection treats one of the specific configurations listed in Table \ref{tab:ranks}.


\smallskip
Some of the algebraic computations required in the proof are rather involved.   
For this reason, only the essential geometric steps and the resulting expressions are reported here.

From now on, since Theorem \ref{thm:main} involves Benenti tensors, we shall assume, unless otherwise specificed, that $A$ is a Benenti tensor, even if part of the results of the present section holds also for degenerate tensors.

\subsection{Case $\dim D=4$: the non-degenerate situation}\label{sec.proof.non.deg}

In this situation, as we said a bit above, the vector fields \eqref{eq:Vi} form a local frame, and
$$
D=\operatorname{span}\{V_1,V_2,\T V_1, \T V_2\}\,.
$$
In this case, the normal forms contained in Theorem \ref{thm:main} can be obtained by considering the normal forms of the pair of real non-proportional projectively equivalent metrics on surfaces, described   in \cite[Appendix]{m_alone}. 
Below  we briefly explain how it can be done.

\smallskip\noindent
The leaves of the distribution $\mathcal{D}$ (recall that such distribution is integrable, see Proposition \ref{prop:para-holomorphic}) are totally geodesic
submanifolds of $(M,g)$, as we establish in the next lemma. Then, if two \pk metrics are pc-projectively equivalent, their restrictions to each leaf of $\mathcal{D}$ are (real) projectively equivalent. 
Thus, the normal forms of such restrictions can be easily obtained by those contained   in \cite[Appendix]{m_alone}. 

\begin{lemma}\label{lem:D-tot-geod}
Let $(M,g,\T)$ be a \pk surface and let $A\in\mathcal{A}(g,\T)$ be a non-parallel Benenti tensor. On any open subset of $M$ where $\mathrm{rank}(\mathcal{D}\oplus\T\mathcal{D})=4$, 
the leaves of $\mathcal{D}$ are totally geodesic submanifolds of $(M,g)$.
\end{lemma}
\begin{proof}
By Corollary \ref{cor:D-orth}, we have 
$$
g(V_j,\T V_h)=0\,.  
$$
Covariantly differentiating the above equation along $V_i$ gives
\begin{equation}\label{eq:KD1}
g(\nabla^g_{V_i} V_j, \T V_h) + g(V_j, \nabla^g_{V_i} \T V_h)=0,
\end{equation}
Since each $\T V_h$ is a Killing vector field (see Corollary \ref{cor:mui.killing}), we have
$$
g(V_j, \nabla^g_{V_i} \T V_h) + g(V_i, \nabla^g_{V_j} \T V_h)=0,
$$
that, substituted into \eqref{eq:KD1}, yields
$$g(\nabla^g_{V_i} V_j, \T V_h) - g(V_i, \nabla^g_{V_j} \T V_h)=0\,.
$$
Since the vector fields $V_i$ commute each other (see Proposition \ref{prop:para-holomorphic}),
we may interchange $i$ and $j$, obtaining
\begin{equation}\label{eq:KD2}
g(\nabla^g_{V_j} V_i, \T V_h) - g(V_i, \nabla^g_{V_j} \T V_h)=0\,.
\end{equation}
On the other hand, by differentiating once more the orthogonality relation 
$g(V_i,\T V_h)=0$ along $V_j$, one obtains
\begin{equation}\label{eq:KD3}
g(\nabla^g_{V_j} V_i, \T V_h) + g(V_i, \nabla^g_{V_j} \T V_h)=0\,.
\end{equation}
By summing \eqref{eq:KD2} and \eqref{eq:KD3}, we get
\begin{equation}\label{eq:Dgeodesic}
g(\nabla^g_{V_j} V_i, \T V_h)=0\,. 
\end{equation}
Since the distribution $\T\mathcal{D}$ is the $g$-orthogonal complement of $\mathcal{D}$ (see Corollary \ref{cor:D-orth}), \eqref{eq:Dgeodesic} implies that $\nabla^g_{V_j} V_i$ has no component along $\T\mathcal{D}$, i.e.\ $\nabla^g_{V_j} V_i \in \Gamma(\mathcal{D})$ for all $i,j$.  Thus, the leaves of $\mathcal{D}$ are totally geodesic.  
\end{proof}
Let $L$ be an arbitrary leaf of the distribution $\mathcal{D}$.  By Lemma \ref{lem:D-tot-geod}, it is a totally geodesic surface of $(M,g)$. We denote by $g_L$ the metric on $L$ induced by $g$. The Benenti tensor $A$ preserves the tangent bundle of $L$, namely the distribution $\mathcal{D}$.  Indeed, since $\grad\rho$ and $\grad\sigma$ are eigenvectors of $A$ corresponding to the eigenvalues $\rho$ and $\sigma$, respectively (see Lemma \ref{lem:grad.eigenvectors}), one readily computes
\begin{equation}\label{eq:AV1AV2}
A V_1 = (\rho+\sigma)\,V_1 - V_2, \qquad 
A V_2 = \rho\sigma\,V_1.
\end{equation}
Thus, $A$ induces a well-defined Benenti tensor $A_L$ on $TL$. Moreover, since each leaf $L$ is totally geodesic and $\mathcal{D}$ is $g$-orthogonal to $\T\mathcal{D}$ (see Corollary \ref{cor:D-orth}), it follows from \eqref{eq.A} that the restriction $A_L$ of $A$ to $TL$ satisfies
\begin{equation}\label{eq:realProj}
\nabla^{g_L}_X A_L = \tfrac{1}{2}\big(  X^\flat \otimes \grad_{g_L}\tr A_L 
  + d\tr A_L \otimes X  \big),
\qquad \forall\, X \in \Gamma(TL),
\end{equation}
where the musical isomorphism is taken with respect to $g_L$.  
In particular, each solution $A_L$ of \eqref{eq:realProj} is related to a pseudo-Riemannian metric on $L$ sharing with $g_L$ the same unparameterized geodesics\footnote{We can easily see that, in the case of unparameterized geodesics, namely curves satisfying \eqref{tauplanar} with $\T=0$, the equation \eqref{eq.A} reduces to \eqref{eq:realProj}, taking into account we are dealing with surfaces, i.e. $n=2$.}.

\smallskip
From \eqref{eq:realProj} it is clear that $A_L$ cannot be parallel with respect to $\nabla^{g_L}$, since this would imply $d(\tr A_L)=0$  and consequently $\mathrm{rank}(\mathcal{D}\oplus\T\mathcal{D})<4$.
We can thus invoke the local normal forms for pair of non-proportional projectively equivalent $2$-dimensional pseudo-Riemannian metrics  contained in the Appendix of  \cite{m_alone}, which, in this case, provides the local expressions of $(g_L,A_L)$ on the surface $L$.
To extend this description to the whole four-dimensional manifold $(M,g,\T,A)$, we now introduce local coordinates that are adapted to the  frame $\{V_1,V_2,\T V_1,\T V_2\}$.

\medskip
Since the vector fields $V_1,V_2,\T V_1,\T V_2$ mutually commute (see Proposition \ref{prop:para-holomorphic}) and are linearly independent on an open subset of $M$ where $\mathrm{rank}(\mathcal{D}\oplus\T\mathcal{D})=4$, there exist local coordinates $(y^1,y^2,y^3,y^4)$ such that
$$V_1 = \partial_{y^1}, \qquad 
V_2 = \partial_{y^2}, \qquad 
\T V_1 = \partial_{y^3}, \qquad 
\T V_2 = \partial_{y^4}.$$
From the orthogonal splitting $TM=\mathcal{D}\oplus\T\mathcal{D}$ established in Corollary \ref{cor:D-orth},   from the fact that $g$ is para-Hermitian and $A$ commutes with $\T$, it follows directly that $g$ and $A$ take block-diagonal form
\begin{equation*}
g(y) =
\begin{pmatrix}
g_L(y) & 0\\
0 & -g_L(y)
\end{pmatrix},
\qquad
A(y) =
\begin{pmatrix}
A_L(y)& 0\\
0 & A_L(y)
\end{pmatrix}.
\end{equation*}
The notation $g_L(y)$ and $A_L(y)$ is justified by the fact that they are  the matrices of the components of $g_L$ and $A_L$, respectively, in the coordinates $(y^1,y^2,y^3,y^4)$. Indeed,
the components of $g_L(y)$ depend only on the coordinates $y^1$ and $y^2$, since $\partial_{y^3}$ and $\partial_{y^4}$ are Killing vector fields (see Corollary \ref{cor:mui.killing}) and $L$ is totally geodesic.
%
Moreover, by \eqref{eq:AV1AV2}, we have
\begin{equation}\label{eq:A'}
A_L(y) =
\begin{pmatrix}
\rho+\sigma & \rho\sigma\\
-1 & 0
\end{pmatrix}.
\end{equation}
%
By \cite[Appendix]{m_alone}, the pair $(g_L,A_L)$ is locally diffeomorphic to one of the several normal forms listed there, according to the nature of the eigenvalues of $A_L$.  Since we restrict our analysis to the open dense subset 
$\widehat M^{\circ}\subseteq M$ defined in Section \ref{sec:gradfields}, where the eigenvalues are distinct, only the following two cases can occur:
\begin{itemize}
  \item when the eigenvalues $\rho$ and $\sigma$ are real and distinct, 
  $(g_L,A_L)$ corresponds to a \emph{real Liouville metric};
  \item when the eigenvalues are complex conjugate, 
  $(g_L,A_L)$ corresponds to a \emph{complex Liouville metric}.
\end{itemize}
In the following subsections, we analyze these two cases separately, deriving the corresponding local normal forms of the \pk metrics.

\subsubsection{The case where the eigenvalues $\rho$ and $\sigma$ of $A$ are real: the real Liouville metrics}
In \cite[Appendix]{m_alone}, it is proved that, in the case of real and distinct eigenvalues $\rho$ and $\sigma$ of $A_L$, there exists a change of coordinates $(y^1,y^2)\mapsto(x^1,x^2)$, the latter referred to as \emph{Liouville coordinates}, such that 
\begin{equation}\label{eq:gLiouville}
g_L(x) = (\rho - \sigma)
\begin{pmatrix}
1 & 0 \\
0 & \varepsilon
\end{pmatrix},
\qquad
A_L(x) =
\begin{pmatrix}
\rho & 0 \\
0 & \sigma
\end{pmatrix},
\end{equation}
where $\varepsilon=\pm1$, $\rho=\rho(x^1)$ and $\sigma=\sigma(x^2)$.

\smallskip
In these coordinates, the gradient fields $V_1:=\grad_{g_L}\tr A_L$ and $V_2:=\grad_{g_L}\det A_L$ can be written explicitly as 
\begin{multline*}
\partial_{y^1}=V_1=\grad_{g_L}\tr A_L
=\frac{1}{\rho-\sigma}\big(\rho_{x^1}\,\partial_{x^1}
+\varepsilon\,\sigma_{x^2}\,\partial_{x^2}\big), \\
\partial_{y^2}=V_2=\grad_{g_L}\det A_L
=\frac{1}{\rho-\sigma}\big(\sigma\rho_{x^1}\,\partial_{x^1}
+\varepsilon\,\rho\,\sigma_{x^2}\,\partial_{x^2}\big).
\end{multline*}
Hence, the Jacobian matrix of the coordinate transformation 
$(y^1,y^2)\mapsto(x^1,x^2)$ is
$$
J=\frac{1}{\rho-\sigma}
\begin{pmatrix}
\rho_{x^1} & \sigma\rho_{x^1}\\
\varepsilon\,\sigma_{x^2} & \varepsilon\,\rho\,\sigma_{x^2}
\end{pmatrix}.
$$
Using this expression, one computes $g_L(y)$:
\begin{equation}\label{eq:g'.Liouville}
g_L(y) = J^{t}\, g_L(x)\, J
= \frac{1}{\rho - \sigma}
\begin{pmatrix}
\rho_{x^1}^2 + \varepsilon\,\sigma_{x^2}^2 &
\sigma\rho_{x^1}^2 + \varepsilon\,\rho\,\sigma_{x^2}^2 \\[0.2cm]
\sigma\rho_{x^1}^2 + \varepsilon\,\rho\,\sigma_{x^2}^2 &
\sigma^2\rho_{x^1}^2 + \varepsilon\,\rho^2\,\sigma_{x^2}^2
\end{pmatrix}.
\end{equation}

\smallskip
We now extend the Liouville coordinates $(x^1,x^2)$ to a system of coordinates in the ambient four-dimensional manifold by setting 
$$
x^3 := y^3, \qquad x^4 := y^4.
$$

\smallskip
Combining \eqref{eq:A'}, \eqref{eq:gLiouville}, and \eqref{eq:g'.Liouville}, we obtain, in the  coordinates $(x^1,x^2,x^3,x^4)$, the explicit expressions of the metric $g$ and of the Benenti tensor $A$ of the real Liouville case of Theorem \ref{thm:main}.
Concerning the symplectic form, since the matrix of the components of the para-complex structure $\T$ in the  coordinates $(y^1,y^2,y^3,y^4)$ is
\begin{equation}\label{eq:T.y}
\T(y)=
\begin{pmatrix}
0&0&1&0\\
0&0&0&1\\
1&0&0&0\\
0&1&0&0
\end{pmatrix},
\end{equation}
the matrix $\T(x)$ of its components with respect to the  coordinates $(x^1,x^2,x^3,x^4)$ turns out to be
$$
\T(x)=(J\oplus \mathrm{Id})\;\T(y)\;(J\oplus \mathrm{Id})^{-1}
$$
and the shape of the symplectic form $\omega$ follows straightforwardly from definition \eqref{eq:omega}.

\subsubsection{The case where the eigenvalues $\rho$ and $\sigma$ of $A$ are complex-conjugate: the  complex Liouville metrics }
Set
$$
R=\tfrac12(\rho+\sigma), \qquad I=\tfrac{1}{2i}(\rho-\sigma),
$$
so that $\rho=R+i\,I$ and $\sigma=\bar\rho=R-i\,I$.  
In \cite[Appendix]{m_alone}, it is proved that there exists a change of coordinates $(y^1,y^2)\mapsto(x^1,x^2)$, the latter referred to as \emph{complex Liouville coordinates},  such that
\begin{equation}\label{eq:gAcomplex}
g_L(x)=
\begin{pmatrix}
0 & I\\
I & 0
\end{pmatrix},
\qquad
A_L(x)=
\begin{pmatrix}
R & -\,I\\
I & R
\end{pmatrix},
\end{equation}
where $\rho(z)=R(z)+i\,I(z)$ is a holomorphic function of the complex variable $z=x^1+i\,x^2$.

\smallskip
In these coordinates, the gradient fields $V_1:=\grad_{g_L}\tr A_L$ and $V_2:=\grad_{g_L}\det A_L$  are
\begin{multline*}
\partial_{y^1}=V_1=\grad_{g_L}\tr A_L
=\frac{2}{I}\big(R_{x^2}\,\partial_{x^1}+R_{x^1}\,\partial_{x^2}\big), \\
\partial_{y^2}=V_2=\grad_{g_L}\det A_L
=\frac{2}{I}\Big((R R_{x^2}+I \,I_{x^2})\,\partial_{x^1}+(R R_{x^1}+I\, I_{x^1})\,\partial_{x^2}\Big).
\end{multline*}
Hence, the Jacobian matrix of $(y^1,y^2)  \mapsto (x^1,x^2)$ is
$$
J=\frac{2}{I}
\begin{pmatrix}
R_{x^2} & R R_{x^2}+I \, I_{x^2} \\
R_{x^1} & R R_{x^1}+I \, I_{x^1}
\end{pmatrix}.
$$
A direct computation gives the components of $g_L(y)$ in the $(y^1,y^2)$ coordinates:
\begin{equation}\label{eq:g'.cLiouville}
g_L(y) = J^{t}\, g_L(x)\, J
= \frac{4}{I}
\begin{pmatrix}
2R_{x^1} R_{x^2} & I(R_{x^1} I_{x^2} + R_{x^2} I_{x^1})+2R R_{x^1} R_{x^2} \\
I(R_{x^1} I_{x^2} + R_{x^2} I_{x^1})+2R R_{x^1} R_{x^2} & 2(RR_{x^1} + I \, I_{x^1})(RR_{x^2} + I \, I_{x^2})
\end{pmatrix}.
\end{equation}

\smallskip
As in the previous case, we  extend the complex Liouville coordinates $(x^1,x^2)$ to a system of coordinates in the ambient four-dimensional manifold by setting
$$
x^3:=y^3, \qquad x^4:=y^4.
$$

\smallskip
Using the Cauchy-Riemann equations $R_{x^1}=I_{x^2}$ and $R_{x^2}=-I_{x^1}$, and combining \eqref{eq:A'}, \eqref{eq:gAcomplex}, and \eqref{eq:g'.cLiouville}, one obtains, in the  coordinates $(z=x^1+i x^2,\,x^3,\,x^4)$, the explicit expressions of the metric $g$ and of the Benenti tensor $A$ of the complex Liouville metrics of Theorem \ref{thm:main}. Concerning the symplectic form, recalling that the components of the para-complex structure $\T$ in the coordinates $(y^1,y^2,y^3,y^4)$ are given by \eqref{eq:T.y}, the matrix of its components with respect to  $(x^1,x^2,x^3,x^4)$ is
$$
\T(x)=(J\oplus \mathrm{Id})\;\T(y)\;(J\oplus \mathrm{Id})^{-1}
$$
and, as in real Liouville case, the shape of the symplectic form $\omega$ follows straightforwardly from definition \eqref{eq:omega}.

\subsection{Case $\dim D=3$}\label{sec.rank.3}

\subsubsection{The case where $\grad\rho\in\Gamma(T^+)\setminus\{ 0\}$ and $\grad\sigma$ is non-isotropic}\label{sec.dimD3.primo}

In this situation, 
\begin{equation}\label{eq.v.f.dim.3}
D=\operatorname{span}\{\T V_1\,,\,\,\T V_2\,,\,\,V_1-\T V_1\}\,.
\end{equation}
The choice of the vector fields spanning $D$ is motivated by the fact that the first two are Killing vector fields (see Corollary \ref{cor:mui.killing}) and the third belongs to $\Gamma(T^-)\setminus\{0\}$.
%
%
\begin{lemma}\label{campotrasverso}
There exists a local transverse vector field $Y$ to $D$ and belonging to $\in \Gamma(T^-)\setminus\{ 0\}$ which commutes with the generators of $D$ contained in \eqref{eq.v.f.dim.3}.
\end{lemma}
\begin{proof}
Note that, for each $p\in M$, $\dim (D_p\cap T^-_p)=1$. Thus, there exists  $w\in T^-_p$ such that $w\notin D_p$. Let us then consider a curve $\gamma:(-\varepsilon,\varepsilon)\subseteq\R\to M$ adapted to $w$, i.e., with $\gamma(0)=p$ and $\gamma'(0)=w$, such that, for each $t\in (-\varepsilon,\varepsilon)$, $\gamma'(t)\in T^-_{\gamma(t)}$. Let  us consider the vector field $W(t):=\gamma'(t)\in T^-_{\gamma(t)}$ along $\gamma(t)$.

Let us define the neighborhood $U$ of $p$ as follows:
$$
U = \bigcup_{\,|s_i|<\delta,\;|t|<\varepsilon} (\varphi^1_{s_1} \circ \varphi^2_{s_2} \circ   \varphi^3_{s_3})(\gamma(t)),
$$
where $\varphi^1_{s_1},\varphi^2_{s_2},\varphi^3_{s_3}$ are, respectively, the local flows of the generators of \eqref{eq.v.f.dim.3}, with $\delta$ sufficiently small. 

We can define a vector field $Y$ on $U$ as follows.
If $q\in U$, then $q= \varphi_{s}(\gamma(t))$ for some combination $\varphi_s$ of the local flow $\varphi^i_{s_i}$ and some $t\in (-\varepsilon,\varepsilon)$. Then we define
$$
Y(q) = \varphi_{s*}(W(t)).
$$  
Since $W(t)\in T^-$  and any local flow $\varphi_s$ preserves $T^-$ as the generators of \eqref{eq.v.f.dim.3} are para-holomorphic (see Proposition \ref{prop:hamilt}), then $Y(q)\in T^-_q$. 
By possibly restricting the neighborhood $U$, we obtain that
$Y(q)\not\in D_q$, since $w\not\in D_p$.
It turns out that $Y \in \Gamma(T^-)$ is transverse to $D$.
\end{proof}

\begin{prop}\label{prop:coord.D.3}
There exists a local system of coordinates $(x^1,x^2,x^3,x^4)$ against which the vector fields 
$$
\T V_1\,,\,\,\T V_2\,,\,\,V_1-\T V_1\,,\,\,Y
$$
are the coordinate vector fields.
\end{prop}
\begin{proof}
The vector fields are independent by construction. Also, the first three  mutually commute in view of Proposition \ref{prop:para-holomorphic} whereas $Y$ commutes with the others in view of Lemma \ref{campotrasverso}.
\end{proof}
In the coordinates of Proposition \ref{prop:coord.D.3}, the matrix $(g_{ij})$ of the metric coefficients of $g$ is
\begin{equation}\label{g.components}
g=\left(
\begin{array}{cccc}
 -\|\grad\sigma\|^2 &-\rho\, \|\grad\sigma\|^2  & \|\grad\sigma\|^2 &  \rho_{x^4}+\sigma_{x^4} \\
-\rho\,\|\grad\sigma\|^2 &  -\rho^2\,\|\grad\sigma\|^2  &\rho\, \|\grad\sigma\|^2  &  \rho\, \sigma_{x^4}+\sigma  \rho_{x^4} \\
 \|\grad\sigma\|^2 &\rho\, \|\grad\sigma\|^2  & 0 & 0 \\
\rho_{x^4}+\sigma_{x^4} & \rho\, \sigma_{x^4}+\sigma \rho_{x^4} & 0 & 0 \\
\end{array}
\right).\end{equation}
Recalling that $\partial_i=\T V_i$, $i=1,2$, are Killing vector fields (see Corollary \ref{cor:mui.killing}), the entries of the above matrix are independent of both $x^1$ and $x^2$. It is easy to realize that actually
$$
\rho=\rho(x^3,x^4)\,,\quad \sigma=\sigma(x^3,x^4)\,.
$$
The matrix representing the para-complex structure $\T$, in the above system of local coordinates, is
$$
\T=\left(\begin{array}{cccc}
 1& 0&0&0\\
0& 1&0&0\\
1&\rho&-1&0\\
0&0&0&-1
\end{array}\right)
$$
and a direct computation shows that $\T\grad\rho=\grad\rho$ implies 
$$
\rho=\rho(x^4)\,.
$$ 
By computing $\|\grad\sigma\|^2$ directly from the expression of $g$ in \eqref{g.components}, and taking into account that the same quantity already appears as an entry of the same matrix, consistency requires that
\begin{equation}\label{cases}
\|\grad\sigma\|^2=-\sigma_{x^3}\quad \text{or}\quad \|\grad\sigma\|^2=\sigma_{x^3}.
\end{equation}
Recalling that $A(\grad \rho)=\rho\grad \rho$ and  $A(\grad \sigma)=\sigma\grad \sigma$, the matrix representing $A$ assumes the following form:
$$
A=\left(
\begin{array}{cccc}
\rho+\sigma& \rho\sigma & 0 & 0 \\
 -1 & 0 & 0 & 0 \\
 0 & 0 & \sigma&f\\
 0 & 0 & 0 & \rho \\
\end{array}
\right)\,, \quad f=f(x^1,x^2,x^3,x^4)\,.
$$
Note that $A$ is $\T$-invariant, so that $A(\partial_4)\in \Gamma(T^-)$, then  $A(\partial_4)$ has components only along  $\partial_3$ and $\partial_4$. Moreover,  by taking into account that $A$ is $g$-symmetric, we have that
$$
f=-(\rho-\sigma)\,\frac{\sigma_{x^4}}{\|\grad\sigma\|^2}.
$$
Let us suppose that the first equation of \eqref{cases} holds true. We note that the $(4,1)$-entry of $\nabla^{g}_{\de_4} \T$ vanishes only when $\sigma_{x_3}=0$. However, the first equation  of \eqref{cases} tells us that $\sigma_{x_3}\neq 0$  since we are assuming that $\grad \sigma$  is not isotropic. Therefore  $\nabla^{g} \T\neq 0$, that is not possible for \pk metrics.

\smallskip\noindent
The only possibility is that the second equation of \eqref{cases} holds true.

\smallskip\noindent
Direct computations show that, in this case, equation \eqref{eq.A} is satisfied if and only if 
\begin{equation}\label{rhopde}
\rho_{x^4}=(\rho-\sigma)F\,,\quad F:=-\frac{\left(\sigma_{x^3}\, \sigma_{x^3x^4}-\sigma_{x^4}\, \sigma_{x^3x^3}\right)}{\sigma_{x^3}^2}.
\end{equation}
(Note that $\sigma_{x^3}\neq 0$ and $F\neq 0$ since, on the contrary,  \eqref{g.components} would be degenerate). Then, differentiating \eqref{rhopde} w.r.t. $x^3$ gives
\begin{equation*}
\frac{F_{x^3}}{F} = \frac{\sigma_{x^3}}{\rho-\sigma}, 
\end{equation*}
from which
$$
F= e^{H}\, (\rho-\sigma)\,,\quad H=H(x^4)\,,
$$
that, substituted in \eqref{rhopde}, gives
$$
\rho_{x^4}=e^{H}\,\big(\rho-\sigma\big)^2.
$$
It follows that $\rho-\sigma$, and then $\sigma$, have to be functions only in the variable $x^4$. This is not possible since, as we already said, $\sigma_{x^3}=0$ would imply that the metric \eqref{g.components}  is degenerate.

\medskip\noindent
To conclude, in this case there are no triples $(g,\T,A)$ satisfying the hypotheses of Theorem \ref{thm:main}.

\subsubsection{The case where $\grad\rho\in\Gamma(T^-)$ and $\grad\sigma$ is non-isotropic}

This case can be treated similarly to that of Section \ref{sec.dimD3.primo}. Indeed, in this situation,
\begin{equation}\label{eq.v.f.dim.3.bis}
D=\operatorname{span}\{\T V_1\,,\,\,\T V_2\,,\,\,V_1+\T V_1\}\,,
\end{equation}
the choice of the generators of $D$ being motivated, as in Section \ref{sec.dimD3.primo}, by the fact that the first two are Killing vector fields and the third belongs to $\Gamma(T^+)\setminus\{0\}$. Since the generators of \eqref{eq.v.f.dim.3.bis} are para-holomorphic, we can choose a local vector field $Y\in\Gamma(T^+)\setminus\{ 0\}$ transverse to $D$ which commutes with them (cf. Lemma \ref{campotrasverso}). Once chosen a system of coordinates where $\T V_1\,,\,\,\T V_2\,,\,\,V_1+\T V_1\,,\,\, Y$ are coordinate vector fields, the reasonings and the computations closely follow those of Section \ref{sec.dimD3.primo}; for this reason, for the sake of brevity, we omit them. We arrive to the same conclusion: also in this case, there are no triple $(g,\T,A)$ satisfying the hypotheses of Theorem \ref{thm:main}.

\subsection{Case $\dim D=2$}\label{sec.rank.2}

In what follows we denote by $E_\rho$ and $E_\sigma$, respectively, the eigenspaces of $A$ relative to the eigenvalues $\rho$ and $\sigma$.

\subsubsection{The case where $\grad\rho$ is not isotropic and $\grad\sigma=0$  }

In this case, $\rho$ is a non-constant eigenvalue of $A$ whereas $\sigma=c\in\R\setminus\{0\}$. 

\smallskip\noindent
Therefore, according to 
\eqref{eq:mu1.mu2}--\eqref{eq:D.2}, 
$$
V_1=\grad\rho\,,\quad V_2=c\grad\rho
$$ 
and 
$$
D=\mathrm{span}\{\T V_1,V_1\}\,.
$$
In order to have a more convenient notation, we put
\begin{equation}\label{eq:def.4.v.f}
X_1:=\T V_1\,,\quad X_2:=V_1\,,\quad Y_1\in\Gamma(T^+\cap E_\sigma)\,,\quad Y_2\in\Gamma(T^-\cap E_\sigma)\,
\end{equation}
so that
\begin{equation}\label{eq:splitting}
TM=\mathrm{span}\{X_1\}\oplus \mathrm{span}\{X_2\}\oplus \mathrm{span}\{Y_1\}\oplus \mathrm{span}\{Y_2\}\,.
\end{equation}
at the points of  $\widehat M^\circ$ where the vector fields are defined.
Note that, since $g$ is not degenerate, $g(X_i,X_i)\neq 0$ and $g(Y_1,Y_2)\neq 0$. Furthermore, $g(X_i,Y_j)=g(X_1,X_2)=g(Y_i,Y_i)=0$. Below, starting from \eqref{eq:splitting}, we shall construct a local system of coordinates adapted to the case we are treating in this section. 
As a first step, we analyse the integrability of the two-dimensional distributions defined by vector fields appearing in \eqref{eq:splitting}; the result is summarised in the following proposition.

\begin{prop}\label{prop:distr.integrabili}
Let $X_1,X_2,Y_1,Y_2$ as in \eqref{eq:def.4.v.f}. The distributions  $\mathrm{span}\{X_1,X_2\}$,  $\mathrm{span}\{X_1,Y_1\}$, $\mathrm{span}\{X_1,Y_2\}$, $\mathrm{span}\{X_2,Y_1\}$ and $\mathrm{span}\{X_2,Y_2\}$ are integrable, whereas $\mathrm{span}\{Y_1,Y_2\}$ is not, as it has a non-vanishing component along $X_1$.
\end{prop}
\begin{proof}
The distribution $\mathrm{span}\{X_1,X_2\}$ is integrable in view of Proposition \ref{prop:para-holomorphic}. Below, we shall prove only the integrability of the distribution  $\mathrm{span}\{X_1,Y_2\}$ as the integrability of the others can be attained by adopting similar computations and reasonings, and the non-integrability of $\mathrm{span}\{Y_1,Y_2\}$.

To prove the integrability of  $\mathrm{span}\{X_1,Y_2\}$ we have to consider both the following relation, where $Z$ is a generic vector field,
\begin{equation*}
g(\nabla^{g}_{Z}Y_i, Y_i)=Z\, g(Y_i,Y_i)- g(\nabla^{g}_{Z}Y_i, Y_i)=-g(\nabla^{g}_{Z}Y_i, Y_i)\ \Rightarrow\  g(\nabla^{g}_{Z}Y_i, Y_i)=0
\end{equation*}
and the following one, where we assume that $Z$ is a section of $\mathrm{span}\{X_1,X_2,Y_2\}$: 
\begin{multline}\label{eq:app.gianni}
\rho\; g\left(\nabla^{g}_{Z}X_i, Y_2\right)= \rho\Big( Z \big(g(X_i,Y_2) \big) -g\left(X_i, \nabla^{g}_{Z}Y_2\right)\Big)=-g\left(\rho X_i, \nabla^{g}_{Z}Y_2\right)=\\
-g\left(A X_i, \nabla^{g}_{Z}Y_2\right)=-g\left(X_i,A( \nabla^{g}_{Z}Y_2)\right)=
g\Big(X_i, \left(\nabla^{g}_{Z}A\right)(Y_2)-\nabla^{g}_{Z}( A Y_2) \Big)=\\
g\Big(X_i, \left(Z^\flat\otimes \Lambda+\Lambda^\flat \otimes Z -(\T Z)^\flat\otimes \T \Lambda-(\T\Lambda)^\flat \otimes \T Z\right)(Y_2)-\nabla^{g}_{Z}( \sigma Y_2) \Big)\overset{(*)}{=}\\
g\Big(X_i, \left(Z^\flat\otimes \Lambda-(\T Z)^\flat\otimes \T \Lambda\right)(Y_2)-\sigma\nabla^{g}_{Z} Y_2 \Big)= -g(\sigma\nabla^{g}_{Z} Y_2,X_i) =\sigma\; g\left(\nabla^{g}_{Z}X_i, Y_2\right)\,,
 \end{multline}
where the step $(*)$ holds true as $\Lambda=\tfrac12\grad\rho=\tfrac12 X_2$, see \eqref{tildeL}. Since $\sigma\neq \rho$, then
\begin{equation}\label{nablaxy1}
 g\left(\nabla^{g}_{Z}X_i, Y_2\right)=0\,,
\end{equation}
from which we obtain
\begin{equation*}
 g\left(X_i, \nabla^{g}_{Z}Y_2\right)=Z\, g\left(X_i, Y_2\right)-g\left(\nabla^{g}_{Z}X_i, Y_2\right)=0\,. 
\end{equation*}
Moreover, since $X_1$ is a Killing vector field,
\begin{equation*}
g\left(\nabla^{g}_{Y_2}X_1, X_2\right)=-g\left(\nabla^{g}_{X_2}X_1, Y_2\right)\overset{\eqref{nablaxy1}}{=}0.
\end{equation*}
Then, we conclude that
$$
g([X_1,Y_2],X_2)=g([X_1,Y_2],Y_2)=0\,,
$$
i.e., $[X_1,Y_2]$ is a section of $ \operatorname{span}\{X_1,Y_2\}$ as  $\mathrm{span}\{X_1,Y_2\}=(\mathrm{span}\{X_2,Y_2\})^\bot$.

\smallskip\noindent
The distribution $\mathrm{span}\{Y_1,Y_2\}=(\mathrm{span}\{X_1,X_2\})^\bot$ is not integrable. Indeed, by assuming that $Z$ is a section of $\mathrm{span}\{Y_1,Y_2\}$, with similar computations as those contained in \eqref{eq:app.gianni}, we arrive to
\begin{equation*}
\sigma\; g\left(\nabla^{g}_{Z}Y_i, X_j\right)=
g\Big(Y_i, \left(\Lambda^\flat \otimes Z -(\T\Lambda)^\flat \otimes \T Z\right)(X_j)-\rho\nabla^{g}_{Z} X_j \Big)
 \end{equation*}
from which we obtain 
$$ 
g\left(\nabla^{g}_{Y_2}Y_1, X_i\right)=\frac{g(Y_1,Y_2)}{2(\sigma-\rho)}(X_1+X_2)^\flat\,(X_i),\qquad g\left(\nabla^{g}_{Y_1}Y_2, X_i\right)=\frac{g(Y_1,Y_2)}{2(\sigma-\rho)}(X_2-X_1)^\flat\,(X_i)\,, 
$$
that imply
$$g\left([Y_1,Y_2], X_i\right)=g\left(\nabla^{g}_{Y_1}Y_2-\nabla^{g}_{Y_2}Y_1, X_i\right)= -\frac{g(Y_1,Y_2)}{\sigma-\rho}X_1^\flat\,(X_i)\,. 
$$
Thus, $[Y_1,Y_2]$ has a non-vanishing component along $X_1$, implying that $\mathrm{span}\{Y_1,Y_2\}$ is not integrable.
\end{proof}

\begin{rem}\label{rem:Y1Y2.along.X1}
Taking into account the last part of the above proof, we can conclude that $[Y_1,Y_2]$ has no components along $X_2$. In particular, for some functions $\mu,\nu,$ and $\gamma\neq 0$,
$$
[Y_1,Y_2]=\mu Y_1 + \nu Y_2 + \gamma X_1 \,.
$$
\end{rem}
In order to construct a distinguished system of coordinates by considering the vector fields of Proposition \ref{prop:distr.integrabili}, taking also into account Remark \ref{rem:Y1Y2.along.X1}, we need the following two technical lemmas.
\begin{lemma}\label{lemma:gianni4}
Let $X_1,X_2,Y_1,Y_2$ be independent vector fields on  a manifold $M$ such that
$$
\mathrm{span}\{X_1,X_2\}\,,\quad \mathrm{span}\{X_2,Y_1\}\,, \quad \mathrm{span}\{X_1,Y_2\}\,, \quad
\mathrm{span}\{X_2,Y_2\}
$$ 
are integrable distributions, with $Y_1$ and $Y_2$ as in Remark \ref{rem:Y1Y2.along.X1}.
Then, there exists a non-zero function $f$ on $M$ such that the following vector fields
\begin{equation}\label{eq.3.campi.c}
X_2\,,\quad fX_1+Y_1\,,\quad Y_2
\end{equation}
mutually span $2$-dimensional integrable distributions.
\end{lemma}
\begin{proof}
Since $\mathrm{span}\{X_2,Y_2\}$ is integrable, up to renaming $X_2\to a X_2$,  $Y_2\to b Y_2$ for suitable functions $a,b$, we can assume 
\begin{equation}\label{eq:X2Y2ugualeazero}
[X_2,Y_2]=0\,.
\end{equation}
By hypothesis, there exist functions $\alpha,\beta,\xi,\eta,\delta,\varphi,\mu,\nu$ and $\gamma\neq 0$ such that
\begin{equation}\label{eq.rel.comm.c}
[X_1,X_2]=\alpha X_1+\beta X_2\,,\quad [X_2,Y_1]=\xi X_2+\eta Y_1\,,\quad [X_1,Y_2]=\delta X_1 + \varphi Y_2\,, \quad [Y_1,Y_2] =\mu Y_1 + \nu Y_2 + \gamma X_1 \,.
\end{equation}
The brackets between the first and second and between the second and the third vector field of \eqref{eq.3.campi.c}, on account of \eqref{eq.rel.comm.c},  are equal, respectively, to
\begin{equation}\label{eq.comm.2.3.c} 
\big(X_2(f) -f \alpha\big) X_1 + \big(- f\beta  + \xi\big) X_2+\eta Y_1\,,\quad  \big( - Y_2(f) + f\delta  + \gamma \big)X_1 + \mu Y_1  + (f \varphi  + \nu ) Y_2\,.
\end{equation}
A direct computation shows that the first vector field of \eqref{eq.comm.2.3.c}  lies into $\mathrm{span}\{X_2,fX_1+Y_1\}$ and the second one of \eqref{eq.comm.2.3.c}  lies into $\mathrm{span}\{fX_1+Y_1,Y_2\}$  if and only if the following system
\begin{equation}
\left\{
\begin{array}{l}
X_2(f)=(\alpha+\eta)f
\\
Y_2(f)=(\delta-\mu)f+\gamma
\end{array}
\right.
\end{equation}
has solution. To prove its existence, it suffices to check that the compatibility conditions of the above system are satisfied. In particular,
$$
Y_2(X_2(f))=  Y_2(\alpha+\eta)f+ (\alpha+\eta)(\delta-\mu)f+ (\alpha+\eta)\gamma\,, \quad X_2(Y_2(f)) =  X_2(\delta-\mu)f + (\delta-\mu)(\alpha+\eta)f + X_2(\gamma) 
$$
and, taking the difference, in view of \eqref{eq:X2Y2ugualeazero}, we obtain the following compatibility condition:
\begin{equation}\label{eq:cond.comp.cc}
\big(X_2(\delta)- Y_2(\alpha) -X_2(\mu) -Y_2(\eta)\big )f + X_2(\gamma)    - (\alpha+\eta)\gamma =0\,.
\end{equation}
Condition \eqref{eq:cond.comp.cc} is actually satisfied in view of the Jacobi identity. Indeed, from 
$$
[X_1,[X_2,Y_2]] + [X_2,[Y_2,X_1]] + [Y_2,[X_1,X_2]]=0
$$
we obtain, by a straightforward computation, in view of  \eqref{eq:X2Y2ugualeazero} and  \eqref{eq.rel.comm.c}, that
$$
-X_2(\delta)X_1+\delta\beta X_2 -X_2(\varphi)Y_2 + Y_2(\alpha)X_1  -\alpha \varphi Y_2
+Y_2(\beta)X_2=0\,.
$$
In particular, by equating to zero the component of $X_1$, we get
\begin{equation}\label{eq:fin.1.c}
X_2(\delta)- Y_2(\alpha)=0\,.
\end{equation}
Moreover, from the identity
$$
[Y_1,[X_2,Y_2]] + [X_2,[Y_2,Y_1]] + [Y_2,[Y_1,X_2]]=0 \,,
$$
by performing similar computations, we obtain
$$
X_2(\mu)Y_1 + \mu \xi X_2+X_2(\nu)Y_2 + X_2(\gamma)X_1-\gamma \alpha X_1 -\gamma \beta X_2 + Y_2(\xi)X_2+Y_2(\eta)Y_1 -\eta \nu Y_2 -\eta \gamma X_1=0
$$
and equating to zero the components of  $X_1$ and $Y_1$ we obtain
\begin{equation}\label{eq:fin.2.c}
 X_2(\gamma)- (\alpha   +\eta) \gamma =0 \, ,\quad X_2(\mu)  +Y_2(\eta)   =0\,.
\end{equation}
Conditions \eqref{eq:fin.2.c} together with \eqref{eq:fin.1.c} give \eqref{eq:cond.comp.cc}.
\end{proof}

\begin{lemma}\label{lem.int.1}
Let $X,Y,Z$ be independent vector fields such that 
$$
\mathrm{span}\{X,Y\}\,, \quad \mathrm{span}\{X,Z\}\,, \quad \mathrm{span}\{Y,Z\}
$$
are integrable distributions. Then there exist three non-zero functions
$\alpha,\beta,\gamma$ such that
\begin{equation}\label{eq.comm.2.2.bis}
[\alpha X, \beta Y]=0\,,\quad  [\alpha X, \gamma Z]=0\,,\quad [\beta Y, \gamma Z]=0\,.
\end{equation}
\end{lemma}
\begin{proof}
By hypothesis, there exist functions $a,b,c,d,e,f$ such that
\begin{equation}\label{eq.comm.2.2}
[X,Y]=aX+bY\,,\quad [X,Z]=cX+dZ\,,\quad [Y,Z]=eY+fZ\,.
\end{equation}
By expanding equations \eqref{eq.comm.2.2.bis}, on account of \eqref{eq.comm.2.2}, we obtain the following three systems:
\begin{equation}\label{sistema.1}
\left\{
\begin{array}{l}
Y(\alpha)-a\alpha=0
\\
Z(\alpha)-c\alpha=0
\end{array}
\right. \,,
\quad 
\left\{
\begin{array}{l}
X(\gamma)+d\gamma=0
\\
Y(\gamma)+f\gamma=0
\end{array}
\right.\,,
\quad
\left\{
\begin{array}{l}
X(\beta)+b\beta=0
\\
Z(\beta)-e\beta=0
\end{array}
\right.\,.
\end{equation}
The compatibility conditions of the systems \eqref{sistema.1} are, respectively, 
\begin{equation*}
-ea-fc - Z(a) +Y(c)=0\,, \quad -ad-bf+X(f)- Y(d) =0\,, \quad -cb+de-X(e)-Z(b)=0\,,
\end{equation*}
that are satisfied in view of the Jacobi identity
$$
[Z,[X,Y]] + [X,[Y,Z]] + [Y,[Z,X]]=0\,.
$$
\end{proof}
Proposition \ref{prop:distr.integrabili} and Remark \ref{rem:Y1Y2.along.X1} place us in the position to apply Lemma \ref{lemma:gianni4}  to the vector fields introduced in \eqref{eq:def.4.v.f}.  Hence, taking into account  Lemma \ref{lem.int.1}, up to rescaling $Y_1$ and $Y_2$ by suitable nowhere-vanishing functions and renaming them, there exist two non-zero functions $f$ and $h$ such that the vector fields
$$
U_2:=hX_2,\quad U_3:=Y_1+f X_1,\quad U_4:= Y_2
$$
mutually commute. Finally, in order to obtain a frame of commuting vector fields, we restrict the commuting fields $U_i$ to a leaf $\Sigma$ of the integrable distribution $\mathrm{span}\{U_2, U_3, U_4\} $ and extend them along the transverse Killing field $X_1$, as detailed below.

\smallskip
Let $\phi_t$ be the local flow of $X_1$. Let us then define the extension $U_i^{\mathrm{ext}}$ of $U_i|_\Sigma$, $i\in\{2,3,4\}$, to a neighborhood of $\Sigma$ small enough so that
$$
U_i^{\mathrm{ext}}(\phi_t(p)) := \phi_{t*}\big(U_i(p)\big)\,, \quad  p \in \Sigma\,, \quad i\in\{2,3,4\}\,,
$$
are well defined. This way, we have that $\phi_{t*} \left(U_i^{\mathrm{ext}} \right)= U_i^{\mathrm{ext}}$. Then
 $$
[X_1, U_i^{\mathrm{ext}}] = 0\,. 
 $$
Furthermore,
$$
[U_i^{\mathrm{ext}}, U_j^{\mathrm{ext}}]
  = \phi_{t*}\Big( [U_i , U_j]\big|_\Sigma\Big) = 0\,, 
 $$
so that the vector fields $U_i^{\mathrm{ext}}$ mutually commute and each of them commutes also with $X_1$. 
Since $X_1$ is para-holomorphic (see Proposition \ref{prop:hamilt}), $\phi_{t*}$ commutes with $\T$. Therefore, 
$$  
\T \, \phi_{t*} \big(Y_1\big)= \phi_{t*}\big( Y_1\big)\,,\quad   \T \, \phi_{t*} \big(Y_2\big)= -\phi_{t*} \big(Y_2\big)\,.
$$
Then, considered that
$$
T^+=\mathrm{span}\{Y_1,X_1+X_2\}\,,\quad T^-=\mathrm{span}\{Y_2,X_1-X_2\}\,,
$$
we have the following decompositions:
$$
\phi_{t*} (Y_1)=\big(a_1 Y_1+a_2(X_1+X_2)\big)\circ\phi_t\,,\quad \phi_{t*} (Y_2)=\big(a_3 Y_2+a_4(X_1-X_2)\big)\circ\phi_t\,.
$$
On the other hand, $X_1$ is also a Killing vector field (see Corollary \ref{cor:mui.killing}), implying that $\phi_t$ is a local isometry. Then, we have that
\begin{equation*}
\begin{array}{l}
a_2\|X_1\|^2_{\phi_t(p)}=g_{\phi_t(p)}(\phi_{t*} \big(Y_1(p)\big),X_1(\phi_t(p)))=g_p( Y_1(p),X_1(p))=0\,,
\\[2mm]
a_4\|X_1\|^2_{\phi_t(p)}=g_{\phi_t(p)}(\phi_{t*}\big( Y_2(p)\big),X_1(\phi_t(p)))=g_p( Y_2(p),X_1(p))=0\,.
\end{array}
\end{equation*}
Since we are assuming that $X_1$ is not isotropic, then $a_2=a_4=0$ and, recalling also that $[X_1,X_2]=0$ (see Proposition \ref{prop:para-holomorphic}), we have that
\begin{equation}\label{eq:appoggio.filippo}
X_1\,,\quad U_2^{\mathrm{ext}}=\phi_{t*}(h X_2)=hX_2\,,\quad U_3^{\mathrm{ext}}=\phi_{t*}(Y_1+f X_1)=a_1Y_1+fX_1,\quad U_4^{\mathrm{ext}}= \phi_{t*}(Y_2)=a_3Y_2
\end{equation}
mutually commute. In the remaining part of the present section, we shall denote by
$
(x^1,x^2,x^3,x^4)
$
a local system of coordinates where the vector fields \eqref{eq:appoggio.filippo} are coordinate ones, i.e., $\partial_1=X_1$, $\partial_i=U_i^{\mathrm{ext}}$, $i\in\{2,3,4\}$.

\medskip\noindent
In this system of coordinates we have that
$$
\T=\begin{pmatrix}
0 &h& -f &0\\
\frac{1}{h} & 0 &\frac{f}{h} &0\\
0 &0 & 1 &0\\
0&0&0&-1
\end{pmatrix}\,,
\quad A=\begin{pmatrix}
\rho & 0 &f(\rho-c)  &0\\
0 & \rho &0 &0\\
0 &0 & c &0\\
0& 0 &0&c
\end{pmatrix}.$$
Furthermore, we have also that
\begin{multline*}
g(\partial_1,\partial_1)=-g(\grad\rho,\grad \rho)=-\frac{g(\grad\rho,\partial_2)}{h}=-\frac{\rho_{x^2}}{h} \,, \quad \rho_{x^1}= g(\grad\rho,\partial_1)=g(\grad\rho, \T\grad\rho)=0\,,
\\
 \rho_{x^3}= g(\grad\rho,\partial_3)=g( \grad\rho,Y_1+f\, \T \grad\rho)=0\,, \quad
\rho_{x^4}= g(\grad\rho,\partial_4)=g(\grad\rho, Y_2)=0\,.
\end{multline*}
Thus,
$\rho=\rho(x^2)$.

\medskip\noindent
By performing similar computations, we obtain also the remaining components of $g$:
$$
g(\partial_1,\partial_3)=f\, g(\T\grad\rho,\T\grad \rho)=- \frac{f\rho_{x^2}}{h}\,, \quad g(\partial_2,\partial_2)=h\, g(\grad\rho,\partial_2)=h \rho_{x^2}\,.
$$
Then, we obtain
$$
g=\begin{pmatrix}
-\frac{\rho_{x^2}}{h} & 0 & - \frac{f\rho_{x^2}}{h}  &0\\
0 & h\rho_{x^2} &0  &0\\
- \frac{f\rho_{x^2}}{h}   &0&-\frac{f^2\rho_{x^2}}{h}  &g_{34}\\
0&0&g_{34}&0
\end{pmatrix}\,.
$$
Since $\partial_1$ is a Killing vector field, $h$, $f$ and $g_{34}$ are independent of $x^1$.

\medskip\noindent
In view of the fact that $\det g=g_{34}^2 \rho_{x^2}^2\neq 0$, the  $(4,1)$-entry of $\nabla^{g}_{\partial_1} \T$ vanishes if and only if
$h_{x^3}=-f_{x^2}$, 
so that there exists a function $H(x^2,x^3,x^4)$ such that
$$h=H_{x^2}\,,\quad  f=-H_{ x^3}.$$

\smallskip\noindent
The $(3,1)$-entry of $\nabla^{g}_{\partial_1} \T$ vanishes if and only if $H_{x^2x^4}=0$, that implies
$$
H(x^2,x^3,x^4)= \alpha(x^2,x^3)+\beta(x^3,x^4).
$$

\smallskip\noindent
Taking into account that $h=\alpha_{ x^2}\neq 0$, the  $(3,1)$-entry of $\nabla^{g}_{\partial_3} \T$ vanishes if and only if
$$(g_{34})_{x^2}=-\rho_{x^2}\, \beta_{x^3x^4}.$$

\smallskip\noindent
By studying the $(3,2)$-entry of the equation \eqref{eq.A} with $X=\partial_3$, we obtain that
$$
g_{34}=(c-\rho) \beta_{x^3x^4}.
$$

\smallskip\noindent
Taking into consideration that $\det g=(\rho-c)^2\rho_{x^2}^2\, \beta_{x^3x^4}^2\neq 0$, by studying the $(4,1)$-entry of the equation \eqref{eq.A} with $X=\partial_1$, we get $\alpha_{x^2x^3}=0$, that implies
$$
\alpha(x^2,x^3)=\xi(x^2)+\eta(x^3).
$$

\noindent
In conclusion, the general solution is
$$\T=\begin{pmatrix}
0 &\xi_{x^2}& \eta_{x^3}+\beta_{x^3} &0\\
\frac{1}{\xi_{x^2}} & 0 &-\frac{\eta_{x^3}+\beta_{x^3}}{\xi_{x^2}} &0\\
0 &0 & 1 &0\\
0&0&0&-1
\end{pmatrix},\qquad\quad A=\begin{pmatrix}
\rho & 0 &(c-\rho) (\eta_{x^3}+\beta_{x^3})&0\\
0 & \rho &0 &0\\
0 &0 & c &0\\
0& 0 &0&c
\end{pmatrix}.$$
$$
g=\begin{pmatrix}
-\frac{\rho_{x^2}}{\xi_{x^2}} & 0 &  \frac{(\eta_{x^3}+\beta_{x^3})\rho_{x^2}}{\xi_{x^2}}  &0\\
0 & \xi_{x^2}\rho_{x^2} &0  &0\\
 \frac{(\eta_{x^3}+\beta_{x^3})\rho_{x^2}}{\xi_{x^2}}   &0&-\frac{(\eta_{x^3}+\beta_{x^3})^2\rho_{x^2}}{\xi_{x^2}} &(c-\rho) \beta_{x^3x^4} \\
0&0&(c-\rho) \beta_{x^3x^4}&0
\end{pmatrix}.
$$
where we recall that $\rho(x^2)$ is an arbitrary non-constant function, $\xi(x^2)$, $\eta(x^3)$ and $\beta(x^3,x^4)$ are arbitrary functions such that $\xi_{x^2}\neq 0$ and $\beta_{x^3x^4}\neq 0$. By putting $\mu=-\xi_{x^2}$ and $\nu=\eta_{x^3}+\beta_{x^3}$, we arrive to the normal forms of the first case of $\dim D=2$ of Theorem \ref{thm:main}.

%
%

\subsubsection{The case where $\grad\rho,\grad\sigma\in\Gamma(T^+)\setminus\{0\}$}\label{sec:dim.2.plus.plus}

In this case, recalling \eqref{eq:mathcal.D}, we have that
\begin{equation*}
\T\mathcal{D}=\mathcal{D}=\operatorname{span}\{V_1=\grad\rho +\grad \sigma,\,V_2=\sigma\grad \rho+\rho \grad\sigma\}\,,
\end{equation*}
which is integrable in view of  Proposition \ref{prop:para-holomorphic}. We shall complete $\{V_1,V_2\}$ to a frame 
by adding two vector fields in $\Gamma(T^-)\setminus\{0\}$ with the following properties: they have to commute, have to commute also with $V_1$ and $V_2$ and have to form with these latter a system of independent vector fields, so that the obtained four vector fields lead to a local system of coordinates. The construction of the aforementioned additional two vector fields goes as follows.

\bigskip\noindent
Let $\Sigma$ be a leaf of the integrable distribution $T^-$. We can take, w.l.o.g., two vector fields $X\in \Gamma(E_\rho\cap T^-)$ and $Y\in\Gamma(E_\sigma\cap T^-)$ on $\Sigma$ such that $[X,Y]=0$. 
Let $X^\mathrm{ext}$ and $Y^\mathrm{ext}$ be, respectively, the extension of $X$ and $Y$
along the flows of $V_1$ and $V_2$. 

\smallskip\noindent
Such extensions are well defined, in a neighborhood of $\Sigma$ small enough, as $V_1$ and $V_2$ commutes.

\smallskip\noindent
Also, since $V_1$ and $V_2$ are para-holomorphic (see Proposition \ref{prop:hamilt}), their local flows preserve $T^-$, implying that both $X^\mathrm{ext}$ and $Y^\mathrm{ext}$ are sections of $T^-$. 

\smallskip\noindent
Furthermore, since in this case $V_1=\T V_1$ and $V_2=\T V_2$ are also Killing vector fields (see Corollary \ref{cor:mui.killing}), being their local flows local isometries, and since on $\Sigma$ we have that $g(X,V_2-\sigma V_1)=g(Y,V_2-\rho V_1)=0$, then 
\begin{equation}\label{XYext}
g(X^\mathrm{ext}, V_2-\sigma V_1)=g(Y^\mathrm{ext}, V_2-\rho V_1)=0. 
\end{equation}
Indeed, if we denote by $\phi_t$ a generic combination of the local flow of $V_1$ and $V_2$, taking into account that $[V_1,V_2]=0$, we have that
$$
\phi_{t*}(V_2-\sigma V_1)=(V_2-\sigma V_1)\circ \phi_t,\quad\phi_{t*}(V_2-\rho V_1)=(V_2-\rho V_1)\circ \phi_t\,.
$$
Equation \eqref{XYext} implies that $X^\mathrm{ext}$ is proportional to $X$ and that $Y^\mathrm{ext}$ is proportional to $Y$.
Indeed, since $X^\mathrm{ext}$  is a section of $T^-$, it can have non-vanishing components only along $X$ and $Y$. However, since $V_2-\sigma V_1=(\rho-\sigma)\grad \sigma \in \Gamma (E_\sigma\cap T^+)$, $g(X^\mathrm{ext}, V_2-\sigma V_1)=0$ implies that the component of $X^\mathrm{ext}$ along $Y$ is zero. The same reasoning applies to $Y^\mathrm{ext}$.

\smallskip\noindent
To conclude, we have constructed four vector fields
\begin{equation}\label{eq:appoggio.filippo.2}
V_1,\, V_2,\, X^\mathrm{ext}, \, Y^\mathrm{ext}
\end{equation}
that mutually commute.

\bigskip\noindent
In the remaining part of the present section, we shall denote by
$
(x^1,x^2,x^3,x^4)
$
a local system of coordinates where the vector fields \eqref{eq:appoggio.filippo.2} are coordinate ones, i.e., 
$$
\partial_1=V_1,\quad \partial_2=V_2,\quad \partial_3=X^\mathrm{ext},\quad\partial_4=Y^\mathrm{ext}\,.
$$
Now, let us compute the components of the metric $g$ in the aforementioned system of coordinates. Since $\partial_1,\partial_2\in\Gamma(T^+)$, we have that
$$
g(\partial_1,\partial_1)=g(\partial_2,\partial_1)=g(\partial_2,\partial_2)=0\,.
$$
Therefore,
$$
\rho_{x^1}+\sigma_{x^1}=0\,,\quad \sigma\rho_{x^1}+\rho\sigma_{x^1}=0\,,
$$
implying that both $\rho$ and $\sigma$ are independent of $x^1$. Similarly, one can see that they are independent also of $x^2$. Furthermore, by considering that  $\partial_3\in \Gamma(E_\rho)$ and that $\partial_4\in \Gamma(E_\sigma)$, we have that
$$
g(\grad \rho, \partial_4)=g(\grad \sigma, \partial_3)=0\,,
$$
implying that $\rho$ is independent of  $x^4$ and that $\sigma$ is independent of  $x^3$. To sum up, the matrix $(g_{ij})$ has the following form:
\begin{equation*}
g=\left(
\begin{array}{cccc}
 0 & 0 & \rho_{x^3}  & \sigma_{x^4} \\
 0 & 0 &\sigma \rho_{x^3}&\rho\sigma_{x^4}\\
\rho_{x^3} &  \sigma \rho_{x^3}& 0 & 0 \\
\sigma_{x^4} &\rho\sigma_{x^4}& 0 & 0 \\
\end{array}
\right).
\end{equation*}
Since we are dealing with a coordinate system adapted to $T^+$ and $T^-$, the matrix of the components of the para-complex structure $\T$ is given by \eqref{eq:matrix.T.adapted} with $n=2$.

\smallskip\noindent
The matrix of the components of the tensor field $A$ turns out to be
$$
A=\left(
\begin{array}{cccc}
\rho+\sigma& \rho\sigma & 0 & 0 \\
 -1 & 0 & 0 & 0 \\
 0 & 0 & \rho&0\\
 0 & 0 & 0 & \sigma \\
\end{array}
\right).$$
A straightforward computation shows that both the equations \eqref{eq.A} and $\nabla^{g} \T=0$ are satisfied. In conclusion, we arrived to the normal forms of the second case of $\dim D=2$ of Theorem \ref{thm:main}.

\subsubsection{The case where $\grad\rho,\grad\sigma\in\Gamma(T^-)\setminus\{0\}$}

This case closely follows that treated in Section \ref{sec:dim.2.plus.plus}. Indeed, one can prove that
$$
-V_1=\grad \rho+\grad\sigma\,,\quad -V_2=\rho\grad\sigma+\sigma\grad\rho\,, \quad X^\mathrm{ext}\,, \quad Y^\mathrm{ext}\,,
$$
where $X^\mathrm{ext}$ and $Y^\mathrm{ext}$ are the extension along $V_1$ and $V_2$ of two vector fields $X\in \Gamma(E_\rho\cap T^+)$, $Y\in\Gamma(E_\sigma\cap T^+)$ defined on a leaf of $T^+$, mutually commute. Also, similarly to the case of Section \ref{sec:dim.2.plus.plus}, one can prove that $X^\mathrm{ext}$ and $Y^\mathrm{ext}$ are, respectively, proportional to $X$ and $Y$.

\medskip\noindent
The general solution is given by  $(g,-\T,A)$, where $(g,\T,A)$ are as in the previous section.

\subsubsection{The case where $\grad\rho\in\Gamma(T^+)\setminus\{0\}$, $\grad\sigma\in\Gamma(T^-)\setminus\{0\}$}

In this case, recalling again \eqref{eq:mathcal.D}, we have that
$$
\T\mathcal{D}=\mathcal{D}=\text{span}\{\T V_1=\grad\rho -\grad \sigma,\,\T V_2=\sigma\grad \rho-\rho \grad\sigma\}
$$
that, in view of  Proposition \ref{prop:para-holomorphic}, is integrable. Similarly to what was done in the previous sections, we shall complete $\{\T V_1,\T V_2\}$ to a frame 
by adding two vector fields: in this case, one vector field will belong to $\Gamma(T^-)\setminus\{0\}$ and the other to $\Gamma(T^+)\setminus\{0\}$. We shall construct these vector fields such that they satisfy the following properties: they have to commute, have to commute also with $\T V_1$ and $\T V_2$ and have to form with these latter a system of independent vector fields.
The construction of the aforementioned additional two vector fields goes as follows.

\smallskip\noindent
Let $X\in\Gamma(T^-\cap E_\rho)$ and  $Y\in\Gamma(T^+\cap E_\sigma)$. Since
$
TM=\mathcal{D}\oplus \mathrm{span}\{X,\,Y\}
$
and $g(X,Y)=g(\grad\rho,\grad\sigma)=0$ (the eigenspaces of $A$ are $g$-orthogonal, see Lemma \ref{lem:grad.eigenvectors}), then $X(\rho)=g(X,\grad \rho)\neq 0$  and $Y(\sigma)=g(Y,\grad\sigma)\neq 0$ (otherwise the metric would be degenerate).
Taking into account that both $X$ and $Y$ are isotropic, i.e., $g(X,X)=g(Y,Y)=0$, then
$$
\left(\mathrm{span}\{X,\,Y\}\right)^\bot=\mathrm{span}\{X,\,Y\}\,.
$$
Now we prove that $\mathrm{span}\{X,\,Y\}$ is integrable. 

\smallskip\noindent
Let $Z$ be an arbitrary section of $\mathrm{span}\{X,\,Y\}$. Then,
$$
g\left(\nabla^{g}_{Z}X, X\right)= Z\,g\left(X, X\right)-g\left(\nabla^{g}_{Z}X,X\right)\quad\Rightarrow\quad g\left(\nabla^{g}_{Z}X, X\right)=0\,.
$$
Similarly, $g\left(\nabla^{g}_{Z}Y, Y\right)=0$. 
Furthermore, we have that
\begin{multline*}
\sigma\; g\left(\nabla^{g}_{Z}Y, X\right)= \sigma\Big( Z \big(g(Y,X) \big) -g\left(Y, \nabla^{g}_{Z}X\right)\Big)=-g\left(\sigma Y, \nabla^{g}_{Z}X\right)=\\
-g\left(A Y, \nabla^{g}_{Z}X\right)=-g\left(Y,A( \nabla^{g}_{Z}X)\right)=
g\Big(Y, \left(\nabla^{g}_{Z}A\right)(X)-\nabla^{g}_{Z}( A X) \Big)=\\
g\Big(Y, \left(Z^\flat\otimes \Lambda+\Lambda^\flat \otimes Z -(\T Z)^\flat\otimes \T \Lambda-(\T\Lambda)^\flat \otimes \T Z\right)(X)-\nabla^{g}_{Z}( \rho X) \Big)\overset{(*)}{=}\\
-\rho\; g\left(Y, \nabla^{g}_{Z}X\right)=
\rho\; g\left(\nabla^{g}_{Z}Y, X\right)
 \end{multline*}
(the step $(*)$ holds true as $g(X,Z)=g(Y,Z)=0$). \\
Thus, since $\sigma\neq \rho$, we have that
$g\left(\nabla^{g}_{Z}Y, X\right)=0$, that implies 
$g\left(\nabla^{g}_{Z}X, Y\right)=0$. Hence, we can conclude that $\mathrm{span}\{X,\,Y\}$ is totally geodesic, then  integrable.

Up to rescaling $X$ and $Y$ by suitable functions, we can suppose $[X,Y]=0$. Now, let $\Sigma$ be a leaf of $\mathrm{span}\{X,\,Y\}$. Let us denote by $X^\mathrm{ext},Y^\mathrm{ext}$, respectively, the extensions of  
$X|_\Sigma$ and $Y|_\Sigma$ along the para-holomorphic Killing vector fields 
 $\T V_1$, $\T V_2$ (cfr. Corollary \ref{cor:mui.killing} and Proposition \ref{prop:hamilt}).  This way, we obtain four linearly independent vector fields
\begin{equation}\label{eq:appoggio.filippo.3} 
\T V_1,\T V_2,X^\mathrm{ext},Y^\mathrm{ext}
\end{equation}
that mutually commute.


Since $\T V_1$ and $\T V_2$ are para-holomorphic vector fields, in view also of the fact that $X\in\Gamma(T^-\cap E_\rho)$ and $Y\in\Gamma(T^+\cap E_\sigma)$, we have that $X^\mathrm{ext}\in\Gamma(T^-)$ and $Y^\mathrm{ext}\in\Gamma(T^+)$. Hence, the following decompositions hold true, for suitable functions $f_1\neq0 ,f_2,h_1\neq 0,h_2$:
\begin{equation}\label{ext.decomp}
X^\mathrm{ext}=f_1 X+ f_2\grad\sigma\,, \quad Y^\mathrm{ext}=h_1 Y+h_2 \grad \rho\,.
\end{equation}
Since $\T V_1$ and $\T V_2$ are also Killing vector fields, any composition of their local flows is a local isometry. In particular, we have that
\begin{equation}\label{gXY}
0=g(X,Y)|_\Sigma=g(X^\mathrm{ext},Y^\mathrm{ext} )=f_1h_2 X(\rho)+h_1 f_2 Y(\sigma).
\end{equation}
In the remaining part of the present section, we shall denote by
$
(x^1,x^2,x^3,x^4)
$
a local system of coordinates where the vector fields \eqref{eq:appoggio.filippo.3} are coordinate ones, i.e., 
$$\partial_1=\grad\rho -\grad \sigma,\quad \partial_2=\sigma\grad \rho-\rho \grad\sigma,\quad \partial_3=X^\mathrm{ext},\qquad\partial_4=Y^\mathrm{ext}\,.$$
Now we compute the components of the metric $g$ in the aforementioned system of coordinates. Note that
$$g(\grad \rho,\partial_1)=g(\grad \rho,\partial_2)=g(\grad \rho,\partial_4)=g(\grad\sigma,\partial_1)=g(\grad\sigma,\partial_2)=g(\grad\sigma,\partial_3)=0\,,$$
implying
$$
\rho=\rho(x^3)\,,\quad \sigma=\sigma(x^4)\,,
$$
from which, in view of \eqref{gXY}, we obtain
\begin{equation}\label{eq:g.apostolov.second}
g=\left(
\begin{array}{cccc}
 0 & 0 & \rho_{x^3} & -\sigma_{x^4}  \\
 0 & 0 &\sigma\,\rho_{x^3}&-\rho\,\sigma_{x^4} \\
 \rho_{x^3}  &  \sigma\,\rho_{x^3} & 0 & 0 \\
-\sigma_{x^4}  &-\rho\,\sigma_{x^4} &0 & 0 \\
\end{array}
\right).
\end{equation}
Taking into account that
$$\grad\rho=\frac{\rho}{\rho-\sigma}\partial_1-\frac{1}{\rho-\sigma}\partial_2,\qquad\grad\sigma=\frac{\sigma}{\rho-\sigma}\partial_1-\frac{1}{\rho-\sigma}\partial_2\,,
$$
one can readily compute the components of the para-complex structure,
\begin{equation}\label{eq:tau.apostolov.second}
\T=\left(\begin{array}{cccc}
 \frac{\rho+\sigma}{\rho-\sigma}& \frac{2\rho\sigma}{\rho-\sigma}&0&0\\
-\frac{2}{\rho-\sigma}& -\frac{\rho+\sigma}{\rho-\sigma}&0&0\\
0&0&-1&0\\
0&0&0&1
\end{array}\right)
\end{equation}
and, in view of  \eqref{ext.decomp}, the components of $A$,
\begin{equation}\label{eq:A.apostolov.second}
A=\left(
\begin{array}{cccc}
\rho+\sigma& \rho\sigma & -\sigma f_2&\rho\, h_2\\
 -1 & 0 & f_2 & -h_2\\
 0 & 0 & \rho&0\\
 0 & 0 & 0 & \sigma \\
\end{array}
\right)\,,
\end{equation}
where $h_2$ can be computed from \eqref{gXY}:
$$
h_2=-f_2\frac{g(\partial_4,\grad\sigma)}{g(\partial_3,\grad\rho)}=-f_2\frac{\sigma_{x^4}}{\rho_{x^3}}\,.
$$
A straightforward computation shows that, by using the results obtained so far, the equation $\nabla^{g} \T=0$ is satisfied, whereas equation \eqref{eq.A} still not.
Indeed, by considering the $(2,3)$-entries of the equation \eqref{eq.A} with $X=\partial_i$, $i\in\{1,2\}$, we obtain $f_2=f_2(x^3,x^4)$, whereas by studying the $(2,3)$-entries of the same equation with $X=\partial_j$, $j\in\{3,4\}$, recalling that $\det g=(\rho-\sigma)^2\rho_{x^3}^2\sigma_{x^4}^2\neq0$, we get
$$
f_2(x^3,x^4)=k\frac{\rho_{x^3}}{\rho-\sigma}\,,\quad k\in\R\,.
$$

\smallskip\noindent
In conclusion, in this case, we have $g$ and $\T$, respectively, given by \eqref{eq:g.apostolov.second} and  \eqref{eq:tau.apostolov.second}, whereas $A$ is \eqref{eq:A.apostolov.second} with $f_2$ and $h_2$ given above. Thus, we arrive to the fourth case of $\dim D=2$ of Theorem \ref{thm:main}.

\subsection{Case $\dim D=1$}\label{sec.rank.1}

\subsubsection{The case where $\grad\rho\in\Gamma(T^+)$ and $\grad\sigma=0$}\label{sec:1.dim.plus.plus}

In this case, $\rho$ is a non-constant eigenvalue of $A$ whereas $\sigma=c\in\R\setminus\{0\}$. 

\smallskip\noindent
We shall consider the following vector fields:
\begin{equation}\label{eq:def.4.v.f.1.dim}
X_1:=\T V_1=V_1\,,\quad X_2\in \Gamma(T^-\cap E_\rho)\,,\quad Y_1\in\Gamma(T^+\cap E_\sigma)\,,\quad Y_2\in\Gamma(T^-\cap E_\sigma)\,.
\end{equation}
We note that $\mathrm{span}\{X_1,Y_1\}$ and $\mathrm{span}\{X_2,Y_2\}$ are integrable as they are, respectively, the eigendistributions $T^+$ and $T^-$ of $\T$. Moreover, similarly to the case treated in Section \ref{sec.rank.2}, we have the following proposition.
\begin{prop}\label{prop:distr.integrabili.1.dim}
Let $X_1,X_2,Y_1,Y_2$ as in \eqref{eq:def.4.v.f.1.dim}. The distributions  $\mathrm{span}\{X_1,X_2\}$, $\mathrm{span}\{X_1,Y_2\}$ and $\mathrm{span}\{X_2,Y_1\}$ are integrable, whereas $\mathrm{span}\{Y_1,Y_2\}$ is not, as it has a non-vanishing component along $X_1$.
\end{prop}
\begin{proof}
The proof closely follows that of Proposition \ref{prop:distr.integrabili}, so, for the sake of brevity we omit it.
\end{proof}


The vector field $X_1$, in view of Corollary \ref{cor:mui.killing}, is a Killing vector field and it will be the first coordinate vector field in a suitable local system of coordinates that we are going to construct below. Firstly, in view of Proposition \ref{prop:distr.integrabili.1.dim}, taking into account Lemmas  \ref{lemma:gianni4} and \ref{lem.int.1}, up to rescaling $X_2$, $Y_1$ and $Y_2$ by appropriate functions and renaming them, there exists a function $f$ such that the following vector fields
$$U_2=X_2,\quad U_3=Y_1+f X_1,\quad U_4= Y_2$$
mutually commute.

Let $\Sigma$ be a leaf of the distribution $\mathrm{span}\{U_2, U_3, U_4\}$. Let $\phi_t$ be the local flow of $X_1$. Let us then define the extension $U_i^{\mathrm{ext}}$ of  $U_i|_\Sigma $ to a neighborhood of $\Sigma$, small enough, as follows:
$$
U_i^{\mathrm{ext}}(\phi_t(p)) := \phi_{t*}\big(U_i(p)\big)\,,  \quad  p \in \Sigma\,.
$$
This way, we have that $U_i^{\mathrm{ext}}$ is constant along the  integral curves of $X_1$, i.e.,  $\phi_{t*} (U_i^{\mathrm{ext}}) = U_i^{\mathrm{ext}}$. Hence, 
 $[X_1, U_i^{\mathrm{ext}}] = 0$ and, moreover,  $[U_i^{\mathrm{ext}}, U_j^{\mathrm{ext}}] = \phi_{t*}\Big( [U_i , U_j]\big|_\Sigma\Big) = 0
$.
Thus, the three vector fields $U_i^{\mathrm{ext}}$, $i\in\{2,3,4\}$, mutually commute and each of them commutes also with $X_1$.

\smallskip\noindent
Since $X_1$ is para-holomorphic (see Proposition \ref{prop:hamilt}), $\T$ commutes with $\phi_{t*}$. Then, 
$$ 
\T \, \phi_{t*} \big(X_2(p)\big)=\phi_{t*}\big( X_2(p)\big)\,,\quad  
\T \, \phi_{t*} \big(Y_1(p)\big)= -\phi_{t*}\big( Y_1(p)\big)\,, \quad   
\T \, \phi_{t*} \big(Y_2(p)\big)= \phi_{t*} \big(Y_2(p)\big),  \quad  p \in \Sigma\,.
$$ 
Therefore, recalling that $\{X_1,Y_1\}$ is a basis of $T^+$ whereas $\{X_2,Y_2\}$ is a basis of $T^-$, we have the following decompositions:
$$  
\phi_{t*} \big(X_2(p)\big)=h_1 Y_2+h_2 X_2\,,\quad \phi_{t*} \big(Y_1(p)\big)=h_3 Y_1+h_4 X_1\,,\quad  \phi_{t*} \big(Y_2(p)\big)=h_5 Y_2+h_6 X_2\,,\quad p \in \Sigma\,.
$$
On the other hand,  $\phi_t$ is a local isometry being $X_1$ a Killing vector field. In particular, we have the following equalities:
$$
\begin{array}{l}
h_6\,g_{\phi_t(p)}(X_1,X_2)=
g_{\phi_t(p)}\left(\phi_{t*} \big(X_1(p)\big),\phi_{t*} \big(Y_2(p)\big)\right)=g_p( X_1(p),Y_2(p))=0\,.
\end{array}
$$
Since $g(X_1,X_2)\neq 0$,
we have that $h_6=0$,
that implies
$$  
\phi_{t*} \big(X_2(p)\big)=h_1 Y_2+h_2 X_2\,, \quad \phi_{t*} \big(Y_1(p)\big)=h_3 Y_1+h_4 X_1\,,\quad  \phi_{t*} \big(Y_2(p)\big)=h_5 Y_2\,.
$$
Finally, the following vector fields
\begin{equation}\label{eq:appoggio.filippo.4}
X_1,\quad U_2^{\mathrm{ext}}=h_2 X_2+h_1 Y_2,\quad U_3^{\mathrm{ext}}=h_3Y_1+(h_4+f) X_1,\quad U_4^{\mathrm{ext}}= h_5Y_2\,,
\end{equation}
mutually commute.
In the remaining part of the present section, we shall denote by
$
(x^1,x^2,x^3,x^4)
$
a local system of coordinates where the vector fields \eqref{eq:appoggio.filippo.4} are coordinate ones, i.e., $\partial_1=X_1$ and $\partial_i=U_i^{\mathrm{ext}}$, $i\in\{2,3,4\}$.
In this system of coordinates, the matrix of the components of the para-complex  structure $\T$ is
\begin{equation*}
\T=\begin{pmatrix}
1 &0 & 0 &0\\
0 & -1 &0 &0\\
0 &0 & 1 &0\\
0&0&0&-1
\end{pmatrix},
\end{equation*}
whereas the matrix of the components of the tensor field $A$ is
\begin{equation*}
A=\begin{pmatrix}
\rho & 0 &f_2(\rho-c) &0\\
0 & \rho &0  &\\
0 &0 & c &0\\
0& f_1(c-\rho) &0&c
\end{pmatrix}.
\end{equation*}
where $f_1=h_1$ and $f_2=h_4+f$.

\smallskip\noindent
Let us compute the metric coefficients $g(\partial_i,\partial_j)$. We have that
\begin{multline*}
\rho_{x_1}=g(\grad\rho,\partial_1)=g(K \grad\rho,K\grad \rho)=0\,, \quad
 \rho_{x_3}= g(\grad\rho,\partial_3)=g(X_1,Y_1+f_2 X_1)=0\,,
\\
 \rho_{x_4}=g(\grad\rho,\partial_4)=g(X_1,Y_2)=0\,,
\end{multline*}
so that $\rho=\rho(x^2)$. Similarly, we obtain also the other components of $g$:
$$
g(\partial_3,\partial_4)=g(Y_1,Y_2)\,,\quad 
g(\partial_2,\partial_3)=g(X_2+f_1Y_2,Y_1+f_2X_1)=f_2\, g(X_2,X_1)+f_1\, g(Y_1,Y_2)=f_2 \rho_{x_2}+f_1 g_{34}\,,
$$
so that we have that
\begin{equation*}
g=\begin{pmatrix}
0& \rho_{x_2} & 0&0\\
\rho_{x_2} & 0&f_2 \rho_{x_2}+f_1 g_{34}  &0\\
0  &f_2 \rho_{x_2}+f_1 g_{34}& 0 &g_{34}\\
0& 0&g_{34}&0
\end{pmatrix}.
\end{equation*}
The components of $g$ are independent of $x^1$ as $\partial_1$ is a Killing vector field. Taking this into account,  $\nabla^{g}\T=0$ if and only if 
\begin{equation}\label{cond1}
\big( f_1g_{34}\big)_{x^4}-\big( g_{34}\big)_{x^2}+\rho_{x^2} f_2=0.
\end{equation}
By considering that $\det g=(g_{34})^2\rho_{x^2}^2\neq 0$ and that $g_{23}$ is independent of $x^1$,
the 
$(1,3)$-entry of equation \eqref{eq.A} with $X=\partial_1$ implies that $f_2$ is independent of $x^1$. Consequently, recalling again that $g_{23}$ is independent of $x^1$, also $f_1$ is independent of $x^1$.

Taking into account\eqref{cond1}, the $(4,2)$-entry of  equation \eqref{eq.A} with $X=\partial_4$ implies that 
$$
g_{34}= (\rho-c)\,\big( f_2\big)_{x^4}.
$$
The $(4,2)$-entry of \eqref{eq.A} with $X=\partial_2$ implies that
\begin{equation}\label{eq:f1.1.dim}
f_1=\frac{(f_2)_{x^2}}{(f_2)_{x^4}}\,,
\end{equation}
whereas the $(4,2)$-entry of \eqref{eq.A} with $X=\partial_3$ implies that \eqref{eq:f1.1.dim} is independent of $x^3$. Then
$$
f_2(x^2,x^3,x^4)=F(x^3,\varphi(x^2,x^4)).
$$
Finally, to sum up, we have that
\begin{equation}\label{eq:g.T.A.dim.1.plus}
\T=\begin{pmatrix}
1 &0 & 0 &0\\
0 & -1 &0 &0\\
0 &0 & 1 &0\\
0&0&0&-1
\end{pmatrix},\quad
A=\begin{pmatrix}
\rho & 0 &(\rho-c)F &0\\
0 & \rho &0  & 0\\
0 &0 & c &0\\
0& (c-\rho)  \frac{F_{x^2}}{F_{x^4}} &0&c
\end{pmatrix},
\end{equation}
$$
g=\begin{pmatrix}
0& \rho_{x^2} & 0&0\\
\rho_{x^2} & 0& \big((\rho-c)F\big)_{x^2} &0\\
0  &  \big((\rho-c)F\big)_{x^2}   & 0 &(\rho-c)F_{x^4}\\
0& 0&(\rho-c)F_{x^4}&0
\end{pmatrix}.
$$
where $\rho=\rho(x^2)$, $F=F(x^3,\varphi(x^2,x^4))$ such that $\rho_{x^2}\neq 0$, $F_{x^4}\neq 0$.  By interchanging $x^2\leftrightarrow x^3$, we obtain the normal forms of the first case of $\dim D=1$ of Theorem \ref{thm:main}.

\subsubsection{The case where $\grad\rho\in\Gamma(T^-)$ and $\grad\sigma=0$}

This case closely follows that treated in Section \ref{sec:1.dim.plus.plus}. The general solution is given by  $(g,-\T,A)$, where $(g,\T,A)$ given by \eqref{eq:g.T.A.dim.1.plus}. By interchanging $x^2\leftrightarrow x^3$, we arrive to the normal forms of the second case of $\dim D=1$ of Theorem \ref{thm:main}.

\section{Four-dimensional para-K\"ahler-Einstein metrics with non-parallel Benenti tensors}\label{sec:Einstein.4.dim}

Taking into account Theorem \ref{thm:main.einstein}, if we find two non-proportional Einstein metrics within the same pc-projective class, we can construct a $2$-parametric family of pc-projectively equivalent Einstein metrics. For this reason, in this section, we determine when a metric $g$ occurring in Theorem \ref{thm:main} is Einstein together with its associated metric $\hat g$ given by \eqref{eq: companion.A}, since $g$ and $\hat{g}$ are pc-projectively equivalent and non-proportional. The subsections below are organized according to the cases of Theorem \ref{thm:main}.

\smallskip
In what follows we shall denote $\E=\Ric(g)-\lambda\,g$ (resp. $\hat\E=\Ric(\hat{g})-\hat\lambda\,\hat{g}$), so that $g$ (resp. $\hat{g}$) is an Einstein metric with Einstein constant $\lambda$ (resp. $\hat{\lambda}$) if and only if  $\E=0$ (resp. $\hat\E=0$). 

\subsection{The non-degenerate case}

\subsubsection{The case of real Liouville metrics}
\begin{prop}\label{prop.Einstein.real.Liouville}
A metric  $g$ of the real Liouville type from Theorem \ref{thm:main} is an Einstein metric with Einstein constant equal to $\lambda$ if and only if the following system 
\begin{equation}\label{eq:system.Einstein.real.Liouville}
\left\{
\begin{array}{l}
3\rho_{x^1}^2+2\lambda\rho^3-3k\rho^2-6h\rho = c_1
\\[0.2cm]
3\sigma^2_{x^2}-\varepsilon(2\lambda\sigma^3-3k\sigma^2-6h\sigma)=c_2
\end{array}
\right.
\end{equation}
is satisfied for some constants $h,k,c_1,c_2\in\R$.
\end{prop}
%
%
\begin{proof}
The equation
$
\E_{11}+\varepsilon\E_{22}=0
$
is equivalent to
\begin{equation*}
\frac{\rho_{x^1x^1x^1}}{\rho_{x^1}} + 2\lambda\rho =  -\varepsilon \frac{\sigma_{x^2x^2x^2}}{\sigma_{x^2}} + 2\lambda\sigma 
\end{equation*}
and, recalling that $\rho=\rho(x^1)$ and $\sigma=\sigma(x^2)$, the above equation gives
\begin{equation*}
\frac{\rho_{x^1x^1x^1}}{\rho_{x^1}} + 2\lambda\rho =  -\varepsilon \frac{\sigma_{x^2x^2x^2}}{\sigma_{x^2}} + 2\lambda\sigma =k\in\R\,
\end{equation*}
i.e.
\begin{equation}\label{eq:gianni.real.app}
\left\{
\begin{array}{lll}
\rho_{x^1x^1}+\lambda\rho^2-k\rho-h_1=0 & &
\\
& & h_1,h_2\in\R\,.
\\
\sigma_{x^2x^2}-\varepsilon(\lambda\sigma^2-k\sigma-h_2)=0
\end{array}
\right.
\end{equation}
If we substitute \eqref{eq:gianni.real.app} into the equation $\E_{11}=0$ we get $h_1=h_2=h$, and, taking this into account, by substituting \eqref{eq:gianni.real.app} in the remaining entries of $\E$, we get zero. Finally, we note that system \eqref{eq:gianni.real.app} leads to \eqref{eq:system.Einstein.real.Liouville}. 
\end{proof}

\begin{prop}\label{prop.Liouville.real.comp.bla}
Let $g$ be a metric of real Liouville type from Theorem \ref{thm:main}. If $g$ is  Einstein, then the metric $\hat{g}$ given by \eqref{eq: companion.A} is Einstein if and only if $\varepsilon c_1+ c_2=0$, where the $c_i$'s are defined within Proposition \ref{prop.Einstein.real.Liouville}, and the Einstein constant $\hat{\lambda}$ of $\hat{g}$ is equal to $\frac12 c_1$.
\end{prop}
\begin{proof}
On account of \eqref{eq:system.Einstein.real.Liouville}, the equation $(\sigma-\rho)(\varepsilon\rho\hat{\E}_{11} - \sigma\hat{\E}_{22})=0$ is (recalling that, in the case we are studying, $\rho\neq \sigma$)
$$
\varepsilon c_1+c_2=0\,.
$$
Substituting this relation into the tensor $\hat{\E}$, we realize  (taking always into account conditions \eqref{eq:system.Einstein.real.Liouville}) that we get zero if and only if $\hat{\lambda}=\frac12 c_1$.
\end{proof}
As we said in the beginning of this section, once we have two non-proportional para-K\"ahler-Einstein metrics belonging to the same pc-projective class, in view of Theorem \ref{thm:main.einstein}, we can construct a $2$-parametric family of para-K\"ahler-Einstein metrics. We will illustrate this procedure by means of a concrete example, that we write below.
\begin{ex}
Let $\varepsilon=1$, $\rho=-\frac{6}{\lambda (x^1)^2}$ and $\sigma=\frac{6}{\lambda(x^2)^2}$. Metric $g$ is an Einstein metric with Einstein constant equal to $\lambda$ as system \eqref{eq:system.Einstein.real.Liouville} is satisfied with $h=k=c_1=c_2=0$. In view of Proposition \ref{prop.Liouville.real.comp.bla}, metric $\hat{g}$ given by \eqref{eq: companion.A} is Ricci-flat. More precisely,
$$
\hat{g}=\frac{\lambda^2\big((x^1)^2+(x^2)^2\big)}{36}\left((x^1)^2dx^1dx^1 + (x^2)^2dx^2dx^2\right) + \frac19\lambda^2\left((x^1)^2-(x^2)^2\right)dx^3dx^3 - \frac43\lambda dx^3dx^4
$$
and formula \eqref{eq:lin.comb.metrics}, with $n=2$, gives the following $2$-parametric para-K\"ahler-Einstein metrics:
\begin{multline}\label{eq:super.appoggio}
-\frac{6\lambda^2}{B}\left((x^1)^2+(x^2)^2\right) \left( \frac{(x^1)^2}{\alpha\lambda (x^1)^2-6\beta}dx^1dx^1 + \frac{(x^2)^2}{\alpha\lambda (x^2)^2+6\beta}dx^2dx^2\right) \\
+\frac{24\lambda^2}{B^2}\left(\alpha\lambda(x^1)^4-\alpha\lambda (x^1)^2 (x^2)^2+\alpha\lambda (x^2)^4-6\beta(x^1)^2+6\beta(x^2)^2\right)dx^3dx^3
\\
-\frac{288\lambda}{B^2}\left(\alpha\lambda(x^1)^2-\alpha\lambda(x^2)^2-6\beta\right)dx^3dx^4 + \frac{864\lambda}{B^2}\alpha dx^4dx^4\,,
\end{multline}
where $B=(\alpha\lambda (x^1)^2-6\beta)(\alpha\lambda (x^2)^2+6\beta)$. The Einstein constant of \eqref{eq:super.appoggio} is $\lambda \alpha^3$.
\end{ex}

\subsubsection{The case of complex Liouville metrics}
In this section we always assume that the Cauchy-Riemann equations (and, of course, their compatibility conditions) coming from $\rho_{\bar z}=0$ are satisfied.
\begin{prop}\label{prop.Einstein.complex.Liouville}
A metric $g$ of the complex Liouville type from Theorem \ref{thm:main} is an Einstein metric with Einstein constant equal to $\lambda$ if and only if the following complex equation 
\begin{equation}\label{eq:system.Einstein.complex.Liouville}
\rho_z^2-\frac16\lambda\rho^3-a\rho^2-2h\rho+d=0
\end{equation}
is satisfied for some constants $a\in\R$ and $h,d\in\C$.
\end{prop}
\begin{proof}
The equation $\E_{12}=0$ can be expressed as follows:
\begin{equation*}
\frac{\rho_{zzz}}{\rho_z}-\frac12\lambda\rho = \frac{\bar{\rho}_{\bar{z}\bar{z}\bar{z}}}{\bar{\rho}_{\bar{z}}}-\frac12\lambda\bar\rho\,.
\end{equation*}
Being $\frac{\rho_{zzz}}{\rho_z}-\frac12\lambda\rho$ a holomorphic function, the above equality implies that it is a real constant, that we denote by $a$, so that  
$\rho_{zzz}-\frac12\lambda\rho\rho_z-a\rho_z=0$,
and
\begin{equation}\label{eq.fin.comp.liou}
\rho_{zz}-\frac14\lambda\rho^2-a\rho=h
\end{equation}
with $h\in\C$, leading to \eqref{eq:system.Einstein.complex.Liouville}. Finally, if we substitute \eqref{eq.fin.comp.liou} in  the tensor field $\E$, we obtain zero.
\end{proof}
\begin{rem}
The (third order) system coming from \eqref{eq:system.Einstein.complex.Liouville} is equivalent to
\begin{equation}\label{eq:sys.app}
\E_{11}=\E_{12}=0\,.
\end{equation}
Indeed, recalling that $\rho(x^1+ix^2)=R(x^1,x^2)+i I(x^1,x^2)$ satisfies the Cauchy-Riemann equation, a direct computation shows that system \eqref{eq:sys.app} vanishes if and only if $\E=0$. This system can be written as
\begin{equation}\label{eq:filippo.non.compact}
\begin{cases}
I\, I_{x^2}I_{x^1x^1x^2}+I\, I_{x^1}I_{x^1x^1x^1}-(I_{x^2})^2I_{x^1x^1}-(I_{x^1})^2I_{x^1x^1}=0&\\
\lambda I(I_{x^1})^2+\lambda I(I_{x^2})^2-2 I_{x^2}I_{x^1x^1x^1}+2I_{x^1}I_{x^1x^1x^2}=0\,,&\\
\end{cases}
\end{equation}
that has the advantage to be expressed only in terms of the real function $I$. 
Considering that $I$ is a non-zero harmonic function, system \eqref{eq:filippo.non.compact} can be put in the following more elegant form:
\begin{equation*}
\begin{cases}
\displaystyle
\left(\frac{(I_{x^1}^2)_{x^1}}{I}\right)_{x^1}
-
\left(\frac{(I_{x^2}^2)_{x^2}}{I}\right)_{x^2}
= 0
\\[4mm]
\displaystyle
\left(\frac{(I_{x^1}^2)_{x^1}}{I}\right)_{x^2}
+
\left(\frac{(I_{x^2}^2)_{x^2}}{I}\right)_{x^1}
= -\lambda\,(I_{x^1}^2 + I_{x^2}^2)
\end{cases}
\end{equation*}
\end{rem}
\begin{rem}
We note that we can obtain some relations between derivatives of lower order by differentiating equation \eqref{eq:system.Einstein.complex.Liouville} w.r.t. $z$ and equating to zero the real and imaginary parts. More precisely, we have that
\begin{equation}\label{eq:first.derivatives.complex}
\left\{
\begin{array}{l}
I_{x^1}I_{x^2} = h_1I + \frac14 \lambda R^2 I + aRI-\frac{1}{12}\lambda I^3+h_2R-\frac12 d_2
\\[0.2cm]
I_{x^2}^2-I_{x^1}^2 = -\frac12 \lambda RI^2 + \frac16\lambda R^3-aI^2+aR^2-2h_2I+2h_1 R -d_1
\end{array}
\right.
\end{equation}
and
\begin{equation}\label{eq:second.derivatives.complex}
\left\{
\begin{array}{l}
I_{x^1x^1}= \frac12\lambda RI +aI+h_2
\\[0.2cm]
I_{x^1x^2}= \frac14\lambda R^2-\frac14\lambda I^2+aR+h_1
\end{array}
\right.
\end{equation}
where $h_1=\frac12(h+\bar{h})$,  $h_2=\frac{1}{2i}(h-\bar{h})$, $d_1=\frac12(d+\bar{d})$ and $d_2=\frac{1}{2i}(d-\bar{d})$. System \eqref{eq:second.derivatives.complex} is nothing but \eqref{eq.fin.comp.liou}.
\end{rem}

\begin{prop}
Let $g$ be a metric of complex Liouville type from Theorem \ref{thm:main}. If $g$ is Einstein, then the metric $\hat{g}$ given by \eqref{eq: companion.A} is Einstein if and only if $h$ and $d$ given by \eqref{eq:system.Einstein.complex.Liouville} are real constants. In this case, the Einstein constant $\hat{\lambda}$ of $\hat{g}$ is equal to $6d$.
\end{prop}
\begin{proof}
In a sense, the proof follows that of the real Liouville case, but with some cares. 
A direct computation shows that $\hat{\E}_{11}+\hat{\E}_{22}=0$ and, on account of this, similarly to the real case, we consider the equation
$$
(\rho-\bar{\rho})(\rho\,\hat{\E}(\partial_z,\partial_z) + \bar{\rho}\,\hat{\E}(\partial_{\bar z},\partial_{\bar z}))=2iI(R\hat{\E}_{11}+I\hat{\E}_{12})=0\,.
$$
Since we are assuming that $g$ is an Einstein metric with Einstein constant equal to $\lambda$, system \eqref{eq:filippo.non.compact} is satisfied. By  substituting it in the above equation, we obtain
$$
\lambda I^3+6I\,I_{x^1x^2}-6I_{x^1}I_{x^2}=0
$$
that, in view of \eqref{eq:first.derivatives.complex} and \eqref{eq:second.derivatives.complex}, gives
$$
2h_2R-d_2=0\,.
$$
Since $R$ cannot be a constant, the above relation yields $h_2=d_2=0$, i.e., the constants $h$ and $d$ of \eqref{eq:system.Einstein.complex.Liouville} are real.

\smallskip\noindent
Taking into account this, a straightforward computation shows that conditions \eqref{eq:filippo.non.compact}, \eqref{eq:first.derivatives.complex} and \eqref{eq:second.derivatives.complex} imply $\hat{\E}=0$ if and only if $\hat{\lambda}=6d$.
\end{proof}
%

\subsection{The degenerate case}

\subsubsection{The case when $\dim D=2$}

\begin{prop}
A metric $g$ of the first type of the case $\dim D=2$ of Theorem \ref{thm:main} is an Einstein metric with Einstein constant equal to $\lambda$ if and only if the following system
\begin{equation}\label{eq:Apostolov.Einstein.1}
\left\{
\begin{array}{l}
\mu=\frac{3(\rho-c)\rho_{x^2}}{2\lambda(\rho-c)^3+c_2(\rho-c)^2+c_1}
\\[0.3cm]
\nu_{x^3} = -\frac16 c_2\nu^2 + f\nu + h
\end{array}
\right.
\end{equation}
is satisfied for some constants $c_1,c_2$ and functions $f=f(x^3)$, $h=h(x^3)$.
\end{prop}
\begin{proof}
The only non-zero entries of $\E$ are $\E_{11}$, $\E_{13}$, $\E_{22}$, $\E_{33}$ and $\E_{34}$. Furthermore, we have the following relations:
$$
\nu\E_{11}+\E_{13}=\mu^2\E_{11}+\E_{22}=\nu^2\E_{11}-\E_{33}=0\,,
$$
so that the system to study reduces to $\E_{11}=0\,, \E_{34}=0$.
The equation $\E_{11}=0$ is an ODE whose solution is given by the first equation of \eqref{eq:Apostolov.Einstein.1}. By substituting it in $\E_{34}=0$ we obtain
\begin{equation}\label{eq:Apos.app}
\left(c_2\nu +3\frac{\nu_{x^3x^4}}{\nu_{x^4}}\right)_{x^4}=0
\end{equation}
that leads to the second equation of \eqref{eq:Apostolov.Einstein.1}.
\end{proof}
\begin{prop}
Let $g$ be a metric of the first type of the case $\dim D=2$ of Theorem \ref{thm:main}. Let $c_1$ and $c_2$ be given by \eqref{eq:Apostolov.Einstein.1}. If $g$ is an Einstein metric, then the metric $\hat{g}$ given by \eqref{eq: companion.A} is Einstein if and only if $c_1=0$.
In this case, the Einstein constant $\hat{\lambda}$ of $\hat{g}$ is equal to $c^2\left(\lambda c-\frac12 c_2\right)$.
\end{prop}
\begin{proof}
The only non-zero entries of $\hat\E$ are $\hat\E_{11}$, $\hat\E_{13}$, $\hat\E_{22}$, $\hat\E_{33}$ and $\hat\E_{34}$. Furthermore, we have the following relations:
$$
\nu\hat\E_{11}+\hat\E_{13}=\mu^2\hat\E_{11}+\hat\E_{22}=\nu^2\hat\E_{11}-\hat\E_{33}=0\,,
$$
so that the system to study reduces to $\hat\E_{11}=0\,, \hat\E_{34}=0$.

\smallskip\noindent
Now, since, by hypothesis, $g$ is an Einstein metric with Einstein constant equal to $\lambda$, system \eqref{eq:Apostolov.Einstein.1} is satisfied. In particular, on account of the first equation of such system, recalling that $\rho$ cannot be a constant and that $c\neq 0$, equation $\hat\E_{11}=0$ is
$$
(2\lambda r^3+c_2 r^2 + c_1)
\big((-2c^3\lambda  + c_2 c^2  + 2\hat{\lambda})r^2 -2c_1 c\, r -c_1c^2\big)=0\,, \quad r:=\rho-c\,.
$$
The quantity $2\lambda r^3+c_2 r^2 + c_1$ cannot vanish as $\mu$ is given by the first equation of \eqref{eq:Apostolov.Einstein.1}. Then, $(-2c^3\lambda  + c_2 c^2  + 2\hat{\lambda})r^2 -2c_1 c\, r -c_1c^2=0$, which implies 
\begin{equation*}
c_1=0\,, \quad \hat\lambda=c^2\left(\lambda c-\frac12 c_2\right)\,.
\end{equation*}
By substituting the above equations, together with the first equation of \eqref{eq:Apostolov.Einstein.1}, into $\hat{\E}_{34}=0$, we obtain \eqref{eq:Apos.app}, that is satisfied as we are supposing that $g$ is an Einstein metric.
\end{proof}
The other metrics of this case, i.e., the second, third and fourth of the case $\dim D=2$ of Theorem \ref{thm:main}, are flat. 
A straightfoward computation shows that, actually, all the metrics \eqref{eq:lin.comb.metrics} with $\hat{g}$ given by \eqref{eq: companion.A}  are flat.

\subsubsection{The case when $\dim D=1$}

\begin{prop}
Let $g$ be a metric of the case $\dim D=1$ of Theorem \ref{thm:main}. If $g$ is an Einstein metric, then it is flat. In this case, all the metrics \eqref{eq:lin.comb.metrics}, with $\hat{g}$ given by \eqref{eq: companion.A}, are flat.
\end{prop}
\begin{proof}
The equation $\E_{13}=0$ gives $\lambda\rho_{x^3}=0$ and, since by hypothesis $\rho_{x^3}\neq 0$, we get $\lambda=0$, i.e., in this case, if the metric $g$ is Einstein, necessarily it must be Ricci-flat. By substituting $\lambda=0$ into the equation $\E_{24}=0$, recalling that $F=F(x^2,\varphi(x^3,x^4))$ is an arbitrary function such that $F_{x^4}\neq 0$, we obtain
$$
\frac{F_{\varphi\varphi}F_{x^2\varphi}-F_{\varphi}F_{x^2\varphi\varphi}}{F_{\varphi}^2}=-\left(\frac{F_{\varphi\varphi}}{F_{\varphi}}\right)_{x^2}=0
$$
that gives
$$
F(x^2,\varphi(x^3,x^4))=\alpha(\varphi(x^3,x^4))\beta(x^2)+\gamma(x^2)
$$
with $\alpha,\beta,\gamma$ arbitrary functions. A direct computation shows that, if $\lambda=0$ and $F$ is given by the above formula, the curvature tensor of $g$ is zero.

\smallskip\noindent
With similar computations one can show that, under the hypothesis that $g$ is Einstein, i.e., in this particular case, flat, also the metrics \eqref{eq:lin.comb.metrics}, with $\hat{g}$ given by \eqref{eq: companion.A}, are flat.
\end{proof}

\section*{Acknowledgements}
The first author gratefully acknowledges support by the project ``Finanziamento alla Ricerca'' under the contract numbers 53\_RBA21MANGIO, and by the PRIN project 2022 “Real and Complex Manifolds: Geometry and Holomorphic Dynamics” (code 2022AP8HZ9). 
The second author is supported by the ``Starting Grant'' under the contact number 53\_RSG22SALFIL. Both authors are members of the GNSAGA of the INdAM.

\subsection*{Declarations}

\subsubsection*{Associated data}

The authors declare that no associated data is included.

\subsubsection*{Conflict of interest }
The authors declare that they have no conflict of interest

\vspace{1.5cm}
\noindent
\textsc{(G. Manno) Dipartimento di Scienze Matematiche ``G. L. Lagrange'', Politecnico di Torino, Corso Duca degli Abruzzi 24, 10129 Torino.}\\
\textit{Email address:} \texttt{giovanni.manno@polito.it}\\

\noindent
\textsc{(F. Salis) Dipartimento di Scienze Matematiche ``G. L. Lagrange'', Politecnico di Torino, Corso Duca degli Abruzzi 24, 10129 Torino.}\\
\textit{Email address:} \texttt{filippo.salis@polito.it}
\end{document}